\documentclass[11pt]{article}
\usepackage{latexsym}
\usepackage{physics}
\usepackage{hyperref}
\usepackage{tikz}
\numberwithin{equation}{section}
\hypersetup{hypertex=true,
	colorlinks=true,
	linkcolor=blue,
	anchorcolor=blue,
	citecolor=blue}
\usepackage{amssymb}
\usepackage{amsmath}
\usepackage{color}
\usepackage{amsthm}

\usepackage{graphicx}
\usepackage{cleveref}
\crefformat{section}{\S#2#1#3} 
\crefformat{subsection}{\S#2#1#3}
\crefformat{subsubsection}{\S#2#1#3}
\usepackage{enumerate}
\usepackage{tikz}
\theoremstyle{plain}
\newtheorem{thm}{Theorem}[section]
\newtheorem{lem}{Lemma}[section]
\newtheorem{cor}[thm]{Corollary}
\newtheorem{prop}[lem]{Proposition}

\theoremstyle{definition}
\newtheorem{rem}[lem]{Remark}
\newtheorem{conv}[lem]{Convention}

\newtheorem{defn}[lem]{Definition}
\def\Q{\tilde{Q}}
\def\cE{{\mathcal{E}}}
\def\Re{\mathrm{Re}}
\def\tP{{\tilde{\P}}}
\def\la{\mathrm{la}}
\def\mL{\mathrm{L}}
\def\mS{\mathrm{S}}
\def\sm{\mathrm{sm}}
\def\lan{\langle}
\def\ran{\rangle}
\def\tO{{\tilde{\Omega}}}
\def\ts{\tilde{s}}

\def\tS{\tilde{\So}}

\def\tD{{\tilde{\Delta}}}
\def\tl{{\tilde{\lambda}}}
\def\RS{\mathrm{RS}}
\def\rank{\mathrm{rank}}

\def\im{\mathrm{im}}
\def\rel{\mathrm{rel}}
\def\bd{\mathrm{bd}}
\def\lan{\langle}
\def\ran{\rangle}

\def\abs{\mathrm{abs}}
\def\Tr{\operatorname{Tr}}
\def\tr{\operatorname{tr}}

\def\tr{\operatorname{tr}}
\def\bm{\bar{M}}

\def\Z{\mathbb{Z}}
\def\supp{\mathrm{supp}}

\def\P{\mathcal{P}}
\def\dvol{\mathrm{dvol}}

\def\l{\lambda}
\def\a{\alpha}

\def\b{\beta}

\def\R{\mathbb{R}}
\def\F{{\bar{F}}}
\def\hc{\hat{c}}

\def\C{\mathbb{C}}
\def\sqt{\sqrt{\frac{2}{T}}}
\def\Im{\mathrm{im}}

\def\Dom{\mathrm{Dom}}

\def\half{\frac{1}{2}}

\def\e{\mathcal{E}}

\def\ep{\epsilon}

\def\So{\operatorname{S}}
\def\H{\mathcal{H}}
\def\A{\mathcal{A}}

\def\o{\omega}
	
	\def\p{\partial}
	\def\a{\alpha}
	\def\half{\frac{1}{2}}
	\def\l{\lambda}
	
\def\T{\mathcal{T}}

\begin{document}

	\newcommand{\be}{\begin{equation}}
		\newcommand{\ee}{\end{equation}}
	\newcommand{\bc}{\begin{cases}}
		\newcommand{\ec}{\end{cases}}
	\newcommand{\bes}{\begin{equation*}}
		\newcommand{\ees}{\end{equation*}}
	\newcommand{\ba}{\begin{aligned}}
		\newcommand{\ea}{\end{aligned}}
	\newcommand{\bas}{\begin{align*}}
		\newcommand{\eas}{\end{align*}}
	\newcommand{\es}{\end{split}}
\newcommand{\bs}{\begin{split}}

	\title{Witten deformation for non-Morse functions and\\ the gluing formula for analytic torsions}
	
	\date{}                                      
	
	\author{Junrong Yan
		\footnote{Beijing International Center for Mathematical Research, Peking University, Beijing, China 100871
		}
  \footnote{Department of mathematics, Northeastern University, Boston, MA, USA 02215, j.yan@northeastern.edu.}
}
	
	\maketitle
	
	\abstract{This paper concentrates on analyzing Witten deformation for a family of non-Morse functions parameterized by $T\in \R_+$, resulting in a novel, purely analytic proof of the gluing formula for analytic torsions in complete generality due to Br\"unning-Ma \cite{bruning2013gluing}. Intriguingly, the gluing formula in this article could be reformulated as the Bismut-Zhang theorem \cite{bismutzhang1992cm} for non-Morse functions, and from the perspective of Vishik's theory of moving boundary problems \cite{vishik1995generalized,lesch2013gluing}, the deformation parameter $T$ parameterize a family of boundary conditions. Our proof also makes use of a connection between small eigenvalues of Witten Laplacians and Mayer-Vietoris sequences. Finally, these new techniques could be extended to analytic torsion forms and play key roles in the study of the higher Cheeger-Müller/Bismut-Zhang theorem for nontrivial flat bundles \cite{MY}.
	}
	\tableofcontents
	\section{Introduction}

	\subsection{Overview}
In \cite{ray1971r}, Ray and Singer introduced analytic torsion (RS-torsion) for unitarily flat vector bundles \( F \) over a closed Riemannian manifold \( M \), conjecturing that it coincides with Reidemeister-Franz torsion (RF-torsion)—a topological invariant introduced by Reidemeister \cite{reidemeister1935homotopieringe}, Franz \cite{Franz}, and de Rham \cite{derh}, which distinguishes homotopy-equivalent but non-homeomorphic CW complexes and manifolds. This conjecture was independently proven by Cheeger \cite{cheeger1979analytic} and M\"uller \cite{muller1978analytic}. Simultaneously, M\"uller then extended the result to unimodular flat bundles \cite{muller1993analytic}, and Bismut and Zhang to arbitrary flat bundles \cite{bismutzhang1992cm}. This conjecture is now known as the Cheeger-M\"uller/Bismut-Zhang Theorem. There are also various extensions to the equivariant cases \cite{bismut1994milnor,lott1991analytic,luck1993analytic}.

 Let \( Y \subset M \) be a hypersurface separating \( M \) into two pieces, \( M_1 \) and \( M_2 \). Ray and Singer \cite{ray1971r}, as a step toward their conjecture, suggested that a gluing formula should relate the analytic torsion of \( M \) to those of \( M_1 \) and \( M_2 \). However, the conjecture was eventually proven by other methods, and the gluing formula was derived from some version of it. For example, L\"uck \cite{luck1993analytic} established gluing formula for unitarily flat vector bundles using the Cheeger-M\"uller theorem and \cite{lott1991analytic}. Br\"unning and Ma \cite{bruning2013gluing}, building on Bismut and Zhang’s \cite{bismut1994milnor} equivariant version of the extended Cheeger-M\"uller theorem, extended the result to general flat bundles. There are also purely analytic proofs of the gluing formula, that is, without relying on any version of the Cheeger-Müller/Bismut-Zhang theorem. Vishik \cite{vishik1995generalized} proved the formula for unitarily flat bundles using moving boundary conditions, while Puchol, Zhang, and Zhu proved the formula for general flat bundles \cite{10.2140/apde.2021.14.77} using adiabatic limits. Related work includes contributions by Hassell \cite{Hassel1998} and Lesch \cite{lesch2013gluing}.

The proof of gluing formula in this paper is based on the analysis of Witten deformation for non-Morse functions and some coupling techniques (see the last paragraph in \cref{idea} for a brief description of coupling techniques). Roughly, we choose a family of smooth functions $f_T$ such that $f_T$ has critical loci approaching $M_1$ and $M_2$ with ``Morse indices" of $0$ and $1$, respectively. Then, based on the philosophy of Witten deformation, letting $T$ range from $0$ to $\infty$, the relationship between analytic torsion on $M$ and analytic torsion on the two pieces above can be understood. Thus, philosophically, the gluing formula can be interpreted as a Bismut-Zhang theorem (extended Cheeger-M\"uller theorem) \cite{bismutzhang1992cm} for non-Morse functions.
\begin{rem}A similar non-Morse function was introduced by Puchol, Zhang, and Zhu in \cite{puchol2020adiabatic}. However, in their work, such functions are used to ensure a uniform spectral gap for their two-parameter Witten Laplacian, rather than to separate the space in the spectral-theoretic sense as in our case. Specifically, we show that the \(k\)-th eigenvalue of the Witten Laplacian associated with non-Morse functions converges to the \(k\)-th eigenvalues of the Laplacians (with suitable boundary conditions) on \(M_1\) and \(M_2\) (see \eqref{ob1} or \Cref{eigencon}).

\end{rem}
	
The approach in this paper can also be applied to analytic torsion forms \cite{Yanforms}, offering an alternative proof of the gluing formula for torsion forms by Puchol-Zhang-Zhu \cite{puchol2020adiabatic}. Moreover, the techniques introduced here and in \cite{Yanforms} play a crucial role in the author's joint project with Martin Puchol on the higher Cheeger-Müller/Bismut-Zhang theorem \cite{MY}. This approach may also be useful for proving gluing formulas for other global spectral invariants, such as eta-invariants, eta-forms, and partition functions from quantum field theory.
	\ \\ \ \\

	\subsection{The gluing formulas}
	Let $(M,g^{TM})$ be a closed Riemannian manifold and $Y\subset M$ be a hypersurface that divides $M$ into two pieces $M_1$ and $M_2$.
	
	Let $F\to M$ be a flat complex vector bundle with a flat connection $\nabla^F$. Let $d^F$ be the covariant differential on $F$-valued differentials $\Omega(M;F)$, which is induced by $\nabla^F$. 
	
	Let $h^F$ be a Hermitian metric on $F$. Let $d^{F,*}$ be the (formal) adjoint operator of $d^F$ associated with $g^{T M}$ and $h^F$. Then $D:=d^F+d^{F,*}$ is a first-order self-adjoint elliptic operator acting on $\Omega (M;F)$. 
	
	Let $N$ be the number operator on $\Omega \left(M;F\right)$, that is , $N \omega=k \omega$ for $\omega \in \Omega^k\left(M; F\right)$. The zeta function $\zeta$ for $\Delta:=D^2$ is defined as follows, for $z \in\{\mathbb{C}$ : $\left.\operatorname{Re}(z)>\frac{1}{2} \operatorname{dim} M\right\}$
	$$
	\zeta(z):=\frac{1}{\Gamma
		( z)}\int_0^\infty t^{z-1}\Tr_s\qty(Ne^{-t\Delta'})dt.
	$$
Here, \(\Delta'\) denotes the restriction of \(\Delta\) to the orthogonal complement of the space of harmonic forms. 
	The function $\zeta(z)$ admits a meromorphic continuation to the whole complex plane, which is holomorphic at $0 \in \mathbb{C}$. See \cref{zetadef} for more details.
	
	Then the Ray-Singer analytic torsion is given by $\mathcal{T}(g^{TM},h^F,\nabla^F):=e^{\frac{1}{2}\zeta'(0)}.$
	
	Let $F_i$ be the restriction of $F$ on $M_i$, $d_i^F$ be the restriction of $d^F$ on $M_i$, $d_i^{F,*}$ be the formal adjoint of $d_i^F$ (i=1,2).
	Let $\Delta_1:=(d_1^F+d_1^{F,*})^2$ and $\Delta_2=(d_2^F+d_2^{F,*})^2$ act on $\Omega_{\text {abs }} \left(M_1;F_1\right)$ and $\Omega_{\text {rel }} \left(M_{2};F_2\right)$ respectively (see \cref{oabs} for the definition of $\Omega_{\text {abs }} \left(M_1;F_1\right)$ and $\Omega_{\text {rel} } \left(M_2;F_2\right)$). Let $\zeta_{i}$ be the zeta functions for $\Delta_i$. Then similarly, one can define Ray-Singer analytic torsion $\T_i(g^{TM_i},h^{F_i},\nabla^{F_i})=e^{\frac{1}{2}\zeta_i'(0)}$, where $g^{TM_i}$, $F_i, h^{F_i}$ and $\nabla^{F_i}$ are the restriction of $g^{TM},F,h^F$ and $\nabla^{F}$ on $M_i$ respectively($i=1,2$).
	
	Lastly, we have the following Mayer-Vietoris exact sequence 
	\be\label{mv000}
	\cdots \rightarrow H^k_{\rel}\left(M_2;F_2\right) \rightarrow H^k\left(M; F\right) \rightarrow H^k_{\abs}\left(M_1;F_1\right) \rightarrow \cdots .
	\ee
	We denote by $\mathcal{T}$ the analytic torsion for the exact sequence above equipped with $L^2$-metrics (induced by Hodge theory).
	
The fully general gluing formula given below was originally established in \cite{bruning2013gluing} using the (equivariant) Cheeger-Müller/Bismut-Zhang theorem. 
	\begin{thm}[Br\"uning-Ma \cite{bruning2013gluing}]\label{main0}
		\begin{align*}
			&\ \ \ \ \log\mathcal{T}(g^{TM},h^F,\nabla^F)-\log\T_1(g^{TM_1},h^{F_1},\nabla^{F_1})-\log\T_2(g^{TM_2},h^{F_2},\nabla^{F_2})-\log \mathcal{T}\\
			&=\frac{1}{2} \chi(Y) \rank(F) \log 2+(-1)^{\dim(M)} \mathrm{rank}(F) \int_{Y} B\left(g^{TM}\right),\end{align*}
		where $B\left(g^{TM}\right)$ is the secondary characteristic form introduced by Br\"uning-Ma in \cite{bruning2006anomaly}, which is zero if $Y$ is totally geodesic in $\left(M, g^{T M}\right)$.
	\end{thm}

\begin{rem}
    The anomaly formulas of Bismut-Zhang \cite[Theorem 4.7]{bismutzhang1992cm} and Br\"uning-Ma \cite[Theorem 0.1]{bruning2006anomaly}, \cite[Theorem 3.4]{bruning2013gluing} play a significant role in simplifying the problem. Using these formulas, it suffices to focus on the case where \( g^{TM} \) and \( h^F \) are product-type near \( Y \) (similar to what Br\"unning-Ma did in \cite{bruning2013gluing}). In the product-type case, the term \( \int_{Y} B(g^{TM}) \) vanishes. Without loss of generality, from now on, we assume that the metrics are of product-type near \( Y \).
\end{rem}

Although many proofs of gluing formulas for analytic torsions already exist, as noted in the overview, the author presents
this proof because the techniques are new and could be applied to study gluing formulas
for other global spectral invariants. Furthermore, the
methods developed here and in \cite{Yanforms} are key in the author’s joint work with Martin
Puchol on the higher Cheeger-M\"uller/Bismut-Zhang theorem for unitarily flat bundles \cite{MY}.

	\subsection{Main ideas}\label{idea}
	
	Let $Y\subset M$ be a hypersurface cutting $M$ into two pieces $M_1$ and $M_2$ (see Figure \ref{fig1}). We then glue $M_1$, $M_2$ and $[-1,1]\times Y$ naturally; we get a manifold $\bm$ (see Figure \ref{fig2}), which is diffeomorphic to the original manifold $M.$

    Let \( f_T \) be a smooth jump function on \( \bm \) (see Figure \ref{fig3}), satisfying \( f_T \equiv -T/2 \) on \( M_1 \) and \( f_T \equiv T/2 \) on \( M_2 \) (see \cref{aw} for additional requirements on \( f_T \)). 
	\begin{figure}[h]
		\setlength{\unitlength}{0.75cm}
		\centering
		\begin{picture}(18,6.5)
			\qbezier(3,2)(5,1)(7,2) 
			\qbezier(3,4)(5,5)(7,4) 
			\qbezier(3,2)(1.5,3)(3,4) 
			\qbezier(4,3.2)(5,2.7)(6,3.2) 
			\qbezier(4.5,3)(5,3.2)(5.5,3) 
			\qbezier(7,2)(7.5,2.3)(8,2.4) 
			\qbezier(7,4)(7.5,3.7)(8,3.6) 
			\qbezier(11,2)(13,1)(15,2) 
			\qbezier(11,4)(13,5)(15,4) 
			\qbezier(15,2)(16.5,3)(15,4) 
			\qbezier(12,3.2)(13,2.7)(14,3.2) 
			\qbezier(12.5,3)(13,3.2)(13.5,3) 
			\qbezier(11,2)(10.5,2.3)(10,2.4) 
			\qbezier(11,4)(10.5,3.7)(10,3.6) 
			\qbezier(8,2.4)(9,2.4)(10,2.4) 
			\qbezier(8,3.6)(9,3.6)(10,3.6) 
			
			\qbezier(9,2.4)(8.7,3)(9,3.6) 
			\qbezier[30](9,2.4)(9.3,3)(9,3.6) 
			
			\put(9,5.3){\vector(0,-1){1.5}} 
			
			\put(3,4.8){$\overbrace{\hspace{40mm}}^{}$} \put(5.1,5.4){$M_{1}$}  
			\put(9.7,4.8){$\overbrace{\hspace{40mm}}^{}$} \put(11.9,5.4){$M_{2}$}  
			\put(3.5,1.2){$\underbrace{\hspace{83mm}}_{}$} \put(8.8,0.3){$M$}  
			\put(9,5.5){$Y$}
		\end{picture}
		\caption{Cutting $M$ along $Y$}
		\label{fig1}
	\end{figure}

	\begin{figure}[h]
		\setlength{\unitlength}{0.75cm}
		\centering
		\begin{picture}(18,6.5)
			\qbezier(1,2)(3,1)(5,2) 
			\qbezier(1,4)(3,5)(5,4) 
			\qbezier(1,2)(-0.5,3)(1,4) 
			\qbezier(2,3.2)(3,2.7)(4,3.2) 
			\qbezier(2.5,3)(3,3.2)(3.5,3) 
			\qbezier(5,2)(5.5,2.3)(6,2.4) 
			\qbezier(5,4)(5.5,3.7)(6,3.6) 
			\qbezier(13,2)(15,1)(17,2) 
			\qbezier(13,4)(15,5)(17,4) 
			\qbezier(17,2)(18.5,3)(17,4) 
			\qbezier(14,3.2)(15,2.7)(16,3.2) 
			\qbezier(14.5,3)(15,3.2)(15.5,3) 
			\qbezier(13,2)(12.5,2.3)(12,2.4) 
			\qbezier(13,4)(12.5,3.7)(12,3.6) 
			\qbezier(6,2.4)(9,2.4)(12,2.4) 
			\qbezier(6,3.6)(9,3.6)(12,3.6) 
			\qbezier(7,2.4)(6.7,3)(7,3.6) 
			\qbezier[30](7,2.4)(7.3,3)(7,3.6) 
			\qbezier(11,2.4)(10.7,3)(11,3.6) 
			\qbezier[30](11,2.4)(11.3,3)(11,3.6) 
			
			
			\put(2.8,2){$M_1$} 
			\put(14.8,2){$M_2$} 
			\put(6.2,2){$\underbrace{\hspace{43mm}}^{}$} \put(8,1){$[-2,2]\times Y$}  
			\put(1,5.0){$\overbrace{\hspace{59mm}}^{}$} \put(4.1,5.5){$\bm_1$}  
			\put(7.2,4){$\overbrace{\hspace{27mm}}^{}$} \put(8,4.5){$[-1,1]\times Y$}  
			\put(9.2,5.0){$\overbrace{\hspace{58mm}}^{}$} \put(12.9,5.5){$\bm_2$}  
			\put(1.5,0.5){$\underbrace{\hspace{113mm}}_{}$} \put(8.8,-0.5){$\bm$}  
		\end{picture}
		\caption{Elongating}
		\label{fig2}
	\end{figure}

    \begin{figure}[h]
    \setlength{\unitlength}{0.3cm}
		\begin{center}
			\begin{tikzpicture}
				\draw[->,thick] (-4,0) -- (4,0);
				\draw[->,thick] (0,-1.5) -- (0,1.5);
				\draw[domain=-4:-1] plot (\x,{-1});
				\draw[domain=-1:0] plot (\x,{(\x+1)*(\x+1)-1});
				\draw[domain=0:1] plot (\x,{-(\x-1)*(\x-1)+1});
				\draw[domain=1:4] plot (\x,{1});
				\draw [dashed] plot coordinates {(-1,-1)(0,-1)};
				\draw [dashed] plot coordinates {(1,1)(0,1)};
				\node at (0.5,-1) {$-T/2$};
				\node at (-0.5,1) {$T/2$};
				\draw [dashed] plot coordinates {(-1,1.4)(-1,-1.4)};
				\draw [dashed] plot coordinates {(1,1.4)(1,-1.4)};
				\node at (-2.5,0.6) {$M_1$};
				\node at (2.5,0.6) {$M_2$};
				\fill [black] (-1,0) circle (2pt);
				\fill [black] (1,0) circle (2pt);
				\node at (-1.3,0.3){$-1$};
				\node at (1.2,0.3){$1$};
			\end{tikzpicture}
		\end{center}
		\caption{Graph of $f_T$}
		\label{fig3}
	\end{figure}

	Let $d_T:=d^{\F}+df_T\wedge$, $d_T^{*}$ be the formal adjoint of $d_T$. Set $D_T:=d_T+d_T^{*}$, $\Delta_T:=D_T^2.$ 
	
	Let $\l_k$ be the $k$-th eigenvalue (counted with multiplicities) of $\Delta_1\oplus\Delta_2$. Let $\{\l_k(T)\}$ be the $k$-th eigenvalue (counted with multiplicities) of $\Delta_T$ (acting on $\Omega(\bm;\F)$).
	
In this paper, we show that (\Cref{eigencon})\be\label{ob1}\lim_{T\to\infty}\l_k(T)=\l_k.\ee
	
	Let $\mathcal{T}(g^{T\bm},h^{\F},\nabla^{\F})(T)$ be the analytic torsion with respect to $\Delta_T$. Then based on (\ref{ob1}), naively, one should expect that 
	\be\label{ideq}\lim_{T\to\infty}\log \mathcal{T}(g^{T\bm},h^{\F},\nabla^{\F})(T)``="\log \mathcal{T}_{1}(g^{TM_1},h^{F_1},\nabla^{F_1})+\log \mathcal{T}_{2}(g^{TM_2},h^{F_2},\nabla^{F_2}),\ee and $$\lim_{T\to 0}\log\mathcal{T}(g^{T\bm},h^{\F},\nabla^{\F})(T)=\log\mathcal{T}(g^{T\bm},h^{\F},\nabla^{\F}).$$ As a result, we can observe the relationship between the analytic torsion on \( M \) and that on the two pieces above, leading to Theorem \ref{main0}.

However, the convergence of eigenvalues alone does not suffice to establish \eqref{idea}. In \Cref{est1} (see also \Cref{rem53}), we establish a uniform polynomial-like lower bound for the eigenvalues of \(\Delta_T\), which, together with the eigenvalue convergence result, ensures the pointwise convergence of the heat supertrace.  To apply the dominated convergence theorem, uniform estimates for both the large-time and small-time heat supertrace are required. The uniform polynomial-like lower bound for the eigenvalues (\Cref{est1}) allows us to control the large-time heat supertrace directly. For the small-time heat supertrace, the problem reduces to analyzing a one-dimensional model using the finite propagation speed technique (see \Cref{conbd}). To study the one-dimensional model, we introduce coupling techniques (see \Cref{idea-coupling} for more details). Next, comparing the coupled and uncoupled heat supertraces leads to certain unknown constants, which we determine by studying Ray-Singer metrics on simple spaces such as \( S^1 \) and finite intervals. Finally, we analyze the heat supertrace associated with small eigenvalues and relate it to the Mayer-Vietoris sequence.

In what follows, we elaborate on some of these key ideas in greater detail.

	\begin{rem}
		Based on the preceding discussion, $T=0$ corresponds to $\theta=\pi/4$ in Vishik's theory of moving boundary problems  (c.f. \cite[\S 4.2]{lesch2013gluing}), and $T=\infty$ corresponds to $\theta=0$. Moreover, from the argument in the proof of Theorem \ref{eigencon}, let $u_k(T)$ be a unit $k$-th eigenform of $\Delta_T$, then there exists a sequence $\{T_l\}_{l=1}^\infty,T_l\to\infty$, such that $u_k(T_l)$ converges to a unit $k$-th eigenform of $\Delta_1\oplus\Delta_2.$
	\end{rem}

	\subsubsection{Coupling techniques}\label{idea-coupling}
	  Let $\zeta_T,\zeta_1$ and $\zeta_2$ be the zeta functions for $\Delta_T,\Delta_1$ and $\Delta_2$ respectively. Let
	\[\zeta_T^{\mL}(z):=\frac{1}{\Gamma(z)}\int_1^\infty t^{z-1}\Tr_s\qty(Ne^{-t\Delta_T'})dt,\]
	and
	\[\zeta_T^{\mS}(z):=\frac{1}{\Gamma(z)}\int_0^1 t^{z-1}\Tr_s\qty(Ne^{-t\Delta_T'})dt\]
    be large and small time components of zeta functions.
    Similarly, we have $\zeta_i^{\mL} $ and $\zeta_i^{\mS}$ ($i = 1, 2$).

The main challenge lies in the small-time component, where the limit \( \lim_{T\to\infty} \Tr_s(Ne^{-t\Delta_T'}) \) for \( t \in (0,1) \) is not uniform.  

Using the finite propagation speed technique (see \Cref{conbd}), it suffices to analyze the one-dimensional model \( \Tr_s(Ne^{-t\Delta^\R_{T}}) \), where \( \Delta^\R_T \) is the (Witten-deformed) Laplacian on the interval \( [-2,2] \) with appropriate boundary conditions. By coupling \( t \) and \( T \), we observe that the quantity \( \Tr_s(Ne^{-t\Delta^{\R}_{t^{-7}T}}) \) exhibits desirable behavior as \( T \to \infty \). Specifically, as \( T \to \infty \) (see  \eqref{lacon1}), we obtain  
\[
t^{-1} \Tr_s(Ne^{-t\Delta^{\R}_{t^{-7}T}}) - t^{-1} \Tr_s(Ne^{-t\Delta^{\R}_{1}}) - t^{-1} \Tr_s(Ne^{-t\Delta^{\R}_{2}}) \to 0 \quad \text{in } L^1(0,1]_t,
\]
where \( \Delta_1^\R \) and \( \Delta_2^\R \) are the usual Hodge Laplacians on the intervals \( [-2,-1] \) and \( [1,2] \), respectively, with suitable boundary conditions.  A similar coupling technique appears in the author's previous joint work with Xianzhe Dai \cite{DY2020index}.

We still need to compare the coupled and uncoupled heat traces. There are unknown terms arising from \( \Tr_s(N^{\R}e^{-t\Delta^{\R}_{t^{-7}T}}) - \Tr_s(N^{\R}e^{-t\Delta^{\R}_{T}}) \), which are independent of \( M \), \( Y \), and \( F \) (see \( \tilde{\zeta}_T'(0) \) and \( \tilde{\zeta}_{T,i}'(0) \) in \Cref{main1} and \Cref{main2}). To compute these terms explicitly, we study the Ray-Singer metrics associated with the simple spaces \( S^1 \) and \( [-2,2] \).

 Finally, for the large-time components: Using the convergence of eigenvalues (\ref{ob1}) and the uniform lower bound on eigenvalues (\Cref{est1}), we apply the dominated convergence theorem to establish that for the large-time components (\Cref{larcon}),  
\[
(\zeta^{\mL}_T)'(0) \to (\zeta_1^{\mL})'(0) + (\zeta_2^{\mL})'(0) \quad \text{as } T \to \infty,
\]
provided 
$
\dim H^k(M;F) = \dim H^k(M_1;F_1) + \dim H^k(M_2, \p M_2;F_2),
$
i.e., no small eigenvalues.	
	\subsubsection{Small eigenvalues and Mayer-Vietoris exact sequence}
	In the general case, \(\dim H^k(M;F) \neq \dim H^k(M_1;F_1) + \dim H^k(M_2, \partial M_2; F_2)\), and we must deal with the small nonzero eigenvalues. Note that the space generated by eigenforms corresponding to small eigenvalues of \(\Delta_T\), denoted by \(\Omega_{\text{sm}}(M,F)(T)\), is finite-dimensional for large values of \(T\). The third crucial component in our proof is the relationship between the analytic torsion for \(\Omega_{\text{sm}}(M,F)(T)\) and the analytic torsion for the Mayer-Vietoris exact sequence (\ref{mv000}) as \(T \to \infty\) (see \Cref{int6p}).
	
	\subsection{Organization}

In \cref{section1}, we provide a brief review of analytic torsion and establish the basic settings. In \cref{intr}, we state several important intermediate results (Theorems \ref{eigencon}–\ref{int6p}), from which Theorem \ref{main0} follows easily.  In \cref{eigen}, we prove the convergence of eigenvalues (Theorem \ref{eigencon}). In \cref{conl}, we establish the convergence of the large-time components of analytic torsion (Theorem \ref{larcon}). In \cref{glus}, we study the heat trace in the small-time regime and reduce the problem to a one-dimensional model. In \cref{cons}, we explore this one-dimensional model and show that when the time \( t \) and the deformation parameter \( T \) are suitably coupled, the behavior of the heat kernel on the tube can be understood. We then partially establish the convergence of the small-time components of analytic torsion (Theorem \ref{main}), leaving some unknown terms, which are computed in \cref{last}. This allows us to fully prove Theorem \ref{main}.  
 Finally, in \cref{lastreal}, we establish a relationship between the analytic torsion for small eigenvalues and the analytic torsion for the Mayer–Vietoris exact sequence (Theorem \ref{int6p}).

	\section{Preliminary}\label{section1}
In this section, we provide a brief review of Hodge theory in \cref{hodgere} and analytic torsion on both closed manifolds and manifolds with boundary in \cref{review-ana}. In \cref{aw}, we introduce a family of non-Morse functions and the Witten-deformed analytic torsion.
	\def\End{\mathrm{End}}
	\subsection{Conventions}
	From now on, we assume that $g^{TM}$ and $h^F$ are product-type near $Y$. That is, there exists a neighborhood $U$ of $Y$, such that $U= (-1,1)\times Y$, and let $(s,y)$ be its coordinate. Then $g^{TM}|_U=ds\otimes ds+g^{TY}$ for some metric $g^{TY}$ on $Y.$ Let $h^F_Y:=h^F|_{\{0\}\times Y}$. For any $v\in F_{(s,y)}$, let $P_{\gamma}\in \End(F_{(s,y)},F_{(0,y)})$ be the parallel transport associated with $\nabla^F$ for the path $\gamma(t)=(st,y),t\in[0,1]$, then we require that $h^F(v,v)=h^F_Y(P_\gamma v,P_\gamma v).$ Hence on $U$, \be\label{ddsh}\nabla^F_{\frac{\p}{\p s}}h^F=0.\ee
	
	All constants appearing in this paper are at least independent of $T$,   if $T$ is sufficiently large. The notations $C$ and $c$, et cetera, denote constants that may vary based on context.

We adopt the following convention throughout the paper:  
\begin{conv}\label{conv1}  
For any differential operator \( D \) acting on \( \Omega(M,F) \), and for \( w \in \Omega(M) \) and \( u \in \Omega(M,F) \), the expression \( Dw \wedge u \) is understood as \( D(w \wedge u) \).  
\end{conv}

	\subsection{A brief review on Hodge theory}\label{hodgere}
	Let $(X,g^{TX})$ be a compact manifold with boundaries $Y=\partial X$, where $Y$ could be an empty set. Let $F\to X$ be a flat vector bundle with flat connection $\nabla^F$, and $h^{F}$ be a Hermitian metric on $F$. We identify a neighborhood of $\partial X$ to $(-1,0] \times Y$, and let $(s,y)$ be its coordinates. Let $\Omega(X;F)$ denote the space of smooth $F$-valued differential forms.

	Set
	\begin{align*}\A_{\mathrm{abs}} (X;F)&:=\left\{\omega \in \Omega (X;F): i_{\frac{\partial}{\partial s}} \omega=0\mbox{ on }Y\right\},\\
		\A_{\mathrm{rel}} (X;F)&:=\left\{\omega \in \Omega (X;F): d s \wedge \omega=0 \text { on } Y\right\};\\
		\Omega_{\mathrm{abs}} (X;F)&:=\left\{\omega \in \Omega (X;F): i_{\frac{\partial}{\partial s}} \omega=0, i_{\frac{\partial}{\partial s}} d^{F} \omega=0\mbox{ on }Y\right\},\\
		\Omega_{\mathrm{rel}} (X;F)&:=\left\{\omega \in \Omega (X;F): d s \wedge \omega=0, d s \wedge d^{F, *} \omega=0 \text { on } Y\right\}.
	\end{align*}
	For the sake of convenience, ``bd" will be adopted to represent ``abs" or ``rel", when it is not necessary to distinguish the boundary conditions.

	Let $d^F:\Omega^*(X;F)\to \Omega^{*+1}(X;F)$ denote the covariant derivative with respect to $\nabla^F$. Let $\nabla$ be the connection on $\Lambda(T^*X)\otimes F$ induced by $\nabla^{F}$ and $g^{TX}$ (more precisely, induced by $\nabla^F$ and the Levi-Civita connection $\nabla^{TX}$ associated to $g^{TX}$). Assuming that $\{e_i\}$ is a local orthonormal frame of $TM$ and $\{e^i\}$ is its dual frame, then locally (c.f. \cite[(4.26)]{bismutzhang1992cm}),\be\label{locdf}d^F=\sum_ke^k\wedge\nabla_{e_i}.\ee  Let $\nabla^*$ be the dual connection of $\nabla$ , that is, for any $s_1,s_2\in\Gamma(\Lambda(T^*M)\otimes F),$ $d\lan s_1,s_2\ran=\lan\nabla s_1,s_2\ran+\lan s_1,\nabla^{*}s_2\ran$, where $\lan\cdot,\cdot\ran$ is the metric on $\Lambda(T^*M)\otimes F$ induced by $g^{TM}$ and $h^F.$ Let $d^{F,*}$  be the formal adjoint of $d^F$, then locally (c.f. \cite[(4.27)]{bismutzhang1992cm}) \be\label{locdfs} d^{F,*}:=-\sum_{i}i_{e_i}\nabla^{*}_{e_i}.\ee
	Let $d^{F}_{\bd}, d^{F,*}_{\bd}$ and $\Delta_{\bd}$ be the restrictions of $d^F,d^{F,*}$ and $\Delta$ to $\A_{\bd}(X; F)$ respectively. 
	
	Let $(\cdot,\cdot)_{L^2}$ be the $L^2$-inner product induced by $h^F$ and $g^{TX}$, that is, for $L^2$-forms $\a$ and $\b$, 
	$$(\a,\b)_{L^2}=\int_X\lan\a,\b\ran\dvol_X.$$
	Then integration by parts implies that
	\be\label{adjointness}(d_{\bd}^F\a,\b)_{L^2}=(\a,d_{\bd}^{F,*}\b)_{L^2}, \forall \a,\b\in \A_{\bd}(X;F).\ee

	
	
	For a closable operator $\So:\H\to\H$ on a Hilbert space $(\H,(\cdot,\cdot)_{\H})$ with a dense domain $\Dom(\So)$, define an inner product $(\cdot,\cdot)_{\So}$ on $\Dom(\So)$ by
	\[(\a,\b)_{\So}:=(\a,\b)_{\H}+(\So\a,\So\b)_{\H},\forall \a,\b\in\Dom(\So).\]
	
	Let $W_{\min}$ be the completion of $\Dom(\So)$ with respect to the norm $\|\cdot\|_{\So}$, then one can extend $\So$ to $\So_{\min}$ naturally with $\Dom(\So_{\min})=W_{\min}$. 
	
	Assume that there is another closable operator $\tS:\H\to\H$, such that $\Dom(\tS)=\Dom(\So)$ and
	\[(\So\a,\b)_{\H}=(\a,\tS\b)_{\H},\forall \a,\b\in\Dom(\So).\] 
	Let $$W_{\max}:=\{\a\in \H:|(\a,\tS\b)_{\H}|\leq M_{\a}\|\b\|_{\H}\mbox{ for some constant $M_\a>0$}, \forall \b\in\Dom(\tS) \}.$$ Because $\Dom(\So)$ is dense, the Riesz representation theorem states that $\gamma\in \H$ exists such that $(
	\gamma, \b)_{\H}=(\a, \tS\b)_{\H}. $ Then we define $\So_{\max}\a=\gamma.$ One can see easily that $\So_{\max}$ is nothing but the adjoint of $\tS.$
	
	By (\ref{adjointness}),  \cite[Proposition A.3]{DY2020cohomology} and the discussion above, one has
	\[\Im(d^F_{\bd,\min})\oplus\Im(d^{F,*}_{\bd,\min})=\Im(d^F_{\bd,\max})\oplus \Im(d^{F,*}_{\bd,\max}).\]

	Together with \cite[Theorem 1.1]{bruning2013gluing}, one has
	
	\begin{thm} [Hodge decomposition]\label{hodec}\ \\
		\begin{enumerate}[(1)]
			\item We have
			$$
			\ker{\Delta_{\bd}}=\ker\left(d^F\right) \cap \ker\left(d^{F, *}\right) \cap \Omega_{\mathrm{bd}}(X;F) .
			$$
			\item The vector space $\ker{\Delta_{\bd}}$ is finite-dimensional.
			\item  We have the following orthogonal decomposition
			$$
			\A_{\mathrm{bd}}(X; F) =\ker(\Delta_{\bd}) \oplus \Im (d_{\bd}^F)  \oplus \Im (d_{\bd}^{F, *}) .$$
			\begin{align*}
				L^2\A_{\bd}(X; F) &=\ker(\Delta_{\bd}) \oplus \Im (d_{\bd,\min}^F) \oplus \Im (d_{\bd,\min}^{F, *})\\
				&=\ker(\Delta_{\bd}) \oplus \Im (d_{\bd,\max}^F) \oplus \Im( d_{\bd,\max}^{F, *}).
			\end{align*}
			Here $L^2\A_{\bd}(X; F)$ is the completion of $\A_{\bd}(X,F)$ with respect to $(\cdot,\cdot)_{L^2}.$
			
		\end{enumerate}
	\end{thm}

	\subsection{Analytic torsion for flat vector bundles}\label{review-ana}
	
	\subsubsection{On closed manifolds}\label{zetadef}

	Let $(M,g^{TM})$ be a closed Riemannian manifold, $F\to M$ be a flat vector bundle with a Hermitian metric $h^F$, and $\nabla^F$ be a flat connection on $F$. Let $d^F$ be the covariant differential on $\Omega(M;F)$ induced by $\nabla^F$, then we have a complex
	\be\label{cplxf} 0 \rightarrow \Omega^0(M;F) \stackrel{d^F}{\rightarrow} \Omega^1(M;F) \stackrel{d^F}{\rightarrow} \cdots \stackrel{d^F}{\rightarrow} \Omega^{\dim(M)}(M;F) \rightarrow 0.\ee
	Denote $H(M;F)$ to be the cohomology of this complex. One can see that $g^{TM}$ and $h^F$ induce an $L^2$-inner product $(\cdot,\cdot)_{L^2(M)}$ on $\Omega^*(M;F)$. 
	
	Let $d^{F,*}$ be the formal adjoint of $d^F$ with respect to metric $g^{TM}$ and $h^F$. Then the Hodge Laplacian $\Delta:\Omega(M;F)\to\Omega(M;F)$ is defined as
	\[\Delta:=(d^F+d^{F,*})^2.\]
	
	$\Omega(M;F)$ has a natural $\Z_2$ grading: 
	\[\Omega^+:=\oplus_{k \mbox{ is even }}\Omega^k(M;F),\Omega^-:=\oplus_{k \mbox{ is odd}}\Omega^k(M;F).\]
	For a trace class operator $A:\Omega(M;F)\to \Omega(M;F)$, $\Tr_s(A)$ denotes its supertrace. If $A$ has an integral kernel $a$, the pointwise supertrace of $a$ is denoted by $\tr_s(a)$. Let $N:\Omega(M;F)\to \Omega(M;F)$ be a linear operator, such that for $\a\in \Omega^k(M;F)$, $N\a=k\a.$ $N$ is called the number operator. Let $\P$ be the orthogonal projection onto $\ker(\Delta)$, $$\Delta':=\Delta|_{\Im(1-\P)}.$$

		The zeta function $\zeta$ for $\Delta$ is defined as
		\[\zeta(z):=\frac{1}{\Gamma(z)}\int_0^\infty t^{z-1}\Tr_s\qty(Ne^{-t\Delta'})dt.\]
		where $\Gamma$ is the Gamma function. The zeta function is well defined whenever $\Re(z)$ is large enough. And it could be extended to a meromorphic function on $\C$. Moreover, $\zeta$ is holomorphic at $0$.
		
		The well-known Ray-Singer analytic torsion $\mathcal{T}(g^{TM},h^F,\nabla^F)$ for the complex (\ref{cplxf}) is defined as
		\[\mathcal{T}(g^{TM},h^F,\nabla^F):=e^{\frac{1}{2}\zeta'(0)}.\]

	\subsubsection{On manifolds with boundary}\label{oabs}
	Let $(X,g^{TX})$ be a compact manifold with boundary $Y=\partial X$, $g^{T X}$ be a Riemannian metric on $X$. Let $F\to X$ be a flat vector bundle with flat connection $\nabla^F$, and $h^{F}$ be a Hermitian metric on $F$. We identify a neighborhood of $\partial X$ to $(-1,0] \times Y$. Let $(s,y)$ be its coordinates.
	
	Let $d^{F, *}$ be the formal adjoint of the de Rham operator $d^{F}$ with respect to the $L^{2}$ metric $(\cdot, \cdot)_{L^2(X)}$ induced from $h^F$ and $g^{TX}$.
	
	Set
	$$\Omega_{\mathrm{abs}} (X;F):=\left\{\omega \in \Omega (X;F): i_{\frac{\partial}{\partial s}} \omega=0, i_{\frac{\partial}{\partial s}} d^{F} \omega=0\mbox{ on }Y\right\},$$
	$$
	\Omega_{\mathrm{rel}} (X;F):=\left\{\omega \in \Omega (X;F): d s \wedge \omega=0, d s \wedge d^{F, *} \omega=0 \text { on } Y\right\}.
	$$
	We write $\Omega_{\mathrm{b d}} (X;F)$ for short if the choice of abs/rel is clear. 
	
	Let $\Delta_{bd}:=(d^F+d^{F,*})^2$ act on $\Omega_{\bd}(X;F)$. According to the Hodge theory, $\ker(\Delta_\bd)\cong H_\bd(X; F)$. Here $H_{\rel}(X;F):=H(X,\p X;F)$, and $H_{\abs}(X;F):=H(X;F)$.
	
	Let $\zeta_{\bd}$ be the zeta functions for $\Delta_{\bd} .$

	The analytic torsion $\T_{bd}(g^{TX},h^F,\nabla^F)$ is defined by
	$e^{ \frac{1}{2} \zeta_{\bd}^{\prime}(0)}.$
	
	In particular, let $Y\subset M$ be a hypersurface cutting $M$ into two pieces $M_1$ and $M_2$. Denote $g^{TM_i}$, $F_i$, $h^{F_i}$ and $\nabla^{F_i}$ to be the restriction of $g^{TM},F,h^F$ and $\nabla^F$ to $M_i$ respectively $(i=1,2)$. One can see that $g^{TM_i}$ and $h^{F_i}$ induce an $L^2$-inner product $(\cdot,\cdot)_{L^2(M_i)}$ on $\Omega^*(M_i;F_i)$. 
	
	Let $\Delta_1:=(d^{F_1}+d^{F_1,*})^2$ act on $\Omega_{\abs}(M_1;F_1)$, and
	$\Delta_2:=(d^{F_2}+d^{F_2,*})^2$ act on $\Omega_{\rel}(M_2;F_2)$.
	
	\begin{defn}\label{rsbd}
		We set $\zeta_i$ to be the zeta function for $\Delta_i$. And let
		
		$$\T_1(g^{M_1},h^{F_1},\nabla^{F_1}):=\T_{\abs}(g^{M_1},h^{F_1},\nabla^{F_1})=e^{\frac{1}{2}\zeta_1'(0)},$$ and $$\T_2(g^{M_2},h^{F_2},\nabla^{F_2}):=\T_{\rel}(g^{M_2},h^{F_2},\nabla^{F_2})=e^{\frac{1}{2}\zeta_2'(0)}.$$
	\end{defn}
	\subsection{Analytic torsion for Witten Laplacian and weighted Laplacian}\label{aw}

	Let $Y\subset M$ be a hypersurface cutting $M$ into two pieces $M_1$ and $M_2$. Denote $g^{TM_i}$, $F_i$, $h^{F_i}$ and $\nabla^{F_i}$ to be the restriction of $g^{TM},F,h^F$ and $\nabla^F$ to $M_i$ respectively$(i=1,2)$. 
	We identify a neighborhood of $Y=\p M_1$ in $M_1$ to $(-2,-1] \times Y$, and $\p M_1$ is identified with $\{-1\}\times Y$. Similarly,  we identify a neighborhood of $Y=\p M_2$ in $M_2$ to $[1,2) \times Y$, and $\p M_2$ is identified with $\{1\}\times Y$.  Let $\bm=M_1\cup[-1,1]\times Y\cup M_2$ (see Figure \ref{fig2}), and $\bar{F},h^{\bar{F}},g^{T\bm}$ and $\nabla^{\bar{F}}$ be the natural extensions of $F,h^F,g^{TM}$ and $\nabla^F$ to $\bm$. One can see that $g^{T\bm}$ and $h^{\F}$ induce an $L^2$-inner product $(\cdot,\cdot)_{L^2(\bm)}$ on $\Omega^*(\bm;\F)$.

	\def\Cn{\mathcal{C}}
	\def\cn{\delta}
	
	Let $f_T,T\geq0$ be a family of odd smooth functions on $[-2,2]$ (see Figure \ref{fig3}), such that 
	\begin{enumerate}[(a)]
		\item $f_T|_{[1,2]}\equiv T/2$,
		\item\label{b} $f_T|_{[1/2,1]}(s)=-T\rho\qty(e^{T^2}(1-s))(s-1)^2/2+T/2$ , where $\rho\in C^\infty([0,\infty))$, such that $0\leq\rho\leq1$, $\rho|_{[0,1/4]}\equiv0,$ $\rho|_{[1/2,\infty)}\equiv1,$ $|\rho'|\leq \cn_1$ and $|\rho''|\leq \cn_2$ for some universal constant $\cn_1$ and $\cn_2.$
		\item\label{c} $\Cn_1T\leq|f_T'|(s)\leq 2\Cn_1 T$, $|f_T''|\leq \Cn_2T$ for some universal constants $\Cn_1$ and $
		\Cn_2$ whenever $s\in[0,1/2].$
		
	\end{enumerate}

	\begin{lem}\label{constr}
		$f_T$ exists.  Moreover, it also satisfies
		\begin{enumerate}[(A)]
			\item\label{A} For $s\in(1-e^{-T^2},1)\cup(-1,-1+e^{-{T^2}})$,$|f_T'|(s)\leq \Cn_3T\left||s|-1\right|$ and $|f_T''|(s)\leq \Cn_4T$ for some universal constant $\Cn_3$ and $\Cn_4$. 
			\item\label{B} For $s\in(1,2)\cup(-2,-1)$, $f_T'(s)=f_T''(s)=0.$
			\item\label{C} For $s\in[1/2,1-e^{-T^2}]$, $f_T'(s)=-T(s-1),f_T''(s)=-T$ and for $s\in[-1+e^{-T^2},-1/2]$, $f_T'(s)=T(s+1),f_T''(s)=T.$
			\item\label{D} $f_T'\geq0.$
		\end{enumerate}  
	\end{lem}
	\begin{proof}
		Let $p_T$ be an odd function on $[-2,-1/4]\cup[1/4,2]$, such that
		\begin{itemize}
			\item $p_T|_{[1,2]}\equiv T/2$,
			\item $p_T|_{[1/4,1]}(s)=-T\rho\qty(e^{T^2}(1-s))(s-1)^2/2+T/2$ , where $\rho\in C_c^\infty([0,\infty))$, such that $0\leq\rho\leq1$, $\rho|_{[0,1/4]}\equiv0,$ $\rho|_{[1/2,\infty)}\equiv1,$ $|\rho'|\leq \cn_1$ and $|\rho''|\leq \cn_2$ for some universal constant $\cn_1$ and $\cn_2.$
		\end{itemize}
		
		It is clear that $p_T(s)=-T(s-1)^2/2+T/2$ if $s\in[1/4,1/2].$
		
		Fix an odd smooth function $q\in[-1/2,1/2]$, such that $q$ is strictly increasing, and $q|_{[1/4,1/2]}=-(s-1)^2/2+1/2.$
		
		Then we set
		\[f_T(s)=\begin{cases}
			p_T(s),\mbox{ if $s\in[1/4,2]$},\\
			Tq(s),\mbox{ if $s\in[0,1/2].$}
		\end{cases}\]
		
		It's easy to check that such $f_T$ satisfies (a), (b) and (c).
		
		Now (B), (C) and (D) are clear, we just need to check (A).
		
		For $s\in(1-e^{-T^2},1)$, by the product rule and the chain rule, 
		\[|p_T'(s)|=|T\rho'(e^{T^2}(1-s))e^{T^2}\cdot (s-1)^2/2-T\rho(e^{T^2}(1-s))(s-1)|\leq T(\delta_1+1)|s-1|,\]
		where the inequality follows from the fact that $|e^{T^2}\cdot(1-s)|\leq 1\leq2$ if $s\in(1-e^{-T^2},1).$ Similarly, we have estimates for the second derivatives.
	\end{proof}
	
	While the behavior of $f_T$ on $[-1,-1+e^{-T^2}]\cup[1-e^{-T^2},1]$ seems uncertain, it is noteworthy that their first two derivatives are controllable, and the intervals above are very small. As we can see in \eqref{addeq3}, only the first two derivatives of $f_T$ are involved in $\Delta_T$ and are considered throughout the discussions in this paper. Additionally, the term $e^{T^2}$ plays an important role in proving Lemma \ref{kx0}.
	
	We could regard $f_T$ as a function on $\bm.$ 
	Let $d_{T}:=d^{\bar{F}}+df_T\wedge.$ Then the Witten Laplacian $\Delta_T$ is the Hodge Laplacian with respect to $d_T$. 
	
	\begin{defn}\label{rsclt}
		We have a complex
		\be\label{cplxff} 0 \rightarrow \Omega^0(\bm,\bar{F}) \stackrel{d_T}{\rightarrow} \Omega^1(\bm,\bar{F}) \stackrel{d_T}{\rightarrow} \cdots \stackrel{d_T}{\rightarrow} \Omega^{\dim(M)}(\bm,\bar{F}) \rightarrow 0.\ee Denote $H(\bm,\bar{F})(T)$ to be the cohomology of this complex, and $\Delta_T$ to be the Hodge Laplacian for $d_T$. Let $\zeta_T$ be the zeta function for $\Delta_T$. Similarly, one could define Ray-Singer analytic torsion $\mathcal{T}(g^{T\bm},h^{\bar{F}},\nabla^{\bar{F}})(T):=e^{\frac{1}{2}\zeta_T'(0)}$ for the complex (\ref{cplxff}). 
	\end{defn}
	
	Lastly, for simplicity, $(\cdot,\cdot)_{L^2}$ (resp. $\|\cdot\|_{L^2}:=\sqrt{(\cdot,\cdot)_{L^2}}$) will be adopted to represent $(\cdot,\cdot)_{L^2(M)}$ (resp. $\|\cdot\|_{L^2(M)}:=\sqrt{(\cdot,\cdot)_{L^2(M)}}$) , $(\cdot,\cdot)_{L^2(\bm)}$ (resp. $\|\cdot\|_{L^2(\bm)}:=\sqrt{(\cdot,\cdot)_{L^2(\bm)}}$)  or $(\cdot,\cdot)_{L^2(M_i)}$(resp. $\|\cdot\|_{L^2(M_i)}:=\sqrt{(\cdot,\cdot)_{L^2(M_i)}}$)  ($i=1,2$), when the context is clear.

	\subsubsection{Witten Laplacian v.s. weighted Laplacian}\label{witwei}
	
	Instead of deforming the de Rham differential $d^{\F}$, we could also deform the metric $h^{\F}$: let $h^{\F}_T:=e^{-2f_T}h^{\F}$. Similarly, $g^{T\bm}$ and $h^{\F}_T$ induce an $L^2$-norm $(\cdot,\cdot)_{L^2(\bm),T}$ on $\Omega(\bm;\F).$ 
	
	Then the formal adjoint $d^{\F,*}_T$ of $d^{\F}$ for $(\cdot,\cdot)_{L^2(\bm),T}$ equals to $e^{f_T}d_T^*e^{-f_T}$. The weighted Laplacian $\tD_T:=d^{\F}d^{\F,*}_T+d^{\F,*}_Td^{\F}.$ Then $\tD_T=e^{f_T}\Delta_T e^{-f_T}.$ Let $l_k(T)$ be the $k$-th eigenvalue of $\tD_T$, then $l_k(T)=\l_k(T).$ Moreover, if \(u\) is an eigenform of \(\Delta_T\) with eigenvalue \(\lambda\), then \(e^{f_T}u\) is an eigenform of \(\tD_T\) with the same eigenvalue \(\lambda\).
	
	\subsubsection{Absolute/Relative boundary conditions for weighted Laplacian}
	Let $\bm_1:=M_1\cup [-1,0]\times Y$, $\bm_2:=M_2\cup[0,1]\times Y$ (see Figure \ref{fig2}), and $\F_i$ be the restriction of $\F$ on $\bm_i$ ($i=1,2$).
	Set
	$$\A_{\mathrm{abs}} (\bm_1;\F_1):=\left\{\omega \in \Omega (\bm_1;\F_1): i_{\frac{\partial}{\partial s}} \omega=0\mbox{ on }\{0\}\times Y\right\},$$
	$$\A_{\mathrm{rel}} (\bm_2;\F_2):=\left\{\omega \in \Omega (\bm_2;\F_2): d s \wedge \omega=0 \text { on } \{0\}\times Y\right\};$$
	$$\Omega_{\mathrm{abs}} (\bm_1;\F_1)_T:=\left\{\omega \in \Omega (\bm_1;\F_1): i_{\frac{\partial}{\partial s}} \omega=0, i_{\frac{\partial}{\partial s}} d^{\F_1} \omega=0\mbox{ on }\{0\}\times Y\right\},$$
	$$\Omega_{\mathrm{rel}} (\bm_2;\F_2)_T:=\left\{\omega \in \Omega (\bm_2;\F_2): d s \wedge \omega=0, d s \wedge d_T^{\F_2, *} \omega=0 \text { on } \{0\}\times Y\right\};$$
	$$\Omega_{\mathrm{abs}} (\bm_1;\F_1)(T):=\left\{\omega \in \Omega (\bm_1;\F_1): i_{\frac{\partial}{\partial s}} \omega=0, i_{\frac{\partial}{\partial s}} d_T \omega=0\mbox{ on }\{0\}\times Y\right\},$$
	$$\Omega_{\mathrm{rel}} (\bm_2;\F_2)(T):=\left\{\omega \in \Omega (\bm_2;\F_2): d s \wedge \omega=0, d s \wedge d_T^{*} \omega=0 \text { on } \{0\}\times Y\right\}.$$

	Let $\tD_{T,i}$ be the restriction of $\tD_T$ acting on $\Omega_{\bd}(\bm_i;\F_i)_T$ (here $\bd$ represents either $
\abs$ or $
\rel$ ), $\Delta_{T,i}$ be the restriction of $\Delta_T$ acting on $\Omega_{\bd}(\bm_i;\F_i)(T)$. Then by Hodge theory and the discussion above, $\ker(\Delta_{T,i})\cong\ker(\tD_{T,i})\cong H_{\bd}(\bm_i;\F_i)$.
	\begin{defn}\label{defn24} Similarly, one could define Ray-Singer analytic torsion $$\mathcal{T}_i(g^{T\bm_i},h^{\bar{F}_i},\nabla^{\bar{F}_i})(T):=e^{\frac{1}{2}\zeta_{T,i}'(0)}$$  for $\Delta_{T,i}$ (or $\tilde{\Delta}_{T,i}$). \end{defn}
	
    One can see that $g^{T\bm_i}$ and $h_T^{\F_i}$ (the restriction of $g^{T\bm}$ and $h^{\F}_T$ on $\bm_i$) induce an $L^2$-inner product $(\cdot,\cdot)_{L^2(\bm_i),T}$ on $\Omega^*(\bm_i;\F_i)$. For simplicity, $(\cdot,\cdot)_{L^2,T}$ (resp. $\|\cdot\|_{L^2,T}:=\sqrt{(\cdot,\cdot)_{L^2,T}}$) will be adopted to represent $(\cdot,\cdot)_{L^2(\bm),T}$ (resp. $\|\cdot\|_{L^2(\bm),T}:=\sqrt{(\cdot,\cdot)_{L^2(\bm),T}}$) , or $(\cdot,\cdot)_{L^2(\bm_i),T}$(resp. $\|\cdot\|_{L^2(\bm_i),T}:=\sqrt{(\cdot,\cdot)_{L^2(\bm_i),T}}$)  ($i=1,2$), when the context is clear.

	\section{Four Key Intermidiate Results}\label{intr}
In this section, we state some key intermediate results, from which Theorem \ref{main0} follows easily. The proofs of these intermediate results occupy the remaining sections.

	Let $\l_k(T)$ be the $k$-th eigenvalue for $\Delta_T$, $\l_k$ be the $k$-th eigenvalue of $\Delta_1\oplus\Delta_2$ acting on $\Omega_{\abs}(M_1;F_1)\oplus\Omega_{\rel}(M_2;F_2),$ 
	and $\tl_k(T)$ be the $k$-th eigenvalue of $\tD_{T,1}\oplus\tD_{T,2}$ acting on $\Omega_{\abs}(\bm_1;\F_1)_T\oplus\Omega_{\rel}(\bm_2;\F_2)_T$ (or equivalently, $k$-th eigenvalue of $\Delta_{T,1}\oplus\Delta_{T,2}$ acting on $\Omega_{\abs}(\bm_1;\F_1)(T)\oplus\Omega_{\rel}(\bm_2;\F_2)(T)$).
	\begin{thm}[Convergence of eigenvalues]\label{eigencon}
		We have $$\lim_{T\to\infty}\l_k(T)=\lim_{T\to\infty}\tl_k(T)=\l_k.$$
	\end{thm}

	Let $\delta>0$ denote half of the first nonzero eigenvalue of $\Delta_1\oplus\Delta_2$. Then by Theorem \ref{eigencon}, all eigenvalues of $\Delta_T$ inside $[0,\delta]$ converge to $0$ as $T\to\infty$,  and all eigenvalues of $\tD_{T,1}\oplus\tD_{T,2}$ inside $[0,\delta]$ are $0$  when $T$ is large enough. 
    
    Let $\Omega_{\sm}(\bm,\bar{F})(T)$ be the space generated by eigenforms for eigenvalues of $\Delta_T$ inside $[0,\delta]$,
	and $\P^{\delta}(T)$ be the orthogonal projection onto $\Omega_{\sm}(\bm,\bar{F})(T)$.
	
	Let $$\zeta_{T,\la}:=\frac{1}{\Gamma(z)}\int_0^\infty t^{z-1}\Tr_s\qty(Ne^{-t\Delta_T'}\qty(1-\P^\delta(T)))dt,$$
	$$\zeta_{T,\sm}:=\frac{1}{\Gamma(z)}\int_0^\infty t^{z-1}\Tr_s\qty(Ne^{-t\Delta_T'}\P^\delta(T))dt.$$
	For $i=1,2$, let
	\[\zeta_i^{\mL}(z):=\frac{1}{\Gamma(z)}\int_1^\infty t^{z-1}\Tr_s(Ne^{-t\Delta'})dt,\]
	\[\zeta_i^{\mS}(z):=\frac{1}{\Gamma(z)}\int_0^1 t^{z-1}\Tr_s(Ne^{-t\Delta'})dt.\]
	Then it is clear that $\zeta_i=\zeta_i^{\mL}+\zeta_i^{\mS}$. 
	
	Similarly, one can define $\zeta_{T,i}^{\mL}$, $\zeta_{T,i}^{\mS}$, $\zeta^{\mL}_{T,\la}$, and $\zeta^{\mS}_{T,\la}$ e.t.c.
	\begin{rem}
	  Here, \(\la/\sm\) refers to large/small eigenvalues, and \(\mL/\mS\) refers to large/small time.
	\end{rem}
	
	\begin{thm}[Convergence of the large-time components]\label{larcon}
		\[\lim_{T\to\infty} (\zeta^{\mL}_{T,\la})'(0)=\lim_{T\to\infty}\sum_{i=1}^2(\zeta_{T,i}^{\mL})'(0)=\sum_{i=1}^2(\zeta_i^{\mL})'(0).\]
		That is, 
		\begin{align*}&\ \ \ \ \lim_{T\to\infty}\int_1^\infty t^{-1}\Tr_s\qty(Ne^{-t\Delta_T'}(1-\P^\delta))dt\\
			&=\lim_{T\to\infty}\sum_{i=1}^2\int_1^\infty t^{-1}\Tr_s(Ne^{-t\Delta_{T,i}'})dt=\sum_{i=1}^2\int_1^\infty t^{-1}\Tr_s(Ne^{-t\Delta_i'})dt.
		\end{align*}
	\end{thm}

	\begin{thm}[Convergence of the small-time components]\label{main}
		As $T\to\infty,$
		\[(\zeta^{\mS}_{T,\la})'(0)=\sum_{i=1}^2(\zeta_i^{\mS})'(0)-\left(T-\log (2)\right)\chi(Y)\rank(F)+o(1),\]
		\[(\zeta^{\mS}_{T,i})'(0)=(\zeta_i^{\mS})'(0)-T\chi(Y)\rank(F)/2+o(1).\]
		
		Thus, as $T\to\infty,$
		\[(\zeta^{\mS}_{T,\la})'(0)-\sum_{i=1}^2(\zeta^{\mS}_{T,i})'(0)=\log(2)\chi(Y)\rank(F)+o(1).\]
		
	\end{thm}
	
	Next, we have the following Mayer-Vietoris exact sequence (c.f. \cite[(0.16)]{bruning2013gluing})
	\be\label{mv}
	\mathcal{MV}:\cdots \stackrel{\p_{k-1}}\rightarrow H_{\text {rel }}^k\left(\bm_2;\F_2\right)  \stackrel{e_k}{\rightarrow}H^k\left(\bm;\F\right)  \stackrel{r_k}{\rightarrow} H_{\mathrm{abs}}^k\left(\bm_1;\F_1\right)  \stackrel{\p_k}{\rightarrow} \cdots .
	\ee
	Let $\H(\bm;\F)(T):=\ker(\tD_{T})$, and $\H(\bm_i;\F_i)(T):=\ker(\tD_{T,i})$. 
	We also have the following Mayer-Vietoris exact sequence induced by Hodge theory and (\ref{mv})
	\be\label{mvp}
	\mathcal{MV}(T):\cdots \stackrel{\p_{k-1,T}}\rightarrow \H^k\left(\bm_2;\F_2\right)(T)  \stackrel{e_{k,T}}{\rightarrow}\H^k\left(\bm;\F\right)(T)  \stackrel{r_{k,T}}{\rightarrow} \H^k\left(\bm_1;\F_1\right)(T)  \stackrel{\p_{k,T}}{\rightarrow} \cdots 
	\ee
	with metric induced by harmonic forms for $\tilde{\Delta}_{T}$ and $\tilde{\Delta}_{T,i},i=1,2.$ Let $\mathcal{T}(T)$ be the analytic torsion for this complex.

	Recall that $$\zeta_{T,\sm}:=\frac{1}{\Gamma(z)}\int_0^\infty t^{z-1}\Tr_s(Ne^{-t\Delta_T'}\P^\delta(T))dt.$$
	And set $\T_{\sm}(g^{T\bm},h^{\F},\nabla^{\F})(T):=e^{\half\zeta_{T,\sm}'(0)}.$

The following theorem describes the relationship between the analytic torsion for small eigenvalues and the analytic torsion for the Mayer–Vietoris exact sequence, which will be proved in \cref{lastreal}.
	\begin{thm}\label{int6p}
		$\lim_{T\to\infty}\log\T_{\sm}(g^{T\bm},h^{\F},\nabla^{\F})(T)-\log\T(T)=0.$
	\end{thm}
	
	Let $\T_i(g^{T\bm_i},h^{\F_i},\nabla^{\F_i})(T)$ be the analytic torsion associated with $\tilde\Delta_{T,i}$, then it follows from the anomaly formula \cite[Theorem 0.1]{bismutzhang1992cm} and \cite[Theorem 0.1]{bruning2006anomaly} that
	
	\begin{thm}\label{int1p}
		\begin{align*}
			&\ \ \ \ \log\T(g^{T\bm},h^{\F},\nabla^{\F})(T)-\sum_{i=1}^2\log\T_i(g^{T\bm_i},h^{\F_i},\nabla^{\F_i})(T)-\log\T(T)\\
			&=\log\T(g^{TM},h^{F},\nabla^{F})-\sum_{i=1}^2\log\T_i(g^{TM_i},h^{F_i},\nabla^{F_i})-\log\T.
		\end{align*}
	\end{thm}

	\begin{proof}[Proof of Theorem \ref{main0}]
		It follows from Theorem \ref{larcon},  \ref{main} and \ref{int6p} that
		\begin{align*}
			&\ \ \ \ \log\T(g^{T\bm},h^{\F},\nabla^{\F})(T)-\sum_{i=1}^2\log\T_i(g^{T\bm_i},h^{\F_i},\nabla^{\F_i})(T)-\log\T(T)
			\\&=\log(2)\chi(Y)\rank(F)/2+o(1).
		\end{align*}
		Hence, by Theorem \ref{int1p}, Theorem \ref{main0} follows (under the assumption of product-type metrics).
\end{proof}

	\section{Convergence of Eigenvalues}\label{eigen}
In this section, we prove the convergence of eigenvalues (\Cref{eigencon}). In \cref{sec-review}, we review key properties of Witten-deformed Dirac-type operators. In \cref{tube}, we derive estimates for eigenforms on the tube. Using these estimates, we establish \Cref{eigencon} in \cref{proof-eigencon} via a Rayleigh quotient argument.  

From this point onward, \textbf{certain statements will be assumed to hold only for sufficiently large \( T \)}, and this condition will not be explicitly restated each time.

 \subsection{Review on Witten deformed Dirac-type operators}\label{sec-review}
Before proceeding, let’s briefly review some basic facts about (Witten-deformed) Dirac-type operators on $F$-valued differential forms (see also \cite[\S 4.a \& \S4.e]{bismutzhang1992cm}). Due to the presence of several subtleties in this section, and for the reader's convenience, the author has included material in this subsection that may seem obvious and straightforward to experts.
 
	For simplicity, in this section, we abbreviate \( d^{\F} \) and \( d^{\F,*} \) as \( d \) and \( d^* \), respectively. Let \( D = d + d^* \) and \( D_T = d_T + d_T^* \) be Dirac-type and Witten-deformed Dirac-type operators. 
 \def\given{1}
 \if\given1
 Given that \( g^{T\bm} \) and \( h^{\F} \) are of product type on \( [-2,2] \times Y \), on \( [-2,2] \times Y \), one has  canonical decomposition $D_T = D_T^\R + D^Y,$
where $D_T^\R$ and $D^Y$ are given locally as follows.
 
 Let $\{e_i\}_{i=1}^{\dim(M)}$ be a local orthonormal frame of $T\bm$ on $[-2,2]\times Y$, such that $e_1=\frac{\p}{\p s}$.  Let $\nabla$ be the connection on $\Lambda(T^*\bm)\otimes\F$ induced by $\nabla^{\F}$ and $g^{T\bm}$, and let $\nabla^*$ be the dual connection of $\nabla$. Let $\{e^i\}$ be the dual frame of $\{e_i\}$, then set $$D_T^\R=ds\wedge\nabla_{\frac{\p}{\p s}}-i_{\frac{\p}{\p s}}\nabla^{*}_{\frac{\p}{\p s}}+(df_T\wedge)+i_{\nabla f_T}, D^Y=\sum_{i\geq2}e^i\wedge\nabla_{e_i}-i_{e_i}\nabla^*_{e_i}.$$ 
 
 For any vector field $X\subset \Gamma(T\bm)$, let $X^{\flat}$ be the 1-form, such that for any  vector field $X'$, $X^{\flat}(X')=g^{T\bm}(X,X').$  For any vector field $X$, let $c(X),\hc(X)\in \End(\Lambda(T^*\bm)\otimes \F),$ such that $$c(X)=(X^\flat\wedge)-i_X, \hc(X)=(X^\flat\wedge)+i_X.$$ Let $\o:=\nabla^{\F,*}-\nabla^\F$, then $\o$ is an $\End(\F)$-valued 1-form. By \eqref{ddsh}, $\o(\frac{\p}{\p s})=0$. Let $\{e_i\}_{i=1}^{\dim(M)}$ be a local frame of $T\bm$. Since $\nabla^\F$ could be non-unitary, $D\neq \sum_{i=1}c(e_i)\nabla_{e_i}.$ However, it is straightforward to see that (by \cite[(4,26) and (4.27)]{bismutzhang1992cm}) $$D= \sum_{i=1}c(e_i)\nabla_{e_i}-i_{e_i}\o(e_i).$$ It can be verified that \( D_T = D + \hc(\nabla f_T) \). Consequently, for any \( \phi \in C^\infty(\bm) \), the Leibniz rule holds for $D_T$ and $D$:
\be\label{addeq1}
	D\phi u=c\left(\nabla \phi\right)u+\phi Du\mbox{ and }D_T\phi u=c\left(\nabla \phi\right)u+\phi D_Tu.
	\ee
	Since $f_T':=\frac{\p}{\p s} f_T$ and $f_T'':=\frac{\p^2}{\p s^2}f_T$ are supported inside $(-2,2)\times Y$, they could be viewed as elements in $C^\infty(\bm)$, then operator $L_T:=f_T''c(\frac{\p}{\p s})\hc(\frac{\p}{\p s})$ could be extended to $\bm$. By \eqref{ddsh}, it can be seen easily that
	\be\label{addeq3}
	\Delta_T=\Delta+L_T+|f_T'|^2,
	\ee
	where $\Delta:=D^2=d^{\F}d^{\F,*}+d^{\F,*}d^{\F}.$
	
	For any differential form $u\in\Omega(\bm,\F)$, one has a
 decomposition on $[-2,2]\times Y$:  Let $\a=i_{\frac{\p}{\p s}}ds\wedge u$ and $\b=i_{\frac{\p}{\p s}} u$, then \be\label{candec} u=\a+ds\wedge\b.\ee It follows from  \eqref{addeq3} that acting on the $\a$-components (on $[-2,2]\times Y$),
	\be\label{addeq31}
	\Delta_T\a=\Delta\a-f_T''\a+| f_T'|^2\a;
	\ee
	while on the $ds\wedge\b$-components (recalling \Cref{conv1}),
	\be\label{addeq32}
	\Delta_Tds\wedge\b=\Delta ds\wedge\b+f_T''ds\wedge\b+| f_T'|^2ds\wedge\b.
	\ee
	Let $\lan\cdot,\cdot\ran$ be the metric on $\Lambda(T^*\bm)\otimes \F$ induced by $g^{T\bm}$ and $h^{\F}$, and $|\cdot|:=\sqrt{\lan\cdot,\cdot\ran}$. Since the metric $g^{T\bm}$ and $h^{\F}$ are product-type, $\Delta_T$ preserves the  decomposition \eqref{candec} on $[-2,2]\times Y$ above. So it is easy to see that for any $\phi\in C_c^\infty((-2,2)\times Y),$ \be\lan\label{25} \Delta_T \phi \a,\phi ds\wedge\b\ran=\lan  \phi \a,\Delta_T\phi ds\wedge\b\ran=0.\ee
As a result, for any \( \phi \in C_c^\infty((-2,2) \times Y) \) (recalling \Cref{conv1}),
	\begin{align}\label{addeq5}
		\begin{split}
			&\ \ \ \	\int_{\bm}|D_T\phi ds\wedge\b|^2+|D_T\phi \a|^2\dvol_{\bm}\\
			&=\int_{\bm}\lan \Delta_T\phi ds\wedge\b,\phi ds\wedge\b\ran+	\lan \Delta_T\phi\a, \phi\a\ran\dvol_{\bm}\mbox{ ( Integration by parts)}\\
			&=\int_{\bm}\lan \Delta_T\phi u,\phi u\ran\dvol_{\bm}=\int_{\bm} |D_T\phi u|^2\dvol_{\bm}\mbox{ (By \eqref{25}, then integration by parts)}.
		\end{split}
	\end{align}

 Lastly, the proof of Theorem \ref{eigencon} frequently makse use of the following extensions. Consider any function $\rho \in C^\infty(\mathbb{R})$ such that $\rho|_{[2,\infty)}\equiv c_2$ and $\rho|_{(-\infty,-2]}\equiv c_1$ are constant functions. We can treat $\rho$ as a smooth function on $\bm$ in the following way: For $(s,y) \in [-2,2]\times Y$, $\rho(s,y)=\rho(s)$, while $\rho|_{M_1-[-2,-1]\times Y}=c_1$ and $\rho|_{M_2-[1,2]\times Y}=c_2$. Similarly, the derivative $\rho'=\frac{d}{ds}\rho$ can be interpreted as an element in $C_c^\infty((-2,2)\times Y)\subset C^\infty(\bm)$.

	\subsection{Estimates of Eigenforms on the tube $[-1,1]\times Y$}\label{tube}
Recall that \(\lambda_k(T)\) denotes the \(k\)-th eigenvalue of \(\Delta_T\), \(\tilde{\lambda}_k(T)\) is the \(k\)-th eigenvalue of \(\Delta_{T,1} \oplus \Delta_{T,1}\), and \(\lambda_k\) is the \(k\)-th eigenvalue of \(\Delta_1 \oplus \Delta_2\).  In this subsection, we establish key estimates to prove Theorem \ref{eigencon}. First, in \Cref{o1}, we obtain a uniform bound for \(\lambda_k(T)\) and \(\tilde{\lambda}_k(T)\) using simple Rayleigh quotient arguments. Next, we derive a trace formula for Witten-deformed Dirac operators in \Cref{limit1}. Then, by \Cref{limit1} and \Cref{cor-limit1}, we prove \Cref{limit2}, which implies that the eigenforms of \(\Delta_T\) converge to differential forms satisfying the desired boundary conditions. Finally, we establish \Cref{limit0}, implying that the \(L^2\)-norm of the eigenforms on the tube becomes negligible as \(T \to \infty\).

		 First, one observes that $\l_k(T)$ and $\tilde{\l}_k(T)$ have uniform upper bounds:
	\begin{lem}\label{o1}
		There exists an increasing $T$-independent sequence $\{\Lambda_k\}_{k=1}^\infty$, such that $\lambda_{k}(T)\leq \Lambda_k$ and $\tl_k(T)\leq \Lambda_k$.
	\end{lem}
	\begin{proof}
		Choose $k$ disjoint balls $\{B_j\}_{j=1}^k$ in $M_1,$ $k$ nonzero smooth functions $\eta_j$ with support $\mathrm{supp}(\eta_j)\subset B_j$. Let $V_k$ be the linear space generated by $\{\eta_j\}$. Then by the Rayleigh quotient argument and the fact that $D_T=D$ on $M_1,$ one can see
		$$\l_{k}(T)\leq \Lambda_k:=\sup_{\psi\in V_k}\frac{\int_{\bm}|D\psi|^2\dvol_{\bm}}{\int_{\bm}|\psi|^2\dvol_{\bm}}=\sup_{\psi\in V_k}\frac{\int_{\bm}|D_T\psi|^2\dvol_{\bm}}{\int_{\bm}|\psi|^2\dvol_{\bm}}.$$
		Similarly, $\tl_k(T)\leq\Lambda_k$ for some sequence $\{\Lambda_k\}.$
	\end{proof}
	
	\def\mycmd{0}
	\ifx\mycmd\undefined
	undefed
	\else
	\if\mycmd1
	\begin{lem}\label{limit1}
		Let $u\in \Omega(\bm;\F)$ be a unit eigenform for eigenvalue $\leq\l$. Then for $s\in\qty[-2,-1+\sqrt{\frac{2}{T}}]\cup\qty[1-\sqrt{\frac{2}{T}},2]$
		\be\label{new10}\int_Y|u|^2(s,y)\dvol_Y\leq C(\lambda+1).\ee
		\be\label{new1}\int_Y|D_Tu|^2(s,y)\dvol_Y\leq C(\lambda^2+1).\ee

		Let $w_i\in \Omega(\bm_i;\F_i)(T),i=1,2$ be a unit eigenform for eigenvalues $\leq\l$. Then for $s\in\qty[-2,-1+\sqrt{\frac{2}{T}}]$ or $s\in\qty[1-\sqrt{\frac{2}{T}},2]$
		\be\label{new20}\int_Y|w_i|^2(s,y)\dvol_Y\leq C(\lambda+1)\ee
		\be\label{new2}\int_Y|D_Tw_i|^2(s,y)\dvol_Y\leq C(\lambda^2+1)\ee
		
		\be\label{new30}\int_Y|w_i|^2(0,y)\dvol_Y\leq C(\lambda+1).\ee
		\be\label{new3}\int_Y|w_i|^2(0,y)+|D_Tw_i|^2(0,y)\dvol_Y\leq C(\lambda^2+1).\ee
		
	\end{lem}
	\begin{proof}\ \\
		$\bullet$ {\textit{We only prove \eqref{new10} and \eqref{new1} for $s\in[-2,-1+\sqrt{\frac{2}{T}}].$} }\\
		First, by trace formula and G\aa rding's inequality,
		\begin{align*}
			\int_Y |u|^2(-1,y)\dvol_Y&\leq C\int_{M_1}|u|^2+|\nabla u|^2\dvol_{M_1}\leq C\int_{M_1} |u|^2+|(d+d^*)u|^2\dvol_{M_1} \\
			&=C\int_{M_1}|u|^2+|(d_T+d_T^*)u|^2\dvol_{M_1}\leq C(\lambda+1).
		\end{align*}
		Similarly, one still has for $s\in [-2,-1]$,
		\be\label{lnew}\int_Y |u|^2(s,y)\dvol_Y\leq C(\lambda+1).\ee
		\def\c{\gamma}
		
		Next, for $\c\in(0,1)$ to be determined, suppose $s_0\in\qty[-1,-1+\sqrt{\frac{\c}{T}}]$ attains the supremum of $$A_T:=\sup_{s\in [-1,-1+\sqrt{\frac{\c}{T}}]}\int_{Y}|u|^2(s,y)\dvol_Y,$$
		then 
		\begin{align}\begin{split}\label{uuu1}
				\int_{Y}|u(s_0,y)-u(-1,y)|^2\dvol_{Y}&\leq \int_{Y}\left|\int_{-1}^{-1+\sqrt{\frac{\c}{T}}}|\frac{\partial}{\partial s'}u(s',y)|ds'\right|^2\dvol_{Y}\\
				&\leq\sqrt{\frac{\c}{T}}\int_{Y}\int_{-1}^{-1+\sqrt{\frac{\c}{T}}}|d u(s',y)|^2+|d^* u(s',y)|^2ds'\dvol_{Y}.\\
		\end{split}\end{align}
		Integration by parts,
		\begin{align}
			\begin{split}\label{uuu2}
				&\ \ \ \ \l\geq\int_{\bm}|d_T u|^2+|d_T^* u|^2\dvol_{\bm}\geq \int_{-1}^{-1+\sqrt{\frac{\c}{T}}}\int_{Y}|d_T u|^2+|d^*_T u|^2\dvol_{Y}ds\\
				&\geq \int_{-1}^{-1+\sqrt{\frac{\c}{T}}}\int_{Y}|d u|^2+|d^* u|^2+(L_Tu,u)+|f_T'|^2|u|^2\dvol_Yds\\
				&-\int_Y|df_T||u|^2\qty(-1+\sqrt{\frac{\c}{T}},y)\dvol_Y\\
				&\geq \int_{-1}^{-1+\sqrt{\frac{\c}{T}}}\int_{Y}|d u|^2+|d^* u|^2-\Cn_4T|u|^2\dvol_Yds-\sqrt{T}A_T\\
				&\geq\int_{-1}^{-1+\sqrt{\frac{\c}{T}}}\int_{Y}|d u|^2+|d^* u|^2\dvol_Yds-(1+\Cn_4\sqrt\gamma)\sqrt{ T}A_T.\\
			\end{split}
		\end{align}
		See \cref{aw} for the definition of $\Cn_4.$
		By (\ref{uuu1}) and (\ref{uuu2}), one can see that
		\[A_T\leq (\l+1)\qty(\sqrt{\frac{\c}{T}}+C)+\sqrt{\c}(1+\Cn_4\sqrt{\c})A_T.\]
		Fix $\c\in(0,1)$, such that $\sqrt{\c}(1+\Cn_4\sqrt{\c})\leq 1/2.$ Thus, whenever $s\in\qty[-1,-1+\sqrt{\frac{\c}{T}}]$,
		\[\int_Y|u(s,y)|^2\dvol_Y\leq 3C(\l+1).\]
		Similarly, one can show that for $s\in\qty[-1+\sqrt{\frac{\c}{T}}, -1+2\sqrt{\frac{\c}{T}}]$
		\[\int_Y|u(s,y)|^2\dvol_Y\leq 3^2C(\l+1).\]
		Let $m=[\frac{2}{\gamma}]+1$, then repeating the arguments above for $m$ times, one can see that
		\be\label{new11}\int_Y|u(s,y)|^2\dvol_Y\leq 3^mC(\l+1)\ee
		whenever $s\in\qty[-1,-1+\sqrt{\frac{2}{T}}].$
		
		Replace $u$ with $D_Tu$ and repeat the arguments above, one has 
		\[\int_Y|D_Tu(s,y)|^2\dvol_Y\leq C(\l^2+1)\] 
		whenever $s\in\qty[-1,-1+\sqrt{\frac{2}{T}}].$\\
		$\bullet$ {\textit{Similarly, we have \eqref{new20} and \eqref{new2}.}}\\
		$\bullet$ {\textit{We only prove \eqref{new30} and \eqref{new3} for $w_1$.}}\\
		Assume that on $[-2,0]\times Y$, $w_1=\a_1+\b_1 ds$, and set
		\[h(s):=\int_Y|\a_1|^2(s,y)\dvol_Y.\]
		Then $h(0)=\int_Y|w_1|^2(0,y)\dvol_Y$.
		
		We first show that $h(0)\leq C(\l+1).$
		
		Notice that for $s\in[-1,0]$, $f_T$ is increasing, one computes
		\begin{align}\begin{split}\label{new4}
				&\ \ \ \ h'(s)=2\int_Y\lan\partial_s\a_1,\a_1\ran\dvol_Y\\
				&=2\int_Y\lan d_T\a_1,ds\wedge\a_1\ran\dvol_Y-2\int_Y\lan \partial_sf_Tds\wedge\a_1,ds\wedge\a_1\ran\dvol_Y\\
				&\leq 2\int_Y\lan d_T\a_1,ds\wedge\a_1\ran\dvol_Y\leq 2\int_Y|d_T\a_1|^2+|\a_1|^2\dvol_Y.
		\end{split}\end{align}
		
		By (\ref{new4}) and (\ref{new20}), one can see that
		\[h(0)=h(-1)+\int_{-1}^0h'(s)ds\leq C(1+\l).\]
		
		Similarly, one can show that
		\[\int_Y|D_Tw_i|^2(0,y)\dvol_Y\leq C(1+\l^2).\]
	\end{proof}
	\else

	We have the following trace formulas for Witten deformed Dirac operators:
	\begin{lem}\label{limit1}
		Let $u\in \Omega(\bm;\F)$. Then for $s\in\qty[-7/4,7/4]$
		\be\label{new10}\int_Y|u|^2(s,y)\dvol_Y\leq C\left(\int_{[-2,2]\times Y}|D_Tu|^2+|u|^2\dvol_{\bm}\right).\ee
		Let $w_i\in \Omega_\bd(\bm_i;\F_i)(T), i=1,2.$ Then
\be\ba\label{new20}&\int_Y|w_1|^2(s,y)\dvol_Y\leq C\left(\int_{[-2,0]\times Y}|D_Tw_1|^2+|w_1|^2\dvol_{\bm_1}\right), \forall s\in\qty[-7/4,0];\\
        &\int_Y|w_2|^2(s,y)\dvol_Y\leq C\left(\int_{[0,2]\times Y}|D_Tw_2|^2+|w_2|^2\dvol_{\bm_2}\right),\forall s\in\qty[0,7/4].\\
        \ea\ee

	\end{lem}
	\begin{proof}Note that on $[-2,2]\times Y$, one has decompositions (see \eqref{candec}) \be\label{addeq4} u=\a+ds\wedge\b \mbox{ and } w_i=\a_i+ds\wedge\b_i(i=1,2).\ee
		Since $g^{T\bm}$ and $h^{\F}$ are product-type, point-wisely, \be\label{abu}|\a|^2+|\b|^2=|u|^2\mbox{ and }\lan\a,ds\wedge \b\ran=0.\ee
		
		Let $\rho\in C_c^\infty(-2,2)$, such that $\rho|_{(-7/4,7/4]}\equiv1$. Then we can regard $\rho$ as a smooth function on $\bm$.

		For $s\in[-2,2]$, set \[h(s):=\int_Y|\rho\a|^2(s,y)\dvol_Y.\]
	Recalling \Cref{conv1}, we compute:
		\begin{align}\begin{split}\label{new4}
				&\ \ \ \ h'(s)=2\Re\int_Y\lan\partial_s\rho\a,\rho\a\ran(s,y)\dvol_Y=2\Re\int_Y\lan i_{\frac{\p}{\p s}}D\rho\a,\rho\a\ran(s,y)\dvol_Y\\
				&=2\Re\int_Y\lan i_{\frac{\p}{\p s}}D_T\rho\a,\rho\a\ran(s,y)\dvol_Y-2\int_Y f_T'|\rho\a|^2(s,y)\dvol_Y\\
				&\leq 2\Re\int_Y\lan i_{\frac{\p}{\p s}}D_T\rho\a,\rho\a\ran(s,y)\dvol_Y\mbox{ (By \eqref{D}, that is, $f_T'\geq0$)}\\
				&\leq \int_Y|D_T\rho\a|^2(s,y)+|\rho u|^2(s,y)\dvol_Y.\\
		\end{split}\end{align}
		 The last inequality follows from the  arithmetic-geometric mean inequality and \eqref{abu}.

		One can see that
		\begin{align}\begin{split}\label{mod12}
				h(s)&=\int_{-2}^sh'(s)ds\mbox{ (Since $h(-2)=0$)}\\&
				\leq \int_{\bm}|u|^2+|D_T\rho\a|^2\dvol_{\bm}\mbox{ (By (\ref{new4}))}\\&\leq\int_{\bm}|u|^2+|D_T \rho u|^2\dvol_{\bm}\mbox{ (By \eqref{addeq5})}\\
    &\leq C\int_{\bm}|u|^2+|D_T  u|^2\dvol_{\bm} \mbox{ (By Leibniz rule \eqref{addeq1})}.\end{split}\end{align}

		\def\htl{\tilde{h}}
		\def\ps{\frac{\p}{\p s}}
		For $s\in[-2,2]$, set \[\htl(s):=\int_Y|\rho\b|^2(s,y)\dvol_Y.\]
		
		Proceeding as before, one computes \begin{align}\begin{split}\label{new50}
				&\ \ \ \ \htl'(s)=2\Re\int_Y\lan\partial_s\rho\b,\rho\b\ran(s,y)\dvol_Y\\
				&=-2\Re\int_Y\lan D_Tds\wedge\rho\b ,\rho\b\ran(s,y)\dvol_Y+2\int_Y\ f_T'|\rho\b|^2(s,y)\dvol_Y\\
				&\geq -\int_Y|D_T\rho ds\wedge\b|^2(s,y)+|\rho u|^2(s,y)\dvol_Y.
		\end{split}\end{align}
		Repeating what we just did before, by (\ref{new50}), one can see that
		\be\label{mod-121}\htl(s)=-\int_{s}^2\htl'(s)ds\leq C \int_{\bm}|u|^2+|D_T  u|^2\dvol_{\bm}.\ee
	As a result, combining \eqref{mod12} and \eqref{mod-121}, we have \eqref{new10} for $u.$

		Repeating the arguments above and noticing that $\b_1(0,y)=0, \a_2(0,y)=0$, we have \eqref{new20} for $w_1$ and $w_2$. 
		
	\end{proof}

    \begin{cor}
       \label{cor-limit1}
	Assume that \( u \) is a unit eigenform of \( \Delta_T \) corresponding to the eigenvalue \( \lambda \). Consider the decomposition \eqref{candec} of \( u \), namely,  
\[ u = \alpha + ds \wedge \beta. \]  
Then, the following estimates hold for any $s\in[-7/4/7/4]$:  
\begin{equation} \label{cor-eq1}  
\int_Y |D_T \alpha|^2(s,y) \, \dvol_Y \leq C(1+\lambda^2),  
\end{equation}  
and  
\begin{equation} \label{cor-eq12}  
\int_Y |D_T (ds \wedge \beta)|^2(s,y) \, \dvol_Y \leq C(1+\lambda^2).  
\end{equation}  

Similarly, if \( w_i \) (\( i = 1,2 \)) is a unit eigenform of \( \Delta_{T,i} \) corresponding to the eigenvalue \( \lambda \), then analogous estimates hold for \( w_i \) (\( i = 1,2 \)).
    \end{cor}
    \begin{proof}
        Let \( \rho \in C_c^\infty(-2,2) \) be a smooth function satisfying \( \rho|_{(-7/4,7/4]} \equiv 1 \). We can regard \( \rho \) as a smooth function on \( \bm \). Then \( \rho \alpha \) is a smooth form on \( \bm \).  

Applying \eqref{new10} for $D_T\rho\a$, for any \( s \in [-7/4,7/4] \), we have  
\be\label{cor-eq5}  
\int_Y |D_T \rho \alpha|^2(s,y) \dvol_Y \leq C \left( \int_{[-2,2] \times Y} \left( |D_T^2 \rho \alpha|^2 + |D_T \rho \alpha|^2 \right) \dvol_{\bm} \right).
\ee  

We have  
\be\label{split-eigen1}  
D_T^2 (\rho u) = (\Delta \rho) u -2 \lan\nabla \rho, \nabla u\ran + \rho \Delta_T u.  
\ee  

Let \(\rho_1 \in C_c^\infty((-2, -13/8) \times Y \cup (13/8, 2) \times Y)\) be a function satisfying \(\rho_1 \equiv 1\) in a neighborhood of \(\supp(|\nabla \rho|)\). Then, applying the Lichnerowicz-type formula and Leibniz's rule, one can easily verify that  
\be\ba\label{eq420}
&\ \ \ \ \int_{\supp(\nabla \rho)}|\nabla u|^2\dvol_{\bm}\leq\int_{\bm} |\nabla (\rho_1 u)|^2 \, \dvol_{\bm} \leq \int_{\bm} |D (\rho_1 u)|^2+C|\rho_1u|^2 \, \dvol_{\bm}\\&=\int_{\bm} |D_T (\rho_1 u)|^2+C|\rho_1u|^2 \, \dvol_{\bm}\leq C'\int_{\bm} |D_T  u|^2+|u|^2 \, \dvol_{\bm}.
\ea\ee
Thus,
\be\label{split-eigen11}
\int_{\bm}|\nabla\rho|^2|\nabla u|^2\dvol\leq C\int_{\bm} |D_T  u|^2+|u|^2 \, \dvol_{\bm}.
\ee

Since \( \Delta_T = D_T^2 \) preserves the decomposition \eqref{candec}, we conclude that on \( [-2,2] \times Y \),  
\be\label{split-eigen}  
|\Delta_T\rho ds \wedge \beta|^2 + |\Delta_T \rho\alpha|^2 = |\Delta_T \rho u|^2.  
\ee  

We then compute 
\be\label{split-eigen3}  
\begin{aligned}  
&\ \ \ \  \int_{[-2,2] \times Y} |D_T^2 \rho \alpha|^2 + |D_T \rho \alpha|^2 \dvol_{[-2,2] \times Y}\\
&\leq  \int_{\bm} |D_T^2\rho u|^2 + |D_T\rho u|^2  \dvol_{\bm}\quad\text{(By \eqref{addeq5} and \eqref{split-eigen})}\\
&\leq C \int_{\bm} (|\Delta_Tu|^2 + |D_T u|^2 + |u|^2) \dvol_{\bm} \quad\text{( By \eqref{addeq1}, \eqref{split-eigen1} and \eqref{split-eigen11} )}\\&\leq C''(\lambda^2 + 1). 
\end{aligned}  
\ee

Thus, estimate \eqref{cor-eq1} follows directly from \eqref{cor-eq5} and \eqref{split-eigen3}. Similarly, estimate \eqref{cor-eq12} can be derived in the same manner.
    \end{proof}
	\fi
	\fi

	\def\mycmde{0}
	\ifx\mycmde\undefined
	undefed
	\else
	\if\mycmd1
	\begin{lem}\label{limadd}
		Assume $u,w_i,i=1,2$ meet the same conditions as in Lemma \ref{limit1}.
		Then 
		\be\label{new5}\int_{-1/2}^{1/2}\int_Y|u(s,y)|^2\dvol_Yds\leq \frac{C(\l+1)}{T^{2}},\ee
		\be\label{new6}\int_{-1/2}^{0}\int_Y|w_1(s,y)|^2\dvol_Yds\leq \frac{C(\l^2+1)}{T^{2}},\ee
		\be\label{new7}\int_{0}^{1/2}\int_Y|w_2(s,y)|^2\dvol_Yds\leq \frac{C(\l^2+1)}{T^{2}}.\ee
		
		Moreover, if $u,w_i,i=1,2$ are harmonic, then for any $l\in\Z^+$, there exists $T$-independent $C_l,$ such that
		\be\label{new51}\int_{-1/2}^{1/2}\int_Y|u(s,y)|^2\dvol_Yds\leq \frac{C_l}{T^{l}},\ee
		\be\label{new61}\int_{-1/2}^{0}\int_Y|w_1(s,y)|^2\dvol_Yds\leq \frac{C_l}{T^{l}},\ee
		\be\label{new71}\int_{0}^{1/2}\int_Y|w_2(s,y)|^2\dvol_Yds\leq \frac{C_l}{T^{l}}.\ee
		
	\end{lem}
	\begin{proof}\ \\
		$\bullet$ {\textit{Proof of \eqref{new5} and \eqref{new51}}.}\\
		Let $0\leq\eta\in C_c^\infty((-2,2))$, such that $\eta|_{[-\frac{1}{2},\frac{1}{2}]}\equiv1$, $\eta|_{[\frac{5}{8},2)\cup [-2,-\frac{5}{8}]}\equiv 0$, $|\nabla \eta|\leq 64$. We can regard $\eta$ as a smooth function on $\bm.$
		
		Integrate by parts, 
		\begin{align}\begin{split}\label{new121}
				&\ \ \ \ \l\geq \int_{\bm}\lan\Delta_Tu,\eta^2u\ran\dvol_{\bm}=\int_{\bm}\lan Du,D\eta^2u\ran+\lan L_Tu,\eta^2u\ran+|f_T'|^2\eta^2|u|^2\dvol_{\bm}\\
				&\geq\int_{\bm}\eta^2\lan Du,Du\ran-|\eta'\eta||\lan Du,u\ran|+\lan L_Tu,\eta^2u\ran+|f_T'|^2\eta^2|u|^2\dvol_{\bm}\\
				&\geq\int_{\bm}\eta^2\lan Du,Du\ran/2-4|\eta'|^2|u|^2+C'T^2\eta^2|u|^2\dvol_{\bm}\\
		\end{split} \end{align}
		(\ref{new121}) implies that
		\be\label{new131}
		\int_{-1/2}^{1/2}\int_Y|u|^2\dvol_Yds\leq \frac{C(\l+1)}{T^2}.\ee
		
		Moreover, if $u$ is harmonic, (\ref{new121}) implies that
		\be\label{new132}
		\int_{-1/2}^{1/2}\int_Y|u|^2\dvol_Yds\leq \frac{C}{T^2} \int_{-5/8}^{5/8}\int_Y|u|^2\dvol_Yds.
		\ee
		Repeating the arguments above, one can see that
		\be\label{new133}
		\int_{-5/8}^{5/8}\int_Y|u|^2\dvol_Yds\leq \frac{C'}{T^2}.
		\ee
		(\ref{new132}) and (\ref{new133}) then imply that
		\be\label{new1320}
		\int_{-1/2}^{1/2}\int_Y|u|^2\dvol_Yds\leq \frac{C_4}{T^4}.
		\ee
		Repeating the arguments above, we have (\ref{new51}) for any $l\in\Z^+.$\\
		$\bullet$ {\textit{Proof of \eqref{new6} and \eqref{new61}}.}\\ 
		Let $\eta_1$ be a nonnegative bounded smooth function on $\bm_1$, such that $\eta_1|_{[-\frac{1}{2},0]\times Y}\equiv1$, $\eta_1|_{\bm_1-[-3/4,0]\times Y}\equiv 0$, $|\nabla \eta_1|\leq 64$. Let $w_1=\a+\b ds$, then
		by Lemma \ref{limit1} and integration by parts,
		\begin{flalign}\begin{split}\label{new12}
				&\ \ \ \ \l\geq \int_{\bm_1}\lan\Delta_Tw_1,\eta_1^2w_1\ran\dvol_{\bm_1}\\
				&=\int_{\bm_1}\lan Dw_1,D\eta_1^2w_1\ran+\lan L_Tw_1,\eta^2_1w_1\ran+|f_T'|^2\eta_1^2|w_1|^2\dvol_{\bm_1}\\
				&-\int_Y \lan ds\wedge w_1, dw_1\ran(0,y)\dvol_Y\\
				&=\int_{\bm_1}\lan Dw_1,D\eta_1^2w_1\ran+\lan L_Tw_1,\eta^2_1w_1\ran+|f_T'|^2\eta_1^2|w_1|^2\dvol_{\bm_1}\\
				&-\int_Y \lan ds\wedge w_1,d_Tw_1\ran(0,y)\dvol_Y+\int_Y\partial_s f_T\lan\a,\a\ran(0,y)\dvol_Y\\
				&\geq \int_{\bm_1}\lan Dw_1,D\eta_1^2w_1\ran+\lan L_Tw_1,\eta^2_1w_1\ran+|f_T'|^2\eta_1^2|w_1|^2\dvol_{\bm_1}-C(\l^2+1)\\
				&\geq\int_{\bm_1}\eta^2\lan Dw_1,Dw_1\ran/2-4|\eta_1'|^2|w_1|^2+C'T^2\eta_1^2|w_1|^2\dvol_{\bm_1}-C(\l^2+1)\\
		\end{split} \end{flalign}
		(\ref{new12}) implies that
		\be\label{new13}
		\int_{-1/2}^{0}\int_Y|w_1|^2\dvol_Yds\leq \frac{C(\l^2+1)}{T^2}.
		\ee
		Similarly, when $w_1$ is harmonic, since $d_Tw_1=0$, (\ref{new12}) also implies (\ref{new61}).
		
	\end{proof}
	\fi\fi
	
	\def\mcmd{0}
	
	\ifx\mcmd\undefined
	undefed
	\else
	\if\mcmd1
	\begin{lem}\label{limit2}
		Assume $u,w_i,i=1,2$ meet the same conditions as in Lemma \ref{limit1}. Moreover, if $u|_{[-2,2]}=v_{1}(s,y)+v_{2}(s,y)ds$, define $P_a(u)(y)=v_2(-1,y)$, $P_b(u)(y)=v_1(1,y)$. 
		\be\label{lemeq1}\int_Y|P_a(u)|^2\dvol_Y\leq \frac{C(\l+1)}{\sqrt{T}},\ee
		and
		\be\label{lemeq2}\int_Y|P_b(u)|^2\dvol_Y\leq \frac{C(\l+1)}{\sqrt{T}}.\ee
		
		Moreover, 
		\be\label{lemeq3}\int_Y|P_a(du)|^2\dvol_Y\leq \frac{C(\mu+1)^2}{\sqrt{T}}\leq\frac{C(\l+1)^2}{\sqrt{T}},\ee
		and
		\be\label{lemeq4}\int_Y|P_b(d^*u)|^2\dvol_Y\leq \frac{C(\mu+1)^2}{\sqrt{T}}\leq\frac{C(\l+1)^2}{\sqrt{T}}.\ee
		Similar statements holds for $w_i,i=1,2.$
	\end{lem}
	\begin{proof}
		Let $\e(T):=\int_Y|v_{2}(-1,y)|^2\dvol_Y$. Notice that on $\Omega(Y;F|_{Y})du$, $\Delta_{T}=-\frac{\partial^2}{\partial s^2}+\Delta^Y+|\frac{\partial}{\partial s}f_{T}|^2+\frac{\partial^2}{\partial s^2}f_T$.  By repeating previous steps,
		
		\begin{align}\begin{split}\label{v1}
				\inf_{s'\in[-1,-1+\sqt]}\int_Y|v_2(s',y)|^2\dvol_Y\geq 
				c\e(T)-\frac{C(\l+1)}{\sqrt{T}}.\\
		\end{split}\end{align}
		
		Let $\eta\in C_c^\infty[-2,2)$ be a bump function, such that $\eta|_{[-2,-1/2]}\equiv1$, $\eta_{[-1/4,2)}\equiv0$.
		
		Note that on $\qty[-1+\sqt,1-\sqt]$, $|f_T'|^2-|f_T''|\geq T$; on $[-1+e^{-T^2},-3/4]$, $f_T''=T\geq0$ , $|f_T'(-1+e^{-T^2})|\leq 1$, by (\ref{v1}) and integration by parts,
		\begin{align}
			\begin{split}\label{v2}
				& c\sqrt{T}\e(T)-C(\l+1)\leq\int_{-1+e^{-T^2}}^{-1+\sqt}\int_YT|v_2(s,y)|^2\dvol_Yds\\
				&\leq \int_{-1+e^{-T^2}}^{-1+\sqt}\int_Y|dv_2(s,y)|^2+|d^*v_2(s,y)|^2+f_T''|v_2|^2+|f_T'|^2u^2\dvol_Yds\\
				&\leq \int_{-1+e^{-T^2}}^{2}\int_Y|d\eta v_2(s,y)|^2+|d^*\eta v_2(s,y)|^2+f_T''|\eta v_2|^2+|\nabla  f_T|^2\eta^2u^2\dvol_Yds\\
				&\leq \int_{-1+e^{-T^2}}^{2}\int_Y|d_T\eta v_2(s,y)|^2+|d_T^*\eta v_2(s,y)|^2\dvol_Yds\\
				&+ \int_Y|df_T||v_2|^2(s,y)\dvol_Y\Big|_{s=-1+e^{-T^2}}\\
				&\leq\int_{\bm}|d_T^*u|^2+|d_T u|^2+|\eta'||\partial_s v_2||v_2|\dvol_{\bm}+C(\l+1).
			\end{split}
		\end{align}
		Moreover,
		\begin{align}\begin{split}\label{v22}
				&\ \ \ \ \int_{\bm}|\eta'||\partial_s v_2||v_2|\dvol_{\bm}\leq \int_{\bm}|\eta'||D_T u||u|+|\eta'||df_T u||u|\dvol_{\bm}\\
				&\leq \int_{\bm}|D_T u|^2+|u|^2+CT|\eta'||u|^2\dvol_{\bm}\leq C(1+\l)+C\int_{\bm}T|\eta'||u|^2\\
				&\leq C(\l+1) \mbox{ (By Lemma \ref{limadd})}.
		\end{split}\end{align}

		According to (\ref{v2}) and (\ref{v22}), $\e(T)\leq \frac{C(\lambda+1)}{\sqrt{T}}.$ Similarly, one has (\ref{lemeq2}).
		
		Replace $u$ with $d_T u$ and notice that $d_Tu=du$ on $\{-1\}\times Y$, (\ref{lemeq3}) follows. Similarly, one has (\ref{lemeq4}).
		
	\end{proof}
	\else

	The following lemma suggests that the eigenforms of \( \Delta_T \) converge to differential forms satisfying the desired boundary conditions. To understand this philosophically, in the neighborhood of \( \{-1\} \times Y \), the behavior of \( \Delta_T \) on the \( ds \wedge \beta \) component is governed by the operator \( -\partial_s^2 + T + T^2(s+1)^2 \), which satisfies \( -\partial_s^2 + T + T^2(s+1)^2 \geq T \). As a result, if \( u \) is a unit eigenform corresponding to the \( k \)-th eigenvalue of \( \Delta_T \), then \( i_{\frac{\partial}{\partial s}} u(-1, y) = \beta(-1, y) \) is expected to be small. Otherwise, by the Rayleigh quotient argument, \( \lambda_k(T) \) would grow significantly as \( T \to \infty \), contradicting \Cref{o1}.
	\begin{lem}\label{limit2}
		Assume $u\in \Omega(\bm,\F)$ and consider the decomposition \eqref{candec} for $u$, that is,
        $u=\a+ds\wedge\b$. Then we have
		\be\label{lemeq1}\int_Y|\b(-1,y)|^2\dvol_Y\leq \frac{C\int_{[-2,2]\times Y}|\Delta_T u|^2+|D_T u|^2+|u|^2\dvol_{\bm}}{\sqrt{T}},\ee
		and
		\be\label{lemeq2}\int_Y|\a(1,y)|^2\dvol_Y\leq \frac{C\int_{[-2,2]\times Y}|\Delta_T u|^2+|D_T u|^2+|u|^2\dvol_{\bm}}{\sqrt{T}}.\ee
		
	Assume $w_i,i=1,2\in\Omega_\bd(\bm_i,\F_i)(T)$. Then similar statements holds for $w_i,i=1,2.$
	\end{lem}
	\begin{proof}

Let $\rho\in C^\infty(-2,2)$, such that $\rho|_{[-7/4,7/4]}\equiv1$. Then we can regard $\rho$ as a smooth function on $\bm$ with support inside $(-2,2)\times Y$.

Then, for any \( s \in [-7/4, 7/4] \), applying \eqref{new10} to \( D_T \rho \, ds \wedge \beta \) and using arguments similar to those in the proof of \eqref{split-eigen3}, we obtain
\be\ba\label{DTa}
&\ \ \ \ \int_Y|D_T\rho ds\wedge\b|^2\dvol_Y\leq C\int_{[-2,2]\times Y}|\Delta_T u|^2+|D_T u|^2+|u|^2\dvol_{\bm}.
\ea
\ee
		Using the same notation as in the proof of Lemma \ref{limit1}, the estimates \eqref{new50}, \eqref{DTa}, and Lemma \ref{limit1} imply that  \def\htl{\tilde{h}}
		\begin{align}\begin{split}\label{v1}
				\inf_{s\in[-1,-1+\sqt]}\htl(s)\geq 
				\htl(-1)-\frac{C\int_{[-2,2]\times Y}|\Delta_T u|^2+|D_T u|^2+|u|^2\dvol_{\bm} }{\sqrt{T}}.
		\end{split}\end{align}
		Let $\eta\in C_c^\infty(-2,0)\subset C^\infty(\R)$, such that $\eta|_{[-3/2,-1/2]}\equiv1$. Then $\eta$ could be regarded as a smooth function on $\bm$ with support inside $[-2,0]\times Y.$

		We compute (refer to the explanations of inequalities provided below),
		\begin{align}
			\begin{split}\label{v2}
				& C\left(\sqrt{T}\htl(-1)-\int_{[-2,2]\times Y}|\Delta_T u|^2+|D_T u|^2+|u|^2\dvol_{\bm}\right) \leq\int_{-1+e^{-T^2}}^{-1+\sqt}\int_YT|\b(s,y)|^2\dvol_Yds\\
				&\leq \int_{-1+e^{-T^2}}^{-1+\sqt}\int_Y|D\b(s,y)\wedge ds|^2+f_T''|\b\wedge ds|^2+|f_T'|^2|\beta\wedge ds|^2\dvol_Yds\\
				&\leq \int_{-1+e^{-T^2}}^{0}\int_Y|D\eta \b(s,y)\wedge ds|^2+f_T''|\eta \b\wedge ds|^2+|f'_T|^2\eta^2|\beta\wedge ds|^2\dvol_Yds\\
				\end{split}
                \end{align}
Here the first inequality follows from \eqref{v1}. For the inequality in the second line, by Condition \eqref{C}, on $[-1+e^{-T^2},-1/2]$, $f_T''=T$, and the other terms are positive. The inequality in the third line follows from the construction of $\eta$ and by Condition \eqref{C} and \eqref{c}, on $\qty[-1+\sqt,0]$, $f_T''+|f_T'|^2\geq0$ if $T$ is large. 

              Applying integration by parts and noting that \( |f_T'(-1+e^{-T^2})| \leq 1 \) by Condition \eqref{C},
                \begin{align}\begin{split}\label{eq34111}
                &\ \ \ \ \int_{-1+e^{-T^2}}^{0}\int_Y|D\eta \b(s,y)\wedge ds|^2+f_T''|\eta \b\wedge ds|^2+|f'_T|^2\eta^2|\beta\wedge ds|^2\dvol_Yds
                \\&\leq \int_{-1+e^{-T^2}}^{0}\int_Y|D_T\eta \b(s,y)\wedge ds|^2\dvol_Yds+ \int_Y|f'_T||\b|^2(-1+e^{-T^2},y)\dvol_Y\\
				&\leq \int_{\bm}|D_T\eta \b\wedge ds|^2\dvol_{\bm}+C\int_Y|\b|^2(-1+e^{-T^2},y)\dvol_Y\\
				&\leq \int_{\bm}|D_T\eta u|^2\dvol_{\bm}+C\int_Y|\b|^2(-1+e^{-T^2},y)\dvol_Y\mbox{ (By \eqref{addeq5})}
			\end{split}
		\end{align}


		According to (\ref{v2}), \eqref{eq34111}, Leibniz rule \eqref{addeq1} and Lemma \ref{limit1},
		$$\htl(-1)\leq \frac{C\int_{[-2,2]\times Y}|\Delta_T u|^2+|D_T u|^2+|u|^2\dvol_{\bm}}{\sqrt{T}}.$$

	\end{proof}
	\fi
	\fi
	
	Note that $T^2a^2 - T \geq T$ if $|a| \geq \frac{\sqrt{2}}{\sqrt{T}}$. One can see $\Delta_T \geq T$ on the tube $[-1,1]\times Y$, except when $||s| - 1| \leq \frac{\sqrt{2}}{\sqrt{T}}$, a very small region for large $T$. Thus, we expect that if \( u \) is an eigenform of \( \Delta_T \), its \( L^2 \)-norm on the tube \( [-1,1] \times Y \) should be small. This expectation is supported by the following lemma.
	
	\begin{lem}\label{limit0}
		Suppose $u,w_i,i=1,2$ meet the same conditions as in Lemma \ref{limit1}. Then 
		\[\int_{-1}^1\int_Y|u|^2(s,y)\dvol_Yds\leq \frac{C\int_{\bm}|D_T u|^2+|u|^2\dvol_{\bm}}{\sqrt{T}}.\]
		Similar statements hold for $w_i,i=1,2.$
	\end{lem}
	\begin{proof}
		By Lemma \ref{limit1}, one can show that
		\be\label{smallint}\int_{||s|-1|\leq \sqt}\int_Y|u|^2(s,y)\dvol_Yds\leq \frac{C\int_{\bm}|D_T u|^2+|u|^2\dvol_{\bm}}{\sqrt{T}}.\ee
		We compute (refer to the explanations of inequalities provided below for each line),
		\begin{align*}
			&\ \ \ \ \int_{-1+\sqt}^{1-\sqt}\int_Y {T}u^2(s,y)\dvol_Yds\leq C\int_{-1+\sqrt{\frac{2}{{T}}}}^{1-\sqt}\int_Y |D u|^2+| f'_{T}|^2u^2+\lan L_Tu,u\ran\dvol_Yds\\
			&\leq C\int_{-1+\sqrt{\frac{2}{{T}}}}^{1-\sqt}\int_Y |D_Tu|^2\dvol_Yds+C'\int_Y|f_T'||u|^2\dvol_Y\Big|_{s=\pm(-1+\sqrt{\frac{2}{T}})}\\
			&\leq C\int_{\bm} |D_Tu|^2\dvol_{\bm}+C'\sqrt{T}\int_Y|u|^2\dvol_Y\Big|_{s=\pm(-1+\sqrt{\frac{2}{T}})}\\
			&\leq C''\sqrt{T}\int_{\bm}|D_T u|^2+|u|^2\dvol_{\bm}\mbox{ (By Lemma \ref{limit1})}.\end{align*}
		Here the first inequality  follows from the fact that on $\qty[-1+\sqt,1-\sqt]$, $-|f_T''|+|f_T'|^2\geq CT$ if $T$ is large (By Condition \eqref{C} and \eqref{c}), and $|Du|^2$ is obviously nonnegative. The second inequality follows from integration by parts. The third inequality follows from the fact that $|f_T'|\qty(\pm\big(1- \sqt\big))=\sqrt{2}\sqrt{T}$, and $\sqrt{2}$ is absorbed by $C$.
		
		Thus,
		\be\label{bigint}\int_{-1+\sqrt\frac{2}{{T}}}^{1-\sqt} \int_Y|u(s,y)|^2\dvol_Ydu\leq\frac{C\int_{\bm}|D_T u|^2+|u|^2\dvol_{\bm}}{\sqrt{T}}.\ee
The lemma for $u$ now follows from \eqref{smallint} and \eqref{bigint}.
		
		Now we prove the lemma for $w_1$. Assume that $w_1=\a+ds\wedge\b$ on the tube. Let $\eta\in C^\infty(-\infty,0]$, such that $0\leq \eta\leq 1$, $\eta|_{(-\infty,-1/2]}\equiv1$, $\eta(s)=0$ if $s\in[-1/4,0].$ We can regard $\eta$ as a smooth function on $\bm_1.$
		
		Proceeding as before, one can show that \be\label{addeq11}\int_{-1}^0\int_Y|\eta w_1|^2\dvol_Y ds\leq \frac{C\int_{\bm_1}|D_T w_1|^2+|w_1|^2\dvol_{\bm}}{\sqrt{T}}.\ee
		Next, we would like to estimate $(1-\eta) w_1$. Note that
		\begin{align}
			\begin{split}\label{mod21}
				&\ \ \ \ C\int_{\bm_1}|D_T w_1|^2+|w_1|^2\dvol_{\bm_1}\geq \int_{\bm_1}\lan D_Tw_1,D_T(1-\eta)^2w_1\ran\dvol_{\bm_1}\\&=\int_{\bm_1}\lan Dw_1,D(1-\eta)^2w_1\ran+\lan L_{T}w_1,(1-\eta)^2w_1\ran+|f_T'|^2(1-\eta)^2|w_1|^2\dvol_{\bm_1} \quad\text{(By \eqref{addeq1})}\\
				&-\int_Y \lan w_1, i_{\frac{\p}{\p s}}dw_1\ran(0,y)\dvol_Y \quad\text{(Integration by parts)}.
		\end{split}\end{align}
		Note that $ i_{\frac{\p}{\p s}}d_Tw_1(0,y)=0$ and $f_T'\geq 0$
		\be\label{mod211}
		\int_Y \lan  w_1,i_{\frac{\p}{\p s}}dw_1\ran(0,y)\dvol_Y=	\int_Y \lan  w_1,i_{\frac{\p}{\p s}}d_Tw_1\ran(0,y)\dvol_Y-\int_Yf_T'|\a|^2(0,y)\dvol_Y\leq0.
		\ee
		While by \eqref{addeq1},
		\begin{align}\begin{split}\label{mod212}&\ \ \ \ \int_{\bm_1}\lan Dw_1,D(1-\eta)^2w_1\ran+\lan L_{T}w_1,(1-\eta)^2w_1\ran+|f_T'|^2(1-\eta)^2|w_1|^2\dvol_{\bm_1}\\
				&\geq\int_{\bm_1}(1-\eta)^2|Dw_1|^2-2|\eta'||1-\eta|| Dw_1||w_1|+\lan L_{T}w_1,(1-\eta)^2w_1\ran\\
				&+|f_T'|^2(1-\eta)^2|w_1|^2\dvol_{\bm}\geq\int_{\bm_1}-|\eta'|^2|w_1|^2+C'T^2(1-\eta)^2|w_1|^2\dvol_{\bm_1}.
			\end{split}
		\end{align}
		The last inequality follows from the fact that for $T$ being large, $|f_T'|^2-|f_T''|\geq C'T^2$ when $s\in [-1/2,0]$ and arithmetic-geometric mean inequality.

		It follows from \eqref{mod21}, \eqref{mod211} and \eqref{mod212}  that
		\be\label{addeq12}\int_{-1}^0\int_Y|(1-\eta) w_1|^2\dvol_Y ds=\int_{\bm}|(1-\eta) w_1|^2\dvol_{\bm} \leq \frac{C\int_{\bm_1}|D_Tw_1|^2+|w_1|^2\dvol_{\bm_1}}{{T^2}}.\ee
		The lemma for $w_1$ then follows from \eqref{addeq11} and \eqref{addeq12}. Similarly, we have the lemma for $w_2.$
	\end{proof}
	
	\subsection{Proof of Theorem \ref{eigencon}}\label{proof-eigencon}
 The following proof is based on a standard Rayleigh quotient  argument, using the established estimates in \Cref{o1}–\Cref{limit0}.
		In the proof, we will identify $u\in \Omega(M_1\cup M_2,\F|_{M_1\cup M_2})$ with $(u|_{M_1},u|_{M_2})\in\Omega(M_1,F_1)\oplus\Omega(M_2,F_2)$. \ \\
		$\bullet$ We will show that $\lim\sup_{T\to\infty}\lambda_{k}(T)\leq \lambda_k.$\\
		Let $u_j=(u_{j,1},u_{j,2})$ be the $j$-th eigenvalue of $\Delta_1\oplus\Delta_2$ on $\Omega_{\rel}(M_1;F_1)\oplus \Omega_{\abs}(M_2;F_2)$ ($1\leq j\leq k$). 
		
		To use the Rayleigh quotient argument, we would like to find a nice extension of $u$ to a differential form on $\Omega(\bm,\F).$ Let $\eta\in C^\infty([0,1])$, such that $\eta|_{[0,1/4]}\equiv0$, $\eta|_{[1/2,1]}\equiv1.$
		For any $u=(v_1,v_2)\in \Omega_{\abs}(M_1;F_1)\oplus \Omega_{\rel}(M_2;F_2)$, let $Q_T:\Omega_{\abs}(M_1;F_1)\oplus\Omega_{\rel}(M_2,F_2)\to W^{1,2}\Omega(\bm;\F)$, such that,
		\begin{equation*}
			Q_T({u})(x)=\begin{cases}
				v_i(x), \mbox{ if  $x\in M_i$;}\\
				\eta(-s)v_1(-1,y)e^{-f_T(s)-T/2}, \mbox{if $x=(s,y)\in [-1,0]\times Y$;}\\
				\eta(s)v_2(1,y)e^{f_T(s)-T/2}, \mbox{if $x=(s,y)\in [0,1]\times Y$.}
			\end{cases}
		\end{equation*}
		Suppose $v_i=\a_i+ds\wedge\b_i(i=1,2)$ on the tube $[-2,2]\times Y$.
		
	 One can see easily that although $Q_T(u)$ is not smooth, but its weak derivative exists:
		\begin{equation}\label{qtdt}D_TQ_T(u)(x)=\begin{cases}D_Tv_i=Dv_i, \mbox{ if }x\in M_i;\\
				D^Y\a_1(-1,y)e^{-f_T-T/2},\mbox{ if } x=(s,y)\in[-1,-1/2]\times Y; \\
				D^Yds\wedge\b_2(1,y) e^{f_T-T/2}, \mbox{ if }x=(s,y)\in[1/2,1]\times Y;\\
				O(e^{-cT})(\a_1(-1,y)+D^Y\a_1(-1,y)), \mbox{ if }x=(s,y)\in[-1/2,0];\\
				O(e^{-cT})(ds\wedge\b_2(1,y)+D^Yds\wedge\b_2(1,y)), \mbox{ if }x=(s,y)\in[0,1/2].
		\end{cases}\end{equation}
		Here the first line is obvious. The second and the third line follow from the boundary conditions and the fact that $(\frac{\p}{\p s}+f_T') e^{-f_T-T/2}=0,(-\frac{\p}{\p s}+f_T') e^{f_T-T/2}ds=0$. The last two lines follow from \eqref{b}, \eqref{c} and \eqref{D} and the fact that $T^2$ could be absorbed by $e^{-aT},\forall a>0$.
		
		Let $\bar{u}=Q_T(u)$, then one can see that $\dim span\{\bar{u}_j\}_{j=1}^k=k$. Moreover, by trace formula \cite[\S 5.5]{evans2022partial} and G\aa rding's inequality, one can show that for any $u=(v_1,v_2)\in span\{u_j\}_{j=1}^k,$
		\begin{equation}\label{trace}\int_Y|v_i((-1)^{i},y)|^2+|D^Y v_i((-1)^i,y)|^2\dvol_Y\leq C(1+\lambda_k^2)\int_{M_i}|v_i|^2\dvol_{\bm}.\end{equation}
		
		Also,
		\begin{align*}
			&\ \ \ \ \int_{\bm} |D_T \bar{u}|^2\dvol_{\bm}\\
			&=\int_{M_1\cup M_2}|D u|^2\dvol+\int_{-1}^0\int_Y |D_T \bar{u}|^2\dvol_Yds+\int_{0}^1\int_Y |D_T \bar{u}|^2\dvol_Yds\\&=I+II+III.
		\end{align*}
		First, notice that $I\leq\lambda_k\int_{M_1\cup M_2}|u|^2\dvol\leq \lambda_k\int_{\bm}|\bar{u}|^2\dvol .$\\
		By \eqref{qtdt} and \eqref{trace},
		\begin{align}\begin{split}\label{ineqii}
				&\ \ \ \ II\leq\int_{-1}^{-1/2}\int_Y|D^Y \a_1|^2(-1,y)e^{-2f_T(s)-T}\dvol_Yds\\
                &+C\int_{-1/2}^{0}\int_Ye^{-cT}(|\a_1|^2+|D^Y\a_1|^2)(-1,y)\dvol_Yds\\
				&\leq \frac{C_2(1+\lambda_k^2)}{\sqrt{T}}\int_{M_1\cup M_2}|\bar{u}|^2\dvol_{\bm}\leq \frac{C_2(1+\lambda_k^2)}{\sqrt{T}}\int_{\bm}|\bar{u}|^2\dvol_{\bm}.
			\end{split}
		\end{align}
		Here $\sqrt{T}$ in the denominator comes from the Gauss-type integral.
		
		Similarly, $III\leq \frac{C_2(1+\lambda_k^2)}{\sqrt{T}}\int_{\bm}|\bar{u}|^2\dvol_{\bm}.$
		By Rayleigh quotient, $\l_{k}(T)\leq \l_k+\frac{2C_2(1+\lambda_k^2)}{\sqrt{T}}$. Consequently, as $T\to \infty$, $\lim\sup_{T\to\infty}\lambda_{k}(T)\leq \lambda_k.$\ \\ \ \\
		$\bullet$ We aim to prove that \( \lim\inf_{T\to\infty} \lambda_k(T) \geq \lambda_k \). To achieve this, we need to construct a suitable ``restriction" mapping from \( \Omega(\bm, \F) \) to \( \Omega(M_1, F_1) \oplus \Omega(M_2, F_2) \).
        
Let \( \rho \in C_c^\infty((-2,2)) \) be a smooth cutoff function satisfying  
\[
\rho|_{(-2,-1-2T^{-1/4})\cup(1+2T^{-1/4},2]} \equiv 0, \quad \text{and} \quad \rho|_{[-1-T^{-1/4},1+T^{-1/4}]} \equiv 1.
\]  
For \( u \in \Omega(\bm, \F) \), consider its decomposition \eqref{candec},  
\[
u = \alpha + ds \wedge \beta.
\]  
Consider the operator \( R_T: \Omega(\bm, \F) \to \mathcal{A}_{\mathrm{abs}}(M_1, F_1) \oplus \mathcal{A}_{\mathrm{rel}}(M_2, F_2) \)  given by  
\[
R_T(u)(x) =
\begin{cases}
    u(x), & \text{if } x \in \bm - [-2,2] \times Y, \\
    u(s,y) - \rho(s) ds \wedge \beta(-1,y), & \text{if } (s,y) \in [-2,-1] \times Y, \\
    u(s,y) - \rho(s) \alpha(1,y), & \text{if } (s,y) \in [1,2] \times Y.
\end{cases}
\]  
From the estimate  
\[
\int_Y |D^Y ds \wedge \beta|^2 (-1,y) \, \dvol_Y \leq \int_Y |D ds \wedge \beta|^2 (-1,y) \, \dvol_Y = \int_Y |D_T ds \wedge \beta|^2 (-1,y) \, \dvol_Y,
\]  
and using \eqref{lemeq1}, \eqref{cor-eq12}, and the properties of \( \rho \), we obtain that if \( u \) is a unit eigenform of \( \Delta_T \) with eigenvalue \( \lambda \), then  
\begin{equation} \label{prove-thm-eq2}
\int_{M_1} \Big|D\big( \rho(s) ds \wedge \beta(-1,y)\big)\Big|^2 \, d\operatorname{vol}_{M_1} \leq \frac{C(1+\lambda^2)}{T^{1/4}}.
\end{equation}  
Similarly, \begin{equation} \label{prove-thm-eq3}
\int_{M_1} \Big|D \big(\rho(s) \a(1,y)\big)\Big|^2 \, d\operatorname{vol}_{M_1} \leq \frac{C(1+\lambda^2)}{T^{1/4}}.
\end{equation}  

Let \( u_j(T) \) denote the \( j \)-th eigenform of \( \Delta_T \). Applying \eqref{prove-thm-eq2}, \eqref{prove-thm-eq3}, \Cref{o1}, \Cref{limit0}, and the Rayleigh quotient argument, we conclude  
\[
\lambda_k(T) + C(\Lambda_k^2+1)T^{-1/4} \geq \sup_{v \in \mathrm{span} \{ R_T(u_j(T)) \}_{j=1}^k} \frac{\int_{M_1\cup M_2} |Dv|^2 \, d\operatorname{vol}_{\bm}}{\int_{M_1\cup M_2} |v|^2 \, d\operatorname{vol}_{\bm}} \geq \lambda_k.
\]  
Thus, we obtain  
\[
\liminf_{T\to\infty} \lambda_k(T) \geq \lambda_k.
\]
	 \ \\
		$\bullet$ Similarly, one can show that $\lim_{T\to\infty}\tl_{k}(T)=\l_k.$

	\section{Large Time Contributions}\label{conl}
	
	\def\H{\mathcal{H}}
	\def\lan{\langle}
	\def\ran{\rangle}
	In this section, we prove Theorem \ref{larcon}. The key idea is to use the Rayleigh quotient argument to establish a uniform lower bound for \(\l_k(T)\), which in turn leads to a uniform upper bound for \(|\Tr_s(N\exp(-t\Delta_T')(1-\P^\delta(T)))|\). Finally, applying the dominated convergence theorem completes the proof of Theorem \ref{larcon}.
	
	Let $M_3=[-1,0]\times Y$, $M_4=[0,1]\times Y$, then $\bm=M_1\cup M_2\cup M_3\cup M_4$.
	Let $\Delta_{i,T,\bd}$ be the restriction of $\Delta_T$ on $M_i$ with boundary condition $\bd$(i=1,2,3,4).
	For $\bd=\rel,\abs,D$ and $N$,  we mean relative, absolute, Dirichlet, and Neumann boundary conditions respectively.
	Let $\l_{k,bd}(T)$ be the $k$-the eigenvalues of $\Delta_{1,T,\bd}\oplus\Delta_{2,T,\bd}\oplus\Delta_{3,T,\bd}\oplus\Delta_{4,T,\bd}$ (acting on $\Omega_{\bd}^*(M_1;F_1)\oplus\Omega_{\bd}^*(M_2;F_2)\oplus\Omega_{\bd}^*(M_3;\F|_{M_3})\oplus\Omega_{\bd}^*(M_4;\F|_{M_4})$), it follows from the  Rayleigh quotient argument that (see for example, \cite[Lemma 3.1]{tachizawa1992eigenvalue})
	\be\label{ineq}\l_{k,D}(T)\geq\l_{k}(T)\geq \l_{k,N}(T).\ee

Using an inequality similar to \eqref{mod211}, we observe that if \( w_i \in \Omega_{\bd}(\bm_i, \F_i)(T) \), 
\[
\int_{\bm_i} |D_T w_i|^2 \, \dvol_{\bm_i} \geq \int_{\bm_i} |D w_i|^2 \, \dvol_{\bm_i} + \langle L_T w_i, w_i \rangle + |f_T'|^2 |w_i|^2 \, \dvol_{\bm_i}.
\]  
Then by the Rayleigh quotient argument, one still has 
\be\label{ineq1}
\tilde{\lambda}_k(T) \geq \lambda_{k,N}(T).
\ee
	Before moving on, let's study the following one-dimensional model problem first. Let $\Delta^{\R}_{T,N,\pm}:=-\partial_s^2+|f_T'|^2\pm f_T''$ on $[-1,0]$ with Neumann boundary conditions, and $\l^{\R}_{k,T,N,\pm}$ be the $k$-th eigenvalues of $\Delta^{\R}_{T,N,\pm}.$
	
	
	\begin{lem}\label{51}
		For $k\geq 2, T\geq1$, one has $\l^{\R}_{k,T,N,\pm}\geq v_k(T)$. Here $\{v_k(T)\}_{k=1}^\infty$ is the collection of $\Big\{T\max\{c_1l^2-c_2,0\}\Big\}_{l=0}^\infty\bigcup \Big\{c_3l^2\Big\}_{l=0}^\infty$, listed in increasing order and counted with multiplicity, and constants $c_1$, $c_2$, and $c_3$ are independent of $T.$  
	\end{lem}
	\begin{proof}
		Let $I_1:=[-1,-1+\frac{1}{\sqrt{T}}]$, $I_2:=[-1+\frac{1}{\sqrt{T}},0]$. It follows from the constructions of $f_T$ that when $T$ is large enough $\Delta^\R_{T,N,\pm}\geq-\p_{s,N}^2$ on $I_2$. That is, for all $\phi\in C^\infty(I_2)$ with $\phi'(-1+\frac{1}{\sqrt{T}})=\phi'(0)=0$,
		\be\label{est11}\int_{I_2}\Delta^\R_{T,N,\pm}\phi\cdot\phi ds\geq \int_{I_2}-\partial_s^2\phi\cdot \phi ds.\ee
		
		For $I_1$, changing the variable $\sqrt{T}(s+1)\to \ts$, and suppose that under new coordinates, ${\Delta^\R_{T,N,\pm}}$ is given by $\tilde{\Delta}^{\R}_{T,N,\pm}.$
Then direct computations yield, on $\tilde{s}\in[0,1]$,
		
		\be\label{est12}\tilde{\Delta}^{\R}_{T,N,\pm}\geq T(-\p_{\ts}^2+\ts^2-C)\geq T(-\p_{\ts}^2-C).\ee 
		
		The lemma then follows from (\ref{est11}), (\ref{est12}) and \eqref{ineq}.
	\end{proof}
	It follows from (\ref{ineq}), \eqref{ineq1} and Weyl's law that
	\begin{lem}\label{est1} If $T\geq1$,
		$\l_k(T)\geq u_k(T)\geq u_k(1)$ and $\tilde{\l}_{k}(T)\geq u_k(T)\geq u_k(1)$.
		Here $\{u_k(T)\}_{k=1}^\infty$ is the collection of $4$ copies of $\{v_l(T)+c_4m^{2/(\dim M-1)}\}_{l=0,m=0}^\infty$ and $2$ copies of $\{c_5 l^{2/\dim M}\}_{l=0}^
		\infty$, listed in increasing order and counted with multiplicities.
		Moreover, constants $c_4$ and $c_5$ are independent of $T.$
	\end{lem}
	\begin{rem}\label{rem53}
	Two copies of the sequence \( \{v_l(T) + c_4 m^{2/(\dim M - 1)}\}_{l=0,m=0}^\infty \) come from the operators \( \Delta_{T,N,\pm}^\mathbb{R} + \Delta^Y \) on \( [-1,0] \times Y \) with Neumann boundary conditions, and another two copies of the same sequence arise from the Laplacian on \( [0,1] \times Y \) with Neumann boundary conditions. Additionally, two copies of \( \{c_5 l^{2/\dim M}\}_{l=0}^\infty \) come from the usual Hodge Laplacian (with Neumann boundary conditions) on \( M_1 \) and \( M_2 \). It is easy to see that there exist constants \(\tau, C > 0\) such that for sufficiently large \(k\),  
\be\label{unif} u_k(1) \geq C k^\tau. \ee
To establish this, it suffices to show that the sequence \(\{v_l(1) + c_4 m^{2/(\dim M - 1)}\}_{l=0,m=0}^{\infty}\) exhibits power growth. Specifically, let \(\{\omega_k\}_{k=0}^{\infty}\) be the same sequence arranged in non-decreasing order. We aim to prove that \(\omega_k \geq C' k^{\tau'}\) for some constants \(C', \tau' > 0\). By the construction of \(v_l(1)\), the sequence \(\{\omega_k\}\) corresponds to the spectrum of some elliptic operator on \([-1,0] \times Y\), guaranteeing its power growth.
	\end{rem}

	\begin{proof}[Proof of Theorem \ref{larcon}]   
First, we claim that for any \( t > 0 \),  
\be\label{point-wise}
\lim_{T \to \infty} \Tr_s\left(Ne^{-t\Delta_T'}\big(1-\P^\delta(T)\big)\right) = \sum_{i=1}^2 \Tr_s\left(Ne^{-t\Delta_i'}\right).
\ee  

This claim follows easily from the uniform lower bound \eqref{unif} and the convergence of eigenvalues. For the reader’s convenience, we provide more details below.  

Fix an orthonormal basis \( \{\phi_j\}_{j=1}^\infty \) of \( L^2(\bm) \), where \( \phi_j \) is the \( j \)-th eigenform of \( \Delta_T \). For each \( k \geq 1 \), let \( \P^k(T) \) be the projection onto the span of \( \{\phi_j\}_{j=1}^k \). Similarly, we define the projection \( \P^k \) corresponding to the eigenspaces of \( \Delta_1 \oplus \Delta_2 \).  

By \Cref{est1} and \eqref{unif}, there exist constants \( C_0, \tau_0 > 0 \) independent of \( T \), such that for all \( k \geq 1 \), we have \( \lambda_k(T) \geq C_0 k^{\tau_0} \) and \( \lambda_k \geq C_0 k^{\tau_0} \). Thus, for any \( \epsilon > 0 \), there exists \( k_0 = k_0(\epsilon, t) \) such that  
\be\label{cp1}
\left|\Tr_s\Big(Ne^{-t\Delta_T'}\big(1-\P^{k_0}(T)\big)\Big)\right| < \epsilon \quad \text{and} \quad \left|\Tr_s(Ne^{-t(\Delta_1'\oplus\Delta_2')}(1-\P^{k_0}))\right| < \epsilon.
\ee  

Moreover, by \Cref{eigencon}, there exists \( T_0 = T_0(k_0,t,\ep) \) such that for \( T \geq T_0 \),  
\be\label{cp2} 
\left|\Tr_s\left(Ne^{-t\Delta_T'}\big(1-\P^\delta(T)\big)\P^{k_0}(T)\right) - \Tr_s\left(Ne^{-t(\Delta_1'\oplus\Delta_2')}\P^{k_0}\right)\right| < \epsilon.
\ee  

The claim then follows directly from \eqref{cp1} and \eqref{cp2}.

        Let $$F(t):=\dim(M)\sum_{k=1}^\infty e^{-t\max\{u_k(1),\delta\}}.$$ Then it is clear that \begin{itemize}
            \item By \eqref{unif}, $F(t)/t\in L^1((1,\infty))$.
            \item \begin{equation}\label{oneinfty}
			t^{-1}|\Tr_s\left(Ne^{-t\Delta_T'}(1-\P^\delta(T))\right)|\leq t^{-1}F(t).\end{equation}
        \end{itemize} 
		
		Hence by \eqref{point-wise} and the dominated convergence theorem,
		\[ \lim_{T\to\infty}\int_1^\infty t^{-1}\Tr_s\qty(Ne^{-t\Delta_T'}(1-\P^\delta(T)))dt=\sum_{i=1}^2\int_1^\infty t^{-1}\Tr_s(Ne^{-t\Delta_i'})dt.\]
		Similarly, one can show that
		\[ \lim_{T\to\infty}\int_1^\infty t^{-1}\Tr_s(Ne^{-t\Delta_{T,i}'})dt=\int_1^\infty t^{-1}\Tr_s(Ne^{-t\Delta_i'})dt.\]
		
	\end{proof}
	
	\section{A Gluing Formula for Heat Trace in Small Time}\label{glus}
  In this section, we use the finite propagation speed argument to reduce the study of the heat trace in the small-time regime to a one-dimensional model (see \Cref{comparison}).  For this purpose, we first introduce several Laplacians on model spaces.
	\subsection{Laplacians on model spaces}

	Let $\Delta_{B,1}$ be the usual Hodge Laplacian on $[-2,-1]$ with absolute boundary conditions, and $k_{B,1}$ be the heat kernel for $\Delta_{B,1}$. It's easy to see that $\ker(\Delta_{B,1})$ is one-dimensional and generated by constant functions. Thus, \be \label{tracer}\Tr_s(e^{-t\Delta_{B,1}})=\lim_{t\to\infty}\Tr_s(e^{-t\Delta_{B,1}})=1.\ee
	
	Let $\Delta_{B,2}$ be the usual Hodge Laplacian on $[1,2]$ with relative boundary conditions, and $k_{B,2}$ be the heat kernel of $\Delta_{B,2}$. Similarly, \be \label{tracer1}\Tr_s(e^{-t\Delta_{B,2}})=-1.\ee
	

	Let $\bar{\Delta}_{B}$ be the usual Hodge Laplacian on $[-2,2]$ with absolute boundary condition on $-2$, relative boundary condition on $2$, and $\bar{k}_{B}$ be its heat kernel.

	We can also regard $f_T$ as a smooth function in $(-2,2)$, and let ${\Delta}^\R_T$ be the Witten Laplacian on $(-2,2)$ with respect to $f_T$, with absolute boundary condition on $-2$, and relative boundary condition on $2$, 
	and ${k}_T$ be the heat kernel for ${\Delta}^\R_T$ .
	
	On the restriction of $F|_{Y}\to Y$, let $\Delta_{Y }$ be the induced Hodge Laplacian on $\Omega(Y;F|_{Y})$.

	\subsection{Finite propogation speed}
The following lemmas are needed to reduce the problems to the one-dimensional model in the small time.
	A standard finite-propagation-speed argument (c.f. \cite[$\S$ 4.4]{taylor1996pseudodifferential}) implies that
	\begin{lem}[Finite Propagation Speed]\label{speed0}
		Let $s_0$ be a smooth section of $\Omega(\bm;\F)$ with a compact support $C$, $s_t:=\exp(\sqrt{-1}tD_T)s_0$. Let $C_t:=\{x'\in \bm: d(C,x')\leq 2t\}$, then the support of $s_t$ is inside $C_t.$
	\end{lem}
	
	For $a,b\in (-2,2)(a<b)$, let 
	$M_{a,b}=[a,b]\times Y.$ It follows from \Cref{speed0} that
	
	
	\begin{lem}\label{speed}
    Let $K_T$ be the heat kernel for $\Delta_T$.
		Suppose $x\notin M_{-9/8,-7/8}\cup M_{7/8,9/8}$ and $d(x,x')\geq 4\delta$ for some $\delta>0$. Then there exist $T$-independent constants $c=c(\delta), C=C(\delta)>0$, such that for $t\in(0,1),$
		\[|K_T(t,x,x')|+|\nabla_x K_T(t,x,x')|\leq Ce^{-c/t}\]
		if $T$ is large enough. Here $\nabla_x$ means the derivative is taken in the $x$-direction.
	\end{lem}
	\begin{proof}
		The proof follows a standard argument, adapted from \cite[$\S$ 13.b]{bismut1991complex}.
		Choose a bump function $\phi$ on $\R$ such that
		$$
		\phi(\lambda)=1,|\lambda| \leqslant \delta ; \quad \phi(\lambda)=0,|\lambda| \geqslant 2 \delta .
		$$
		Let $f_1,f_2\in S(\R)$ such that
		$$
		\hat{f}_1(\lambda)=(4 \pi t)^{-\frac{1}{2}} \exp \left(-\lambda^2 / 4 t\right) \phi(\lambda), \quad \hat{f}_2(\lambda)=(4 \pi t)^{-\frac{1}{2}} \exp \left(-\lambda^2 / 4 t\right)\big(1-\phi(\lambda)\big),
		$$
		where $S(\R)$ is the Schwartz space on $\R.$
		
		Since $D_T^2=\Delta_T$, one can see that
		$$
		f_1(D_T)+f_2(D_T)=\exp \left(-t \Delta_T\right).
		$$
		Denote $K_1$ and $K_2$ to be the integral kernel of $f_1(D_T)$ and $f_2(D_T)$ respectively, then
		$$
		K_T(t, x, x')=K_{{1}}(t, x, x')+K_2(t, x, x').
		$$

		
		
		By Lemma \ref{speed0} and the construction of $\phi$, 
		$K_1(t,x,x')=0$ if $d(x,x')\geq 4\delta.$ So it suffices to establish the estimate for $K_2.$

		Let $R_t=f_2(D_T)$, then for any $L^2$-section $s$,
		$$
		\begin{aligned}
			R_t(s)(x):=\int_{\bm} K_2(t, x, x') s(x') d x' 
			&=\frac{1}{\sqrt{4 \pi t}} \int_{|\lambda| \geqslant \delta} \exp \left(-\frac{\l^2}{4t}\right)[1-\phi(\lambda)] \exp (i \lambda D_T) s(x) d \lambda .
		\end{aligned}
		$$
		
		
		Then for any  $L^2$-section $s,s'$, suppose that we have orthogonal decomposition
		$s=\sum_{l}s_l,s'=\sum_{l}s_l'$ associated with eigenforms of $D_T$, that is, $D_T s_l=\l_ls_l,D_Ts_l'=\l_ls_l'.$
		
		One can see that 	
		\begin{align}
			\begin{split}\label{ktwos}
				&\ \ \ \ \left|\left(\Delta^k_TR_ts,s'\right)_{L^2}\right|=\left|\left(\int_{\bm} \Delta^k_{T,x}K_2(t, x, x') s(x') d x',s'(x)\right)\right|\\
				&=\left|\left(\frac{1}{\sqrt{4 \pi t}} \int_{|\lambda| \geq \delta} \exp \left(-\lambda^2 / 4 t\right)[1-\phi(\lambda)] \exp (i \lambda D_T) D_T^{2 k} s(x) d \lambda,s'(x)\right)\right|\\
				&=\left|\sum_l\left(\frac{1}{\sqrt{4 \pi t}} \int_{|\lambda| \geq \delta} \exp \left(-\lambda^2 / 4 t\right)[1-\phi(\lambda)] \exp (i \lambda \l_l) \l^{2k}_l s_l(x) d \lambda,s'_l(x)\right)\right|\\
				&\leq C\sum_l \exp \left(-\delta^2 / 16 t\right) |(s_l,s_l')| \\
				&\leq C_k \exp (-\delta^2 / 16t)\|s\|_{L^2}\|s'\|_{L^2}\mbox{ (By Cauchy-Schwarz inequality)},
			\end{split}
		\end{align}
		where $\Delta_{T,x}$ means that the operator is acting on the $x$ component. 
		Here the equality in the third line follows from the fact that $s=\sum s_l$ is an orthogonal decomposition associated with eigenforms of $D_T$. For the first inequality, notice that $i\l_l\exp(i\l\l_l)d\l=d\exp(i\l\l_l)$, we integration by parts for $2k$ times, the $\l_l^{2k}$ term would eventually becomes 1. Since we still integral over $\{\l:|\l|\geq \delta\}$, replace $\exp(-\l^2/4t)$ by its upper bound $\exp(-\delta^2/8t)\exp(-\l^2/8t)$ and do the integration, the result follows.
		
		By \eqref{ktwos}, we have  \be\label{opn1}\|\Delta_T^kR_t\|\leq C_k \exp (-\delta^2 / 16t).\ee

		Let $\rho\in C_c^\infty(M)$ be a bump function, such that, $0\leq \rho\leq1$; in $\bm-M_{-\frac{17}{16},-\frac{15}{16}}- M_{\frac{15}{16},\frac{17}{16}}$, $\rho\equiv 1$; in $M_{-\frac{33}{32}, -\frac{31}{32}}\cup M_{\frac{31}{32}, \frac{33}{32}}$, $\rho\equiv 0$.

		Notice that when restricted in  $ \bm-M_{-\frac{33}{32},-\frac{31}{32}}- M_{\frac{31}{32},\frac{33}{32}}$, $\Delta_T\geq \Delta$ if $T$ is big enough (See Lemma \ref{constr}, \eqref{addeq31} and \eqref{addeq32}). Here $\Delta_T\geq \Delta$ on some open subset $U\subset \bm$ means that for all $\phi\in \Omega^*(U;\F|_{U})$ with compact support inside $U$, 
		\be\label{gar}(\Delta_T\phi,\phi )_{L^2}\geq (\Delta\phi,\phi )_{L^2} \mbox{ or equivalently } \|D_T\phi\|_{L^2}\geq \|D\phi\|_{L^2}.\ee
		
		Now 
		\begin{align}\begin{split}\label{opn3}
				&\ \ \ \ 	\int_{\bm}|D \rho R_ts|^2\dvol
				\leq\int_{\bm}|D_T \rho R_ts|^2\dvol\mbox{ (By \eqref{gar} )}\\
				&\leq C\int_{\bm}|D_T R_ts|^2+|R_ts|^2\dvol\mbox{ (By \eqref{addeq1})}\\
				&= C\int_{\bm}\lan \Delta_T R_ts, R_ts\ran+|R_ts|^2\dvol\\
				&\leq C_1'\exp(-\delta^2/16t)\|s\|_{L^2}^2\mbox{ (by \eqref{opn1})}.
		\end{split}\end{align}

	By Gårding's inequality, \eqref{opn3}, and \eqref{opn1} (with \(k=0\)), together with the definition of \(R_t\) and Sobolev's embedding theorem, we deduce that  
\[
| \rho(x) K_2(t, x, x')|+|\nabla_x \rho(x) K_2(t, x, x')| \leq C e^{-\delta^2 / 16t}.
\]	
	\end{proof}
	Let $\eta_i(i=1,2)$ be a smooth function on $(-\infty,\infty)$ satisfying
	\begin{enumerate}[(1)]
		\item $0\leq\eta_i\leq1$.
		\item $\eta_1\equiv1$ in $(-\infty,-3/2)$,$\eta_1\equiv 0$ in $(-5/4,\infty)$;
		\item $\eta_2(s)=\eta_1(-s)$.
	\end{enumerate}
	Let $\eta$ be an even function, such that $\eta|_{[0,\infty)}\equiv\eta_2.$
	We can regard $\eta_i$ as a smooth function on $M_i\subset\bm(i=1,2) $ and $\eta$ as a smooth function on $\bm.$
	
	Note that on $\bm-[-1,1]\times Y$, $\Delta_T=\Delta_1\oplus \Delta_2$; on $(-2,2)\times Y$, $\Delta_T=\Delta_T^\R+\Delta_Y$; on $(-2,-1)\times Y\cup(1,2)\times Y$, $\Delta_T=\bar\Delta_B+\Delta_Y$, by finite propagation speed, we should have
	\begin{lem}\label{heatm}
		There exist $T$-independent constants $C$ and $c>0,$ such that for $t\in(0,1)$,
		\[\qty|\Tr_s(N(1-\eta)e^{-t\Delta_T})-\Tr_s(N(1-\eta)e^{-t\Delta^\R_{T}}\otimes e^{-t\Delta_Y})|\leq Ce^{-c/t},\]
		\[\left|\Tr_s(N\eta e^{-t(\Delta_T)})-\sum_{i=1}^2\Tr_s(N\eta_ie^{-t(\Delta_i)})\right|\leq Ce^{-c/t},\]
		and
		\[\qty|\Tr_s(N\eta_ie^{-t\Delta^\R_{T}}\otimes e^{-t\Delta_Y})-\Tr_s(N\eta_ie^{-t\bar{\Delta}_{B}}\otimes e^{-t\Delta_Y})|\leq Ce^{-c/t}.\]
	\end{lem}
	\begin{proof}
		Let $f_1$ and $f_2$ be functions constructed in Lemma \ref{speed} for some $T$-independent and small $\delta>0$ .

		Let $u$ be a unit eigenform of eigenvalue $\l$ for $\Delta_T$. For any $j\in \Z_+$, it is easy to see that $|f_2(s)|\leq\frac{C_je^{-c/t}}{(|s|^2+1)^j}$ (similar to the first inequality in \eqref{ktwos}), hence
		\begin{align}\begin{split}\label{speed22}
				\qty|\Big(N(1-\eta)f_2(D_T)u,u\Big)_{L^2}|\leq \frac{C_je^{-c/t}}{(|\l|+1)^j}.
		\end{split}\end{align}
		By Lemma \ref{est1} and (\ref{speed22}), choosing large enough $j$, one has \be
		\label{speed23}|\Tr_s(N(1-\eta)f_2(D_T))|\leq Ce^{-c/t}.\ee Similarly, 
		\be\label{speed24}|\Tr_s(N(1-\eta)f_2(D_T^\R\oplus D^Y)|\leq Ce^{-c/t}.\ee
		Let  $M_1'=M_1-(-2,-1]\times Y$ and $M_2'=M_2-[1,2)\times Y$. Since $L^2(\bm)=L^2(M_{1}')\oplus L^2([-2,2]\times Y)\oplus L^2(M_2'),$ let $\{u_k\}$, $\{w_k\}$ and $\{v_k\}$ be an orthonormal basis of $L^2(M_{1}')$, $L^2([-2,2]\times Y)$ and $L^2(M_2')$ respectively. 
		Then by Lemma \ref{speed0} and the fact that the support of $1-\eta$ is away from $M_1'$ and $M_2'$, a standard argument shows that, if $\delta$ is small enough,
		\be\label{speed25}(N(1-\eta)f_1(D_T) u_k,u_k)=(N(1-\eta)f_1(D_T) v_k,v_k)=0,\ee
		\be\label{speed251}(N(1-\eta)f_1(D^\R_T\oplus D^Y) u_k,u_k)=(N(1-\eta)f_1(D^\R_T\oplus D^Y) v_k,v_k)=0,\ee
		and 
		\be\label{speed26}(N(1-\eta)f_1(D_T) w_k,w_k)=(N(1-\eta)f_1(D^\R_T\oplus D^Y) w_k,w_k).\ee
		
		Since $e^{-ts^2}=f_1(s)+f_2(s)$, by (\ref{speed23}), (\ref{speed24}), (\ref{speed25}), \eqref{speed251} and (\ref{speed26})
		\[|\Tr_s(N(1-\eta)e^{-t\Delta_T})-\Tr_s(N(1-\eta)e^{-t\Delta^\R_{T}}\otimes e^{-t\Delta_Y})|\leq Ce^{-c/t}.\]
		
		Similarly, one can derive the other two inequalities.
		
	\end{proof}
	
	Similarly,
	\begin{lem}\label{conbd}
		There exists $T$-independent $C,c>0,$ such that for $t\in(0,1),i=1,2$,
		\[|\Tr_s(N(1-\eta_i)e^{-t\Delta_i})-\Tr_s(N(1-\eta_i)e^{-t\Delta_{B,i}}\otimes e^{-t\Delta_Y})|\leq Ce^{-c/t},\]
		\[|\Tr_s(N\eta_ie^{-t\Delta_{B,i}}\otimes e^{-t\Delta_Y})-\Tr_s(N\eta_ie^{-t\bar{\Delta}_{B}}\otimes e^{-t\Delta_Y})|\leq Ce^{-c/t}.\]
	\end{lem}
	
	\subsubsection{Comparison of $e^{-t(\Delta_1\oplus\Delta_2)}$ and $e^{-t\Delta_T}$}
	By the finite propagation speed, or more explicitly by \Cref{heatm} and \Cref{conbd}, we establish the following lemma, which reduces the problem to one-dimensional models.
    \begin{lem}\label{comparison}
         There exists $C,c>0$, such that for $t\in(0,1]$,
     	\begin{align}\begin{split}\label{ntn12}
			&\left|\Tr_s(Ne^{-t\Delta_T})-\Tr_s(Ne^{-t(\Delta_1\oplus\Delta_2)})\right.\\
            &\left.\quad-\left(\Tr_s(N^\R e^{-t{\Delta}_T^\R})-\Tr_s\big(N^\R e^{-t(\Delta_{B,1}^\R\oplus\Delta_{B,2}^\R)}\big)\right)\chi(Y)\rank(F)\right|\leq Ce^{-c/t}
	\end{split}\end{align}
    \end{lem}
    \begin{proof}
    By \Cref{conbd}, we have for $t\in(0,1]$,
	\begin{align}\begin{split}\label{heatbd1}
			&\ \ \ \ \left|\Tr_s(Ne^{-t(\Delta_1\oplus\Delta_2)})-\sum_{i=1}^2\Tr_s(N\eta_ie^{-t\Delta_i})\right.\\
			&+\left.\sum_{i=1}^2\left(\Tr_s(N\eta_ie^{-t\bar{\Delta}_{B}}\otimes e^{-t\Delta_Y})-\Tr_s(Ne^{-t\Delta_{B,i}}\otimes e^{-t\Delta_Y})\right)\right|\\
			&\leq C\exp(-c/t).
	\end{split}\end{align}
	Next, notice that on $[-2,-1]\times Y$, the number operator can be decomposed as $N=N^{Y}+N^\R$ (Here $N^Y$ and $N^\R$ are the number operators on $Y$ and $\R$ components respectively),
	\begin{align}\begin{split}\label{trnby}
			&\ \ \ \ \Tr_s(Ne^{-t\Delta_{B,i}}\otimes e^{-t\Delta_Y})\\
			&=\Tr_s(N^\R e^{-t\Delta_{B,i}}\otimes e^{-t\Delta_Y})+\Tr_s(e^{-t\Delta_{B,i}}\otimes N^Ye^{-t\Delta_Y})\\
			&=\Tr_s(N^\R e^{-t\Delta_{B,i}^\R})\chi(Y)\rank(F)+\Tr_s(N^{Y}e^{-t\Delta_Y})\Tr_s(e^{-t\Delta_{B,i}}).
	\end{split}\end{align}

	As a result, by (\ref{tracer}), (\ref{tracer1}), \eqref{trnby}  and the fact that $\sum_{i=1}^2\Tr_s(e^{-t\Delta_{B,i}})=0$,
	\begin{align}\begin{split}\label{heatbd2}
			&\ \ \ \   \sum_{i=1}^2\Tr_s(Ne^{-t\Delta_{B,i}}\otimes e^{-t\Delta_Y})\\
			&=\Tr_s(N^\R e^{-t\Delta_{B,1}^\R})\chi(Y)\rank(F)+\Tr_s(N^\R e^{-t\Delta_{B,2}^\R})\chi(Y)\rank(F).\\
	\end{split}\end{align}

	It follows from Lemma \ref{heatm} that
	\begin{align}\begin{split}\label{heatm1}
			&\ \ \ \ \left|\Tr_s(Ne^{-t\Delta_T})-\sum_{i=1}^2\Tr_s(N\eta_ie^{-t\Delta_i})\right.\\
			&+\left.\sum_{i=1}^2\left(\Tr_s(N\eta_ie^{-t\bar{\Delta}_{B}}\otimes e^{-t\Delta_Y})-\Tr_s(Ne^{-t\Delta^\R_{T}}\otimes e^{-t\Delta_Y})\right)\right|\\
			&\leq C\exp(-c/t).
	\end{split}\end{align}
	Moreover,
	\begin{align}\begin{split}\label{trnry}
			&\ \ \ \ \Tr_s(Ne^{-t\Delta^\R_{T}}\otimes e^{-t\Delta_Y})\\
			&=\Tr_s(N^{Y}e^{-t\Delta_Y})\Tr_s(e^{-t{\Delta}_T^\R})+\Tr_s(N^\R e^{-t{\Delta}_T^\R})\Tr_s(e^{-t\Delta_Y})\\
			&=\Tr_s(N^{Y}e^{-t\Delta_Y})\Tr_s(e^{-t{\Delta}_T^\R})+\Tr_s(N^\R e^{-t{\Delta}_T^\R})\chi(Y)\rank(F).\\
	\end{split}\end{align}
	
	Note that $\Tr_s(e^{-t{\Delta}_T^\R})=\lim_{t\to\infty}\Tr_s(e^{-t{\Delta}_T^\R})=\dim(\ker({\Delta}_T^\R)_0)-\dim(\ker({\Delta}_T^\R)_1)$. Here $\ker({\Delta}_T^\R)_i$ denotes the space of harmonic $i$-forms$(i=0,1)$. Since $f_T$ is odd, one can see easily that if $u(s)\in\ker({\Delta}_T^\R)_0$, then $u(-s)ds\in\ker({\Delta}_T^\R)_1.$
	
	As a result, $\Tr_s(e^{-t{\Delta}_T^\R})=0$. So by \eqref{trnry},
	\begin{align}\begin{split}\label{heatm2}
			\Tr_s(Ne^{-t\Delta^\R_{T}}\otimes e^{-t\Delta_Y})=\Tr_s(N^\R e^{-t{\Delta}_T^\R})\chi(Y)\rank(F).
	\end{split}\end{align}
	
	Estimate \eqref{ntn12} then follows from (\ref{heatbd1}), (\ref{heatbd2}), (\ref{heatm1}), (\ref{heatm2}).
	\end{proof}
	\subsection{Proof of Theorem \ref{main} when $\dim(M)$ is even}
	First, by Theorem \ref{eigencon}, if $T$ is large, $\dim(\Im\P^\delta(T))=\dim(\ker\Delta_1)+\dim(\ker\Delta_2)$. So for $t\in(0,1)$ (note that $e^{-ct}-1=O(t)$ for $|c|\leq\delta$),
	\begin{align}\begin{split}\label{tracens}
			&\ \ \ \ \Tr_s(Ne^{-t\Delta_T})-\Tr_s(N e^{-t(\Delta_1\oplus\Delta_2)})\\
			&=\Tr_s\Big(Ne^{-t\Delta_T}(1-P^\delta(T))\Big)-\Tr_s(N e^{-t(\Delta_1\oplus\Delta_2)'})+O(t).
	\end{split}\end{align}
Here, and throughout, \( O(g(t)) \) denotes terms satisfying \( |O(g(t))| \leq C g(t) \) for some constant \( C \) independent of \( T \), where \( g(t) \) is any positive function and \( t \in (0,1) \).
    
	When $M$ is even-dimensional, $\chi(Y)=0.$
	It follows from \eqref{ntn12} and (\ref{tracens}) that
	\[\Big|\Tr_s\Big(Ne^{-t\Delta_T}(1-P^\delta)\Big)-\Tr_s(N e^{-t(\Delta_1\oplus\Delta_2)'})\Big|=O(t)\] 
	for $t\in(0,1).$
	So by \eqref{point-wise} and the dominated convergence theorem,
	\[\lim_{T\to\infty}\int_0^1t^{-1}\Big|\Tr_s\Big(Ne^{-t\Delta_T}(1-P^\delta)\Big)-\Tr_s(N e^{-t(\Delta_1\oplus\Delta_2)'})\Big|dt=0.\]
	Hence, \[(\zeta^{\mS}_{T,\la})'(0)=(\zeta_1^{\mS})'(0)+(\zeta_2^{\mS})'(0)+o(1).\]
Similarly, \[(\zeta^{\mS}_{i,T})'(0)=(\zeta_i^{\mS})'(0)+o(1).\]

	\section{One-dimensional Models and Coupling Techniques}\label{cons}
To prove Theorem \ref{main} when \( \dim(M) \) is odd, it suffices, by \eqref{ntn12}, to estimate one-dimensional heat trace
\[
\Tr_s(N^\R e^{-t{\Delta}^\R_{T}}) - \Tr_s(N^\R e^{-t(\Delta_{B,1}^\R\oplus\Delta^\R_{B,2})})
\]  
This involves the coupling techniques. More explicitly,
we estimate coupled heat trace $\Tr_s(N^\R e^{-t{\Delta}^\R_{t^{-7}T}})$ in \cref{coups}, and show that (\Cref{lem77}) \[|\Tr_s(N^\R e^{-t{\Delta}^\R_{t^{-7}T}}) - \Tr_s(N^\R e^{-t(\Delta_{B,1}^\R\oplus\Delta^\R_{B,2})})|=O(t^{0.1}),\]
which enable us to apply dominated convergence theorem.

In \Cref{partial72}, we establish \Cref{main1} and \Cref{main2}, which are equivalent to \Cref{main} if we disregard certain mysterious terms arising from the comparison between the coupled and uncoupled heat traces:  
\[
\Tr_s(N^\R e^{-t{\Delta}^\R_{T}}) - \Tr_s(N^\R e^{-t{\Delta}^\R_{t^{-7}T}}).
\]  
Finally, in \Cref{last}, we explicitly compute these terms, thereby fully establishing \Cref{main}.
 
 First, it is straightforward to see that
    \begin{prop}\label{prop71}
        	Let \be\label{vtpm} V_{T,\pm}=|f_T'|^2\pm f_T''.\ee Recall that $k_{T}$ is the heat kernel for Witten Laplacian $\Delta^\R_{T}$ on $(-2,2)$ with the absolute boundary condition on $-2$ and the relative boundary condition on $2$. Then restricted to functions, \be\label{delpm1}\Delta^\R_{T}=\Delta^\R_{T,-}:=-\partial^2_s+V_{T,-}\ee with the Neumann boundary condition on $-2$ and the Dirichlet boundary condition on $2$. Restricted on 1-forms (and identify $u(s)ds$ with $u(s)$ )
	\be\label{delpm2}\Delta^\R_{T}=\Delta^\R_{T,+}:=-\partial^2_s+V_{T,+}\ee with Dirichlet boundary condition on $-2$ and Neumann boundary condition on $2$.
    \end{prop}


	
	
	
	
	\subsection{Heat trace estimate when $t$ and $T$ are coupled}\label{coups}
	
	\def\tT{{\tilde{T}}}
	
	We introduce the following coupling:  
\[
\tilde{T} = t^{-7} T.
\]  
Our goal is to establish the key result \Cref{lem77}.  

Before proceeding with the estimates, we briefly outline the idea behind the coupling technique. Let \(\lambda_{k,T,\pm}\) denote the \(k\)-th eigenvalue of \(\Delta^\mathbb{R}_{T,\pm}\), and let \(\phi_{k,T,\pm}\) be the corresponding eigenfunction such that \(\{\phi_{k,T,\pm}\}_{k=1}^\infty\) forms a complete orthonormal basis of \(L^2([-2,2])\). Denote by \(k_{T,\pm}\) the heat kernel associated with \(\Delta^\mathbb{R}_{T,\pm}\). Throughout this section, we assume \(t \in (0,1]\).  

Expressing the heat kernel as  
\[
k_{T,\pm} = \sum_k e^{-t\lambda_{k,T,\pm}} \phi_{k,T,\pm}^* \otimes \phi_{k,T,\pm},
\]  
Lemma \ref{limit0} implies that  
\[
\int_{-1}^{1} |\phi_{k,T,\pm}|^2 \leq \frac{C\lambda_{k,T,\pm} + 1}{\sqrt{T}}.
\]  
This observation shows that, with the coupling \(t\) and \(T\) as defined above, the integral of the heat kernel over \([-1,1]\) can be bounded by a positive power of \(t\) (see Proposition \ref{heat1form}).   Similarly, by Lemma \ref{limit2}, the heat kernel \(k_{\tilde{T},\pm}\) behaves similarly to the standard heat kernel with the desired boundary conditions on \([-2,-1] \cup [1,2]\) (see Lemma \ref{lacon}).
	
	\begin{prop}\label{heat1form}
		If $t\in(0,1],$
		\[\int_{-1}^1 |k_{\tT,\pm}(t,s,s)|ds\leq \frac{Ct^2}{\sqrt{T}}\]
		for some constant $C$ that doesn't depend on $T.$
	\end{prop}
	\begin{proof}
		By Lemma \ref{limit0},
		\begin{equation}\label{Tl}\int_{-1}^1 |\phi_{k,\tT,\pm}|^2(s)ds\leq \frac{C(1+\l_{k,\tT,\pm})}{\sqrt{\tT}}.\end{equation}
		For a fixed $a\in(0,1)$, we have
		\begin{align*}
			\int_{-1}^1 k_{\tT,\pm}(t,s,s)ds&=\int_{-1}^1 \sum_{k}e^{-t\lambda_{k,\tT,+}}|\phi_{k,\tT,+}|^2ds\\
			&\leq C\sum_{k\geq1}e^{-t\lambda_{k,\tT,+}}\frac{1+\l_{k,\tT,+}}{\sqrt{\tT}}\mbox{ (By \eqref{Tl})}\\&\leq C\sum_{k\geq1}e^{-at\lambda_{k,\tT,+}}\frac{1}{t\sqrt{\tT}}\leq \frac{Ct^2}{\sqrt{T}}.
		\end{align*}
	The first inequality in the last line follows from the fact that for some $C>0$, \( |e^{-xt} xt| \leq C\) for all \( x > 0 \) and \( t > 0 \). The second inequality in the last line follows from the one dimensional version of Lemma \ref{est1} that $\l_{k,\tT,+}\geq c (k-C)_+^{2}$ for some $\tT$-independent positive constants $c$ and $C.$
	\end{proof}
	
	\def\bk{\bar{k}}
Recall that \(\bar{\Delta}_{B}\) denotes the standard Hodge Laplacian on \([-2,2]\) with absolute boundary conditions at \(-2\) and relative boundary conditions at \(2\). Let \(\bar{k}_{B}\) be its heat kernel. Denoting by \(\bar{k}_{B,+}\) the restriction of \(\bar{k}_B\) to \(1\)-forms (where we identify \(u(s) ds\) with \(u(s)\)), we impose Dirichlet boundary conditions at \(-2\) and Neumann boundary conditions at \(2\).
	
	To orient the reader to the details in the proof of Lemma \ref{addlem} and the first part of the proof in Lemma \ref{kx0}, we provide a concise explanation here. Naively, it's apparent that on $(-2, 0)$, $\Delta_{T,+}^\R\geq\bar{\Delta}_{B,+}$ except on a very small interval $(-1,-1+e^{-T^2})$ (when $T$ is large). Intuitively, this suggests that the upper bound of $|k_{T,+}|(t,s,s),t\in(0,1),s\in(-2,-1)$ should be controlled by the upper bound of the standard heat kernel.
	
	\begin{lem}\label{addlem}
		$0\leq k_{T,+}\leq \frac{e^{CT}}{\sqrt{t}}$ for $t\in(0,1].$
	\end{lem}

	\begin{proof}
		    The lemma is  the maximal principle. 
		
		By \Cref{constr}, there exists $T$-independent $C_0$, such that $V_{T,\pm}>-C_0T$. Then
		\be\label{ppsekt}(\p_t-\p_s^2)e^{-C_0Tt}k_{T,+}(t,s,s')=e_{T}(s)e^{-C_0Tt}k_{T,+}(t,s,s')\ee
		where $e_T=-V_{T,+}-C_0T<0.$
		
		By contradiction, assume that $e^{-CTt}k_{T,+}(t_0,s_0,s')<0$ for some $(t_0,s_0)$. Then by boundary conditions, we must have $t_0>0$ and $s_0\in(-2,2].$
		Now suppose $e^{-C_0Tt}k_{T,+}(t,s,s')$ achieves its global minimal (which is negative by assumption) at $(t_*,s_*)\in(0,1]\times[-2,2]$.
		
		Then $s_*\in(-2,2]$. If $s_*\in(-2,2)$, then $\p_se^{-C_0Tt_*}k_{T,+}(t_*,s,s')|_{s=s_*}=0$. If $s_*=2$, by boundary condition, we also have $\p_se^{-C_0Tt_*}k_{T,+}(t_*,s,s')|_{s=s_*}=0$. As a result, the second derivative test implies that $\p_s^2e^{-C_0Tt_*}k(t_*,s,s)|_{s=s_*}\geq0.$
		
		However, \eqref{ppsekt} then implies that 
		\be\label{ptpo}	\p_te^{-C_0Tt}k_{T,+}(t,s_*,s')|_{t=t_*}=\p_s^2e^{-C_0Tt_*}k_{T,+}(t_*,s,s')|_{s=s_*}+e_{T}(s_*)e^{-C_0Tt_*}k_{T,+}(t_*,s_*,s')>0\ee
		For the last inequality, notice that $e_T(s_*)<0$ and $e^{-C_0Tt_*}k_{T,+}(t_*,s_*,s)<0$ by assumption. 
		However, \eqref{ptpo} contradicts with the fact that $e^{-C_0Tt}k_{T,+}(t,s,s')$ achieves its global minimal at $(t_*,s_*).$
		
		As a result, $e^{-C_0Tt}k_{T,+}(t,s,s')\geq0$, which implies $k_{T,+}(t,s,s')\geq0.$
		
		Now 
		\be\label{ppsekt1}(\p_t-\p_s^2)(e^{-C_0Tt}k_{T,+}(t,s,s')-\bk_{B,+}(t,s,s'))=e_{T}(s)e^{-C_0Tt}k_{T,+}(t,s,s')\leq 0\ee
		
		Proceeding as what we just did, we have
		\be\label{theineqabove} e^{-C_0Tt}k_{T,+}(t,s,s')-\bk_{B,+}(t,s,s')\leq 0.\ee
		The Lemma then follows from the fact that $\bk_{B,+}\leq\frac{C}{\sqrt{t}}$.
	\end{proof}
	
	\begin{lem}\label{kx0}
		For $t\in(0,1), s\in(-2,-1)$, and $T\geq1$, 
		\[ 0\leq k_{T,+}(t,s,-1)\leq\min\{\frac{C}{t^{3/2}T^{1/4}},\frac{C}{\sqrt{t}}\}\] and for  $t\in(0,1), s\in(1,2)$, and $T\geq 1,$
		\[0\leq k_{T,-}(t,s,1)\leq \min\{\frac{C}{t^{3/2}T^{1/4}},\frac{C}{\sqrt{t}}\}.\]
	\end{lem}
	\begin{proof} 
		$\bullet$ We first establish that \( k_{T,+}(t,s,-1) \leq \frac{C}{\sqrt{t}} \):\\

		Let $\rho\in C_c^\infty(-\infty,\infty)$ be a nonnegative function, such that $\rho\equiv1$ on $(-\infty,-1/8)$, $\rho\equiv 0$ on $(-1/16,\infty)$.
		By \eqref{vtpm}, \eqref{delpm1}, \eqref{delpm2} and a straightforward computation,
		\begin{align*}&\ \ \ \ (\p_t+\Delta^\R_{T,+,s'})(\rho(s') \bk_{B,+}(t,s,s')-k_{T,+}(t,s,s'))\\
			&=-2\rho'(s')\partial_{s'}\bk_{B,+}(t,s,s')-\rho''(s')\bk_{B,+}(t,s,s')+\rho(s')V_{T,+}(s')\bk_{B,+}(t,s,s').\end{align*}
		Here $\Delta^\R_{T,+,s'}$ means that the operator is acting on the $s'$ component.

		Set $$h(t,s,s')=-2\rho'(s')\partial_{s'}\bk_{B,+}(t,s,s')-\rho''(s')\bk_{B,+}(t,s,s')+\rho(s')V_{T,+}(s')\bk_{B,+}(t,s,s'),$$
		$$l(t,s,s')=-2\rho'(s')\partial_{s'}\bk_{B,+}(t,s,s')-\rho''(s')\bk_{B,+}(t,s,s').$$
		
		For $s\in(-2,-1)$, we compute
		\begin{align}\begin{split}\label{newed1}\ \ \ \ &0\leq k_{T,+}(t,s,-1)=\bk_{B,+}(t,s,-1)-\int_0^t\int_{-2}^{2}k_{T,+}(t-t',-1,s')h(t',s,s')ds'dt'\\&\leq\bk_{B,+}(t,s,-1)+C\int_0^t\int_{-1}^{-1+e^{-T^2}}\frac{T^2e^{CT}}{\sqrt{t-t'}\sqrt{t'}}ds'dt'\\&-\int_0^t\int_{-2}^{2}k_{T,+}(t-t',-1,s')l(t',s,s')ds'dt'.
		\end{split}	\end{align}
		Here the equality in the first line follows from the Duhamel's principle. The inequality in the second line follows from the fact that $\rho V_{T,+}(s)\geq 0$ if $s\notin [-1,-1+e^{-T^2}].$ Therefore outside $[-1,-1+e^{-T^2}]$, $h$ is replaced by $l$, while inside $[-1,-1+e^{-T^2}]$, the extra term $\rho(s')V_{T,+}\bar{k}_{B,+}(t',s,s')k_{T,+}(t-t',-1,s')$ is bounded above by $\frac{CT^2e^{CT}}{\sqrt{t-t'}\sqrt{t'}}$ by Lemma \ref{addlem}.
		
		It is obvious that if $T$ is large,\be\label{newed2}
		\int_0^t\int_{-1}^{-1+e^{-T^2}}\frac{T^2e^{CT}}{\sqrt{t-t'}\sqrt{t'}}ds'dt'\leq C.
		\ee
		Also, since supports of $\rho'$ and $\rho''$ are away from $(-2,-1)$, by Lemma \ref{speed}
		\be\label{newed3}
		-\int_0^t\int_{-2}^{2}k_{T,+}(t-t',-1,s')l(t',s,s')ds'dt'\leq C\int_0^t\int_{-1/16}^{-1/8}e^{-c/(t-t')}e^{-c/t'}dsdt' \leq C'.\ee
		By \eqref{newed1}, \eqref{newed2}, \eqref{newed3} and the classical heat kernel estimate for $\bar{k}_{B,+}$, $k_{T,+}(t,s,-1)\leq \frac{C}{\sqrt{t}}.$
		\ \\
		$\bullet$ Next we show that $|k_{T,+}|(t,s,-1)\leq \frac{C}{t^{3/2}T^{1/4}}$: \\
		Let $\phi_{k,T,+}$ be a unit eigenfunction for $\l_{k,T,+}$, then $\phi_{k,T,+}$ satisfies
		\begin{equation*}
			\begin{cases}
				-\phi''_{k,T,+}=\l_{k,T,+}\phi_{k,T,+}\mbox{, in $(-2,-1)$}\\
				\phi_{k,T,+}(-2)=0.
			\end{cases}
		\end{equation*}
		Hence $\phi_{k,T,+}(s)=c_k\sin(\l_{k,T,+}(s+2))$ in $(-2,-1)$ for some $c_k$.
		$$\frac{2c_k^2\l_{k,T,+}+\sin(2\l_{k,T,+})}{4\l_{k,T,+}}=\int_{-2}^{-1}|\phi_{k,T,+}|^2\leq\int_{-2}^{-2}|\phi_{k,T,+}|^2= 1$$ implies that $c_k\leq C$ for some constant $C$ that doesn't depend on $T$. As a result, $|\phi_{k,T,+}|(s)\leq C$ if $s\in(-2,-1).$
		
		Moreover, Lemma \ref{limit2} implies that $|\phi_{k,T,+}|^2(-1)\leq \frac{C(\l_{k,T,+}+1)^2}{\sqrt{T}}$. Hence, by a similar estimates in Lemma \ref{est1}, for a fixed $a\in(0,1),$ similar to the proof of Proposition \ref{heat1form},	
		\begin{align*}
			|k_{T,+}(t,s,-1)|&= |\sum_{k}e^{-t\lambda_{k,T,+}}\phi_{k,T,+}(s)\phi_{k,T,+}(-1)|\\
			&\leq C\sum_{k}e^{-t\lambda_{k,T,+}}\frac{{\l_{k,T,+}+1}}{T^{1/4}}\leq C\sum_{k}e^{-at\lambda_{k,T,+}}\frac{1}{{t}T^{1/4}}\\
			&\leq \frac{C}{t^{3/2}T^{1/4}}.
		\end{align*}
		The second inequality could be proved similarly. 
	\end{proof}
	
	
	Recall that $\Delta_{B,1}^\R$ is the usual Hodge Laplacian on $[-2,-1]$ with absolute boundary conditions, and let $k_{B,1}$ be the heat kernel of $\Delta_{B,1}^\R$. Let $k_{B,1,+}$ be the restriction of $k_{B,1}$ on 1-forms (and identify $u(s)ds$ with $u(s)$). 
	Then $k_{B,1,+}$ satisfies the Dirichlet boundary on $\{-1,-2\}$. While $\Delta_{B,2}^\R$ is the usual Hodge Laplacian on $[1,2]$ with relative boundary conditions, $k_{B,2}$ be the heat kernel of $\Delta_{B,2}^\R$. Let $k_{B,2,-}$ be the restriction of $k_{B,2}$ on functions.

	\begin{lem}\label{lacon}
		For $t \in (0, 1]$ and T sufficiently large,
		\[\int_{-2}^{-1}|k_{\tT,+}-k_{B,1,+}|(t,s,s)ds\leq \frac{Ct^{0.1}}{T^{0.01}}; \]
		\[\int_{1}^{2}|k_{\tT,-}-k_{B,2,-}|(t,s,s)ds\leq \frac{Ct^{0.1}}{T^{0.01}}.\]
	\end{lem}
	\begin{proof}
		
		First, one can see that for classical heat kernels, $$|\partial_{s'}k_{B,1,+}(t,s,s')|_{s'=-1}|\leq \frac{C(s+1)e^{-(s+1)^2/t}}{t^{3/2}}$$ for $s\in[-2,-1].$	Notice that, $k_{\tT,+}(t',s,s')$ and $k_{B,1,+}(t',s,s')$ satisfy the same equation for $t'\in(0,t),s\in[-2,-1]$. That is, $$ (\p_{t'}-\p_{s}^2)k_{\tT}(t',s,s')=(\p_{t'}-\p_{s}^2)k_{B,1,+}(t',s,s')=0, t'\in(0,t),s\in(-2,-1).$$ Also, both $k_{\tT,+}$ and $k_{B,1,+}$ satisfy the Dirichlet boundary condition at $-2$. It follows from Duhamel principle and the inequality above that for $s\in (-2,-1)$,
		\begin{align*}&\ \ \ \ |k_{\tT,+}-k_{B,1,+}|(t,s,s)\leq\int_0^t\Big|\partial_{s'}k_{B,1,+}(t-t',s,s')|_{s'=-1}k_{\tT,+}(t',s,-1)\Big|dt'\\
			&\leq C'\int_0^t \frac{(s+1)e^{-(s+1)^2/(t'-t)}}{(t-t')^{3/2}}\left|k_{\tT,+}(t',-1,s)\right|dt'=J_1(t,s).\end{align*}
		Hence,
		\begin{align*}
			&\ \ \ \ \int_{-2}^{-1}J_1(t,s)ds\leq C\int_{-2}^{-1}\int_0^t \frac{(s+1)e^{-(s+1)^2/(t'-t)}}{(t-t')^{3/2}}|k_{\tT,+}(t',s,-1)|dt'ds\\
			&\leq C\int_{0}^t\int_{-2}^{-1} \frac{(s+1)e^{-(s+1)^2/(t'-t)}}{(t-t')^{3/2}}|k_{\tT,+}(t',s,-1)|dsdt'\mbox{ (By Fubini's Theorem)}\\
			&\leq C'\int_{0}^t \frac{\sup_{s\in(-2,-1)}|k_{\tT,+}(t',s,-1)|}{\sqrt{t-t'}}dt'\mbox{ (Integrate over $ds$)}\\
			&= C'\left(\int_{0}^{t^{1.2}/T^{0.02}}\frac{\sup_{s\in(-2,-1)}|k_{\tT,+}(t',s,-1)|}{\sqrt{t-t'}}dt'+\int_{t^{1.2}/T^{0.02}}^t\frac{\sup_{s\in(-2,-1)}|k_{\tT,+}(t',s,-1)|}{\sqrt{t-t'}}dt'\right)\\
			&=I_1+I_2.
		\end{align*}
		While by Lemma \ref{kx0},
		\begin{align*}
			&\ \ \ \ I_1\leq C\int_0^{t^{1.2}/T^{0.02}}\frac{1}{\sqrt{t-t'}\sqrt{t'}}dt'\leq \frac{C'}{\sqrt{t}}\int_0^{t^{1.2}/T^{0.02}}\frac{1}{\sqrt{t'}}dt'=\frac{C''t^{0.1}}{T^{0.01}}.
		\end{align*}
		Here the second inequality follows from the fact that $\frac{1}{\sqrt{t-t'}}\leq \frac{c}{\sqrt{t}}$ if $t'\in[0,\frac{t^{1.2}}{T^{0.02}}]$ for some $T$-independent $c$ (Note that $T$ is large).
		
		By Lemma \ref{kx0} again,
		\begin{align*}
			&\ \ \ \ I_2\leq 
			C\int_{t^{1.2}/T^{0.02}}^t\frac{1}{\sqrt{t-t'}}\frac{1}{{t'}^{3/2}\tT^{1/4}}dt'\\
			&\leq \frac{C'}{{t^{0.05}T^{0.22}}}\int_{t^{1.2}/T^{0.02}}^t\frac{1}{\sqrt{t-t'}}dt' \leq\frac{C''t^{0.45}}{T^{0.22}}.
		\end{align*}
		Here the first inequality in the second line follows from the fact that $t'\geq t^{1.2}/T^{0.02}$ if $t'\in[t^{1.2}/T^{0.02},t]$ and the definition of $\tT$. Also, for example, the $t^{0.05}$ term comes from $7/4-1.2\times 3/2=-0.05.$
		
		Similarly, $\int_{1}^{2}|k_{\tT,-}-k_{B,1,-}|(t,s,s)ds\leq \frac{Ct^{0.1}}{T^{0.01}}.$
	\end{proof}
	\begin{lem}\label{strkt}
		\[\left|\int_{1}^2\tr_s(k_{\tT}(t,s,s))ds+1\right|\leq Ce^{-ct}+\frac{Ct^2}{\sqrt{T}}\]
		for some constants $C$ and $c$ that don't depend on $T.$
	\end{lem}
	\begin{proof}
		Let $p_T$ be a family of smooth functions on $(-\infty,2]$, such that 
		\begin{enumerate}[(1)]
			\item $p_T|_{[1,2]}\equiv T/2$,
			\item $p_T|_{(-\infty,1]}(s)=-T\rho(e^{T^2}(1-s))(s-1)^2/2+T/2$ , where $\rho\in C_c^\infty([0,\infty))$, such that $0\leq\rho\leq1$, $\rho|_{[0,1/4]}\equiv0,$ $\rho|_{[1/2,\infty)}\equiv1,$ $|\rho'|\leq \cn_1$ and $|\rho''|\leq \cn_2$ for some universal constant $\cn_1$ and $\cn_2.$
		\end{enumerate}
		\def\tDb{\bar{\Delta}}
		\def\tkb{\bar{k}}
		\def\tlb{\bar{\l}}
		Let $\tDb_{T}$ be the Witten Laplacian associated with $p_T$ with relative boundary condition on $2$, $\tkb_T$ be its heat kernel. Note that $\ker(\tDb_T)$ contains the 1-form $e^{p_T}ds$.
		Let $\tlb_k(T)$ be the $k$-th eigenvalue of $\tDb_T$, and $\l_{B,2,k}$ be the $k$-th eigenvalue of $\Delta_{B,2}$. Then repeating what we did before, we still have $\tlb_k(T)\to \l_{B,2,k}.$
		As a result, $\ker(\Delta_T)$ is one-dimensional and generated by the one-form $e^{p_T}ds$ if $T$ is large. Thus, \be\label{traceone}\Tr_s(e^{-t\tDb_T})=\int_{-\infty}^2\tr_s(\tkb_{T}(t,s,s))ds=-1.
		\ee 
		
		Moreover, one still has Proposition \ref{heat1form} for $\tDb_T$. That is, $\int_{-\infty}^1\tkb_{\tT}(t,s,s)ds\leq \frac{Ct^2}{\sqrt{T}}.$   
		As a result, by (\ref{traceone})
		\[\left|\int_{1}^2\tr_s(\tkb_{\tT})(t,s,s)ds+1\right|\leq \frac{Ct^2}{\sqrt{T}}.\]
		
		Moreover, since $p_T=f_T$ on $[1/2,2]$, it follows from Lemma \ref{speed} and Duhamel's principle (Proceeding as in Lemma \ref{kx0}) that for $s\in[1,2],$
		$|\tkb_{T}(t,s,s)-k_T(t,s,s)|\leq Ce^{-c/t}$. Thus, the lemma follows.
		
	\end{proof}
 Building on the above estimates, we now establish the key result of this subsection.
     \begin{lem}\label{lem77} For $t \in (0, 1]$ and $T$ being sufficiently large,
        \[\Tr_s(N^\R e^{-t\Delta^\R_{t^{-7}T}})=\Tr_s(N^\R e^{-t\Delta_{B,1}^\R})+\Tr_s(N^\R e^{-t\Delta_{B,2}^\R})+O(t^{0.1})\]
        Here, recall that \( O(g(t)) \) denotes terms satisfying \( |O(g(t))| \leq C g(t) \) for some constant \( C \) independent of \( T \), where \( g(t) \) is any positive function and \( t \in (0,1) \).
    \end{lem}
    \begin{proof} The lemma follows from the two equalities below easily.   Recall that $\tT=t^{-7}T.$ 
        \begin{align*}
		&\ \ \ \ \Tr_s(N^\R e^{-t\Delta^\R_{\tT}})=-\int_{-2}^2\tr_s(N^{\R}k_{\tT})(t,s,s)ds\\
		&=-\int_{-2}^{-1}\tr_s(N^{\R}k_{\tT})(t,s,s)ds-\int_{1}^2\tr_s(N^{\R}k_{\tT})(t,s,s)ds+O({t^2})\mbox{ (By Proposition \ref{heat1form})}\\
		&=-\int_{-2}^{-1}k_{\tT,+}(t,s,s)ds-\int_{1}^{2}k_{\tT,+}(t,s,s)ds+O({t^2})\quad\text{(By \Cref{prop71})}\\
        &=\Tr_s(N^\R e^{-t\Delta_{B,1}^\R})-\int_{1}^{2}k_{\tT,+}(t,s,s)ds+O({t^{0.1}})\quad\text{(By Lemma \ref{lacon})}.
        \end{align*}
    \begin{align*}
      &\ \ \ \   -\int_{1}^{2}k_{\tT,+}(t,s,s)ds\\
      &=\int_{1}^{2}\tr_s(k_{\tT})(t,s,s)ds-\int_1^2k_{\tT,-}(t,s,s)ds+O({t^2})\quad\text{(By \Cref{prop71})}\\
		&=-1-\int_1^2\bk_{B,2,-}(t,s,s)ds+O(t^{0.1})\mbox{ (By Lemma \ref{lacon} and \ref{strkt})}\\
		&=-1+\Tr_s((N^\R-1) e^{-t\Delta_{B,2}^\R})+O(t^{0.1})\\
		&=\Tr_s(N^\R e^{-t\Delta_{B,2}^\R})+O(t^{0.1})\mbox{ (By \eqref{tracer1})}.
	\end{align*}
    \end{proof}
	
	\subsection{Appearance of Mysterious Terms}\label{partial72}

	By \Cref{lem77}, one dimensional version of \eqref{point-wise} and the dominated convergence theorem,
	\be\label{lacon1}\lim_{T\to\infty}\int_0^1t^{-1}|\Tr_s(N^\R e^{-t\Delta^\R_{t^{-7}T}})-\Tr_s(N^\R e^{-t(\Delta_{B,1}^\R\oplus\Delta_{B,2}^\R)})|dt=0.\ee
	Let \[\tilde{\zeta}_T(z):=\frac{1}{\Gamma(z)}\int_0^1t^{z-1}(\Tr_s(N^\R e^{-t\Delta^\R_T})-\Tr_s(N^\R e^{-t\Delta^\R_{t^{-7}T}}))dt.\]
	
	By (\ref{ntn12}), (\ref{tracens}), and (\ref{lacon1}), one can see that
	
	\begin{thm}\label{main1}
		As $T\to\infty,$
		\be\label{eq713}(\zeta^{\mS}_{T,\la})'(0)=\sum_{i=1}^2(\zeta_i^{\mS})'(0)+\tilde{\zeta}_T'(0)\chi(Y)\rank(F)+o(1).\ee
		In particular, if $\dim(H^k(M;F))=\dim(H^k_{\abs}(M_1;F_1))+\dim(H^k_{\rel}(M_2;F_2))$ for all $k$, then $\zeta_{T,\la}=\zeta_T$ whenever $T$ is large enough. As a result, by Theorem \ref{larcon} and \eqref{eq713}, if $$\dim(H^k(M;F))=\dim(H^k_{\abs}(M_1;F_1))+\dim(H^k_{\rel}(M_2;F_2))$$ for all $k$, then
		\be\begin{aligned}\label{eq714}&\ \ \ \ \log \T(g^{T\bm},h^{\F},\nabla^{\bar{F}})(T)\\
			&=\log \T_{\abs}(g^{TM_1},h^{F_1},\nabla^{F_1})+\log \T_{\rel}(g^{TM_2},h^{F_2},\nabla^{F_2})+\chi(Y)\rank(F)\tilde{\zeta}_T'(0)/2+o(1).\end{aligned}\ee
	\end{thm}

	\def\tz{\tilde{\zeta}}
	Similarly, let \[\tilde{\zeta}_{T,i}(z):=\frac{1}{\Gamma(z)}\int_0^1t^{z-1}(\Tr_s(N^\R e^{-t\Delta^\R_{T,i}})-\Tr_s(N^\R e^{-t\Delta^\R_{t^{-7}T,i}}))dt.\]

	\begin{thm}\label{main2}
		As $T\to\infty,$
		\begin{align*}&\ \ \ \ \log \mathcal{T}_{i}(g^{T\bm_i},h^{\F_i},\nabla^{\bar{F}_i})(T)=\log \T_{i}(g^{TM_i},h^{F_i},\nabla^{F_i})+\chi(Y)\rank(F)\tilde{\zeta}_{T,i}'(0)/2+o(1).\end{align*}
	\end{thm}

	To prove Theorem \ref{main}, it suffices to show that $\tilde{\zeta}_T'(0)=\log(2)-T+o(1)$ and $\tilde{\zeta}_{T,i}'(0)=-T/2+o(1)$, which will be accomplished in the next section. More precisely, see Proposition \ref{zeta10} and Proposition \ref{zeta11}.

	\section{Ray-Singer Metrics associated with $S^1$ and $[-2,2]$}\label{last}
Proving Theorem \ref{main} completely requires computing the mystery terms \(\tilde{\zeta}_T'(0)\) and \(\tilde{\zeta}_{T,i}'(0),i=1,2\) in \Cref{main1} and \Cref{main2}.  Since \(\tilde{\zeta}_T'(0)\) and \(\tilde{\zeta}_{T,i}'(0)\) for \( i = 1,2 \) are, by definition, independent of \( M, F, \) and \( Y \), we can compute them by studying the Ray-Singer metrics associated with the simple spaces \( S^1 \) and \([-2,2]\).  
    
	\subsection{Ray-Singer metrics associated with $S^1$}
	First, let $p_T\in C^\infty(\R)$ be a smooth function with period 8 as follows
	\begin{enumerate}[(1)]
		\item $p_T(s)=f_T(s),\forall s\in[-2,2]$;
		\item $p_T(s)=f_T(4-s),\forall s\in[2,6]$.
	\end{enumerate}
	
	Let $S^1(8)$ denote the circle with length 8. Then we regard $S^1(8)$ as the interval $[-2,6]$ with $-2$ and $6$ identified, so we could regard $p_T$ as a smooth function on $S^1(8)$.

	Let $d_{T}=d+dp_T\wedge$ be the Witten deformation of de Rham differentials on $S^1(8)$, and $\Delta^{S^1}_{T}:=d_{T}d_{T}^*+d_{T}^*d_{T}$  be its Witten Laplacian. Let $H(S^1(8),d_{T})$ be the cohomology for $d_{T},$
	and $|\cdot|_{RS,T}$ be the $L^2$-metric on $\det(H^*(S^1(8),d_{T}))$  induced by $\Delta^{S^1}_{T}$-harmonic forms. 
	
	Notice that $\tau: H^*(S^1(8),d)\to H^*(S^1(8),d_{T})$, $[w]\mapsto [e^{-p_T}w]$ is an isomorphism.
	
	Let $\T(S^1)(T)$  be the analytic torsion for $\Delta^{S^1}_{T}$, $\|\cdot\|_{\RS,T}=|\cdot|_{\RS,T}\T(S^1)(T)$ be the associated Ray-Singer metric on $\det(H^*(S^1(8),d_{T})).$
	
	\begin{lem}\label{s1}
		$\log \T(S^1)(0)=\log \T(S^1)(T)-2\log(2)+T+o(1)$ as $T\to\infty.$
	\end{lem}
	\begin{proof}
		It follows from \cite[Theorem 4.7]{bismutzhang1992cm} that $\|\cdot\|_{\RS,0}=\tau^*\|\cdot\|_{\RS,T}$. To show $\log \T(S^1)(0)=\log \T(S^1)(T)-2\log(2)+T+o(1)$, it suffices to show $\log |\cdot|_{\RS,0}=\log\tau^*|\cdot|_{\RS,T}+2\log(2)-T-o(1)$. \\
		Let 
		\[\a(T):=\int_{S^1(8)}e^{2p_T}ds,\]
		then $[e^{2p_T}ds]=[\a(T)ds/8]$ in $H^*(S^1(8),d)$. Moreover, $\a(T)ds$ is $\Delta^{S^1}_{0}$-harmonic.
		Let $\rho_T:=[1]\otimes [e^{2p_T}ds]^{-1}\in \det(H^*(S^1,d)),$ then one computes
		\[|\rho_T|^2_{\RS,0}=\frac{{\int_{S^1(8)}1ds}}{\int_{S^1(8)}{(\a(T)/8)^2}ds}=\frac{64}{(\a(T))^2}.\]
		On the other hand, $\tau[1]=[e^{-p_T}],\tau[e^{2p_T}ds]=[e^{p_T}ds]$. Since $e^{-p_T}$ and $e^{p_T}ds$ are both $\Delta^{S^1}_T$-harmonic,
		\[\tau^*|\rho_T|^2_{RS,T}=|\tau(\rho_T)|^2_{RS,T}=\frac{\int_{S^1(8)}e^{-2p_T}ds}{\int_{S^1(8)}e^{2p_T}ds}=1.\]
		Here the last equality follows from the fact that $f_T$ is an odd function on $[-2,2]$, thus $\int_{S^1(8)}e^{-2p_T}ds=\int_{S^1(8)}e^{2p_T}ds.$\\
		Hence, 
		\[\log|\cdot|_{\RS,0}=\log\tau^*|\cdot|_{\RS,T}-\log(\a(T))+3\log 2.\]
		A simple calculation yields $\a(T)=2e^T(1+o(1)). $ The lemma follows.
	\end{proof}
	\ \\
	Let $\Delta^{[0,2]}_1$ be the usual Hodge Laplacian on $\Omega([0,2])$ with absolute boundary conditions,
	$\Delta^{[0,2]}_2$ be the usual Hodge Laplacian on $\Omega([0,2])$ with relative boundary conditions. Let $\T([0,2],i)$ be the analytic torsion with respect to $\Delta^{[0,2]}_i(i=1,2).$
	\\
	By a straightforward computation (recall that our $S^1$ has length $8$),
	\begin{align}\begin{split}\label{tzero}
			\log \T(S^1)(0)&=-3\log(2),\\
			\log \T([0,2],i)&=-\log(2).
	\end{split}\end{align}
	Hence, by Lemma \ref{s1} and (\ref{tzero}), 
	$\log \T(S^1)(T)=-\log(2)-T+o(1)$.\\
	While by \eqref{eq714} and (\ref{tzero})
	\[\log \T(S^1)(T)=-2\log(2)+\tilde{\zeta}'_T(0).\]
	Hence, 
	\begin{prop}\label{zeta10}
		$\tilde{\zeta}_T'(0)=\log(2)-T+o(1).$
	\end{prop}

	\subsection{Ray-Singer metric associated with $[-2,2]$}
	Let $q_T$ be a smooth even function on $[-2,2],$ such that for $s\in[-2,0]$, $q_T(s)=f_T(s-2)$.
	Let $d_{T,2}=d+dq_T\wedge$ be the Witten deformation of de Rham differentials on $[-2,2]$, and $\Delta^{[-2,2]}_{T,2}$ be its Witten Laplacian with relative boundary conditions. Let $H_{\rel}([-2,2],d_{T,2})$ be the cohomology for $d_{T,2},$
	and $|\cdot|_{RS,T,2}$ be the $L^2$-metric on $\det(H_{\rel}^*([-2,2],d_{T,2}))$  induced by $\Delta^{\R}_{T,2}$-harmonic forms. 
	
	Notice that $\tau_2: H_{\rel}^*([-2,2],d)\to H_{\rel}^*([-2,2],d_{T})$, $[w]\mapsto [e^{-q_T}w]$ is an isomorphism.
	
	Let $\mathcal{T}_{2}(T)$  be the analytic torsion for $\Delta^{\R}_{T,2}$, $\|\cdot\|_{\RS,T,2}=|\cdot|_{\RS,T,2}\mathcal{T}_2(T)$ be the associated Ray-Singer metric on $\det(H_{\rel}^*([0,2],d_{T})).$
	
	Similarly, for $d_{T,1}:=d-dq_T$, let $\Delta^\R_{T,2}$ be its Witten Laplacian with absolute boundary conditions. Let $\mathcal{T}_{1}(T)$  be the analytic torsion for $\Delta^{\R}_{T,1}$, $\|\cdot\|_{\RS,T,1}=|\cdot|_{\RS,T,1}\mathcal{T}_1(T)$ be the associated Ray-Singer metric on $\det(H_{\abs}^*([-2,0],d_{T})).$
	
	Notice that $\tau_1: H_{\abs}^*([-2,2],d)\to H_{\abs}^*([-2,2],d_{T})$, $[w]\mapsto [e^{q_T}w]$ is an isomorphism.
	
	By \cite[Theorem 3.4]{bruning2013gluing},
	\begin{lem}\label{expert}
		\[\partial_T \tau_i^*\log \|\cdot\|_{\RS,T,i}=0, i=1,2.\]
	\end{lem}

	\begin{lem}\label{s2}
		$\log \mathcal{T}_{i}(0)=\log \mathcal{T}_{i}(T)+{T}/2-\log(2)/2+o(1), i=1,2$ as $T\to\infty.$
	\end{lem}
	\begin{proof}
		By Lemma \ref{expert}, it suffices to show $\log |\cdot|_{\RS,0,i}=\log\tau_i^*|\cdot|_{\RS,T,i}-T/2+\log(2)/2+o(1)$. 
		
		Let 
		\[\a(T):=\int_{-2}^2e^{2q_T}ds,\]
		then $[e^{2q_T}ds]=[\a(T)ds/4]$ in $H_{\rel}^*([-2,2],d)$. Moreover, $\a(T)ds$ is $\Delta^{[-2,2]}_{2}$-harmonic. Here $\Delta^{[-2,2]}_2$ is the restriction of $-\p_s^2$ on $[-2,2]$ with relative boundary conditions.
		
		Let $\rho_{T,2}:=[e^{2q_T}ds]^{-1}\in \det(H_{\rel}^*([-2,2],d)),$ then one computes
		\[|\rho_{T,2}|^2_{\RS,0,2}=\frac{1}{\int_{-2}^2(\a(T)/4)^2ds}=\frac{4}{(\a(T))^2}.\]
		On the other hand, $\tau_2[e^{2q_T}ds]=[e^{q_T}ds]$. Since $e^{q_T}ds$ is $\Delta^{[-2,2]}_{T,2}$-harmonic,
		\[\tau_2^*|\rho_{T,2}|^2_{RS,T,2}=|\tau_2(\rho_{T,2})|^2_{RS,T,2}=\frac{1}{{\int_{-2}^2(e^{q_T})^2ds}}=\frac{1}{{\a(T)}}.\]
		Hence, 
		\[\log|\cdot|_{\RS,0,2}=\log\tau_2^*|\cdot|_{\RS,T,2}-\log(\a(T))/2+\log(2).\]
		A simple calculation yields $\a(T)=2e^T(1+o(1)). $ 
		Hence, 
		\[\log|\cdot|_{\RS,0,2}=\log\tau_2^*|\cdot|_{\RS,T,2}-T/2+\log(2)/2+o(1).\]
		Similarly,
		\[\log|\cdot|_{\RS,0,1}=\log\tau^*|\cdot|_{\RS,T,1}-T/2+\log(2)/2+o(1).\]

	\end{proof}
	\ \\
	Let $\Delta^{[-2,2]}_{1}$ be the usual Hodge Laplacian on $\Omega([-2,2])$ with absolute boundary conditions,
	$\Delta^{[-2,2]}_{2}$ be the usual Hodge Laplacian on $\Omega([-2,2])$ with relative boundary conditions. Let $\mathcal{T}([-2,2],i)$ be the analytic torsion with respect to $\Delta^{[-2,2]}_{i}(i=1,2).$
	
	Similarly, let $\Delta^{[-1,1]}_{1}$ be the usual Hodge Laplacian on $\Omega([-1,1])$ with absolute boundary conditions,
	$\Delta^{[-1,1]}_{2}$ be the usual Hodge Laplacian on $\Omega([-1,1])$ with relative boundary conditions. Let $\mathcal{T}([-1,1],i)$ be the analytic torsion with respect to $\Delta^{[-1,1]}_{i}(i=1,2).$
	\\
	By a straightforward computation 
	\begin{align}\begin{split}\label{tzero1}
			\log \mathcal{T}([-2,2],i)&=-3\log(2)/2,\\
			\log \mathcal{T}([-1,1],i)&=-\log(2).
	\end{split}\end{align}
	Hence, by Lemma \ref{s2} and (\ref{tzero1}), 
	$\log \mathcal{T}_{i}(T)=-3\log(2)/2-T/2+\log(2)/2+o(1)$.\\
	While by Theorem \ref{main1} and (\ref{tzero1})
	\[\log \mathcal{T}_{i}(T)=-\log(2)+\tilde{\zeta}'_{i,T}(0)+o(1).\]
	Hence, 
	\begin{prop}\label{zeta11}
		$\tilde{\zeta}_{T,i}'(0)=-T/2+o(1).$
	\end{prop}
	
	\section{Proof of Theorem \ref{int6p}}\label{lastreal}
	
	\def\te{\tilde{e}}
	\def\tir{\tilde{r}}
	\def\tp{\tilde{\p}}
 To prove \Cref{int6p}, we use a finite-dimensional analogue of the "gluing formula" for torsions, specifically \cite[Proposition 2.6]{lesch2013gluing} (or \cite[Theorem 3.1 and Theorem 3.2]{milnor1966whitehead}). 
 
 \cite[Proposition 2.6]{lesch2013gluing} establishes a relation between the analytic torsion of a (finite) double complex, in which the vertical sequences are short exact, and the analytic torsion of its associated long exact sequence. Here, we apply this result to a double complex associated with \( \tO_{\sm}(\bm,\F)(T) \), where \( \tO_{\sm}(\bm,\F)(T) \) denotes the space spanned by eigenforms of \( \tD_T \) corresponding to eigenvalues in \( [0, \delta] \). More precisely, for sufficiently large \( T \), we first obtain the following sequence of maps:
\be\label{short}  
0 \to \H^k(\bm_2,\F_2)(T) \stackrel{\te_{k,T}}\rightarrow \tO_{\sm}^k(\bm,\F)(T) \stackrel{\tir_{k,T}}\rightarrow \H^k(\bm_1,\F_1)(T) \to 0.  
\ee  
Here, the maps are defined as  
\[
\te_{k,T}(u) := \tP^{\delta}(T) \cE(u), \quad \forall u \in \H^k(\bm_2,\F_2)(T);
\]
\[
\tir_{k,T}(u) := \tP_{T,1}(u|_{\bm_1}), \quad \forall u \in \tO_{\sm}^k(\bm,\F)(T),
\]
where \(\tP_{T,i}\) denotes the orthogonal projection onto \(\ker(\tilde{\Delta}_{T,i})=\H(\bm_i,\F_i)(T)\) for \( i = 1,2 \), and since \( \tilde\Delta_T = e^{f_T} \Delta_T e^{-f_T} \) (see also \cref{witwei} for more details), $\tP^\delta(T):=e^{f_T}\P^{\delta}(T)e^{-f_T}$ is the orthogonal projection associated with $\tO_{\sm}(\bm,\F)(T)$. 	For any $L^2$-form $w$ on $M_i$ (or $\bm_i$), let $\cE(w)\in L^2(\bm)$ be an extension of $w$, such that outside $M_i$ (or $\bm_i$), $\cE(w)=0.$

We will show that \eqref{short} is exact for sufficiently large \( T \) (\Cref{exathm}). Actually, we can prove that \(\tilde{e}_{k,T}\) and \(\tilde{r}_{k,T}^*\) are almost isometric (\Cref{last8p}).

Next, we consider the following chain complexes of finite-dimensional vector spaces:  
\begin{align}\label{com1}  
&0 \to \H^0(\bm_i,\F_i) \stackrel{0}\rightarrow \H^1(\bm_i,\F_i) \stackrel{0}\rightarrow \cdots \stackrel{0}\rightarrow \H^{\dim M}(\bm_i,\F_i) \to 0, \\  
\label{com2}  
&0 \to \tO_{\sm}^0(\bm,\F) \stackrel{d^{\F}}\rightarrow \tO_{\sm}^1(\bm,\F) \stackrel{d^{\F}}\rightarrow \cdots \stackrel{d^{\F}}\rightarrow \tO_{\sm}^{\dim M}(\bm,\F) \to 0.  
\end{align}  

We then show that \( d^{\F} \circ \te_{k,T} = 0 \) and \( \tir_{k,T} \circ d^{\F} = 0 \) (\Cref{exapp}), allowing us to associate a double complex with \eqref{short}–\eqref{com2}. Applying \cite[Proposition 2.6]{lesch2013gluing} to this double complex, along with the almost-isometry of \(\tilde{e}_{k,T}\) and \(\tilde{r}_{k,T}^*\) (\Cref{last8p}), completes the proof of \Cref{int6p}.
	
	To establish the almost-isometry of \(\tilde{e}_{k,T}\) and \(\tilde{r}_{k,T}^*\), we need to prove \Cref{last1p} and \Cref{last2p}.
	\def\Rr{{\mathcal{R}}}
	Let $\eta \in C^{\infty}([0,1])$, such that $\eta|_{[0,1 / 4]} \equiv 0,\left.\eta\right|_{[1 / 2,1]} \equiv 1$.  Recall that $Q_T: \Omega_{\mathrm{abs}}\left(M_1 ; F_1\right) \oplus \Omega_{\mathrm{rel}}\left(M_2 ; F_2\right) \rightarrow W^{1,2}\Omega(\bar{M}, \bar{F})$ is
	$$
	Q_T(u)(x):=\left\{\begin{array}{l}
		u_i(x), \text { if } x \in M_i, \\
		\eta(-s) u_1(-1, y) e^{-f_T(s)-T / 2}, \text { if } x=(s, y) \in[-1,0] \times Y, \\
		\eta(s) u_2(1, y) e^{f_T(s)-T / 2}, \text { if } x=(s, y) \in[0,1] \times Y ,
	\end{array}\right.
	$$
	for $u=(u_1,u_2)\in\Omega_{\mathrm{abs}}\left(M_1 ; F_1\right) \oplus \Omega_{\mathrm{rel}}\left(M_2 ; F_2\right)$. Set $\Q_T:=e^{f_T}Q_T.$ 
	
For any $L^2$-form $w$ on $M_i$ (or $\bm_i$), let $\cE(w)\in L^2(\bm)$ be an extension of $w$, such that outside $M_i$ (or $\bm_i$), $\cE(w)=0.$

	It can be observed that
	\begin{prop}\label{last1p} For $u \in \operatorname{ker}\left(\Delta_1\right) \oplus \operatorname{ker}\left(\Delta_2\right)$,
		$$
		\begin{array}{c}
			\left\|Q_T u-\cE(u)\right\|_{L^2}^2 \leq\frac{C\|u\|^2_{L^2}}{\sqrt{T}}, \\
			\left\|\mathcal{P}^\delta(T) Q_T(u)-\cE(u)\right\|_{L^2}^2 \leq \frac{C\|u\|_{L^2}^2}{\sqrt{T}}.
		\end{array}
		$$
		for some constant $C$ that is independent of $T$. 
		
		As a result,
		$$
		\begin{array}{c}
			\left\|\Q_T u-\cE(e^{f_T}u)\right\|_{L^2,T}^2 \leq\frac{C\|e^{f_T}u\|^2_{L^2,T}}{\sqrt{T}}, \\
			\left\|\tP^\delta(T) \Q_T(u)-\cE(e^{f_T}u)\right\|_{L^2,T}^2 \leq \frac{C\|e^{f_T}u\|_{L^2,T}^2}{\sqrt{T}}.
		\end{array}
		$$
		Recall that $\|\cdot\|_{L^2,T}$ is the norm induced by $g^{T\bm}$ and $h^{\F}_T:=e^{-2f_T}h^{\F}.$
		
		As a result, when $T$ is large enough, $\tP^\delta(T) \Q_T(u)$ spans $\tO_{s m}(\bar{M}, \bar{F})(T)$ for $u \in \operatorname{ker}\left(\Delta_1\right) \oplus \operatorname{ker}\left(\Delta_2\right)$.
	\end{prop}
	\begin{proof} For $u \in \operatorname{ker}\left(\Delta_1\right) \oplus \operatorname{ker}\left(\Delta_2\right)$, set $v_T=Q_T(u)-\mathcal{P}^\delta(T) Q_T(u)=(1-\P^\delta(T))Q_T(u)$.
		First, by the trace formula, G\aa rding's inequality and the boundary conditions,
		\begin{align}\begin{split}\label{6}
				&\ \ \ \ \int_Y\left|u\left((-1)^i, y\right)\right|^2 +\left|D^Y u\left((-1)^i,y\right)\right|^2\operatorname{dvol}_Y \leq C  \int_{M_i}\left|u\right|^2 \mathrm{dvol}.
		\end{split}\end{align}
		By (\ref{6}) and a straightforward computation, one can see that (e.g. \eqref{ineqii})
		\be\label{qteu}
		\left\|Q_T u-\cE(u)\right\|_{L^2}^2 \leq\frac{C}{\sqrt{T}}\|u\|^2_{L^2}
		\ee
		and
		\be\label{8}
		\left\|D_T Q_T u\right\|_{L^2}^2 \leq \frac{C}{\sqrt{T}}\|u\|^2_{L^2} .
		\ee
		
		Moreover,
		\be\label{9}
		\delta\left\|v_T\right\|_{L^2}^2 \leq\left\|D_T v_T\right\|_{L^2}^2 \leq\left\|D_T Q_T u\right\|_{L^2}^2 .
		\ee
		For the last inequality of \eqref{9}, just note that $D_T$ commutes with $\P^\delta(T)$, so $D_Tv_T=(1-\P^{\delta}(T))D_TQ_Tu.$

		(\ref{8}) and (\ref{9}) then imply that
		$$
		\left\|v_T\right\|_{L^2}^2 \leq \frac{C\|u\|^2_{L^2}}{\delta \sqrt{T}}.
		$$
		That is,
		\be\label{qteu1}
		\left\|\mathcal{P}^\delta(T) Q_T(u)-Q_T(u)\right\|_{L^2}^2 \leq \frac{C\|u\|_{L^2}^2}{\delta \sqrt{T}} .
		\ee
		The proposition follows from \eqref{qteu} and \eqref{qteu1}.
	\end{proof}
	
	Notice that if $u\in\Omega_{\bd}(M_i,F_i)$, then $\Q_Tu\in\Omega_{\bd}(\bm_i,\F_i)_T$. By Hodge theory and Theorem \ref{eigencon}, when $T$ is big enough, all eigenvalues of $\tD_{T,i}$ inside $[0,\delta]$ must be $0.$ Let $\tP_{T,i}$ be the orthogonal projection onto $\ker(\tD_{T,i})$. Similarly, one has
	\begin{prop} \label{last2p}For $u \in \operatorname{ker}\left(\Delta_i\right)$,
		$$
		\begin{array}{c}
			\left\|\Q_T u-\cE(e^{f_T}u)\right\|_{L^2,T}^2 \leq\frac{C\|e^{f_T}u\|^2_{L^2,T}}{\sqrt{T}}, \\
			\left\|\tP_{T,i} \Q_T(u)-\cE(e^{f_T}u)\right\|_{L^2,T}^2 \leq \frac{C\|e^{f_T}u\|_{L^2,T}^2}{\sqrt{T}}.
		\end{array}
		$$
		As a result, when $T$ is large enough, $\tP^\delta(T) \Q_T(u)$ spans $\H(\bar{M}_i, \bar{F}_i)(T)$ for $u \in \operatorname{ker}\left(\Delta_i\right)$.
	\end{prop}

	\begin{prop}\label{last8p}
		$\te_{k,T}$ and $\tir^*_{k,T}$ are almost isometric embeddings as $T\to\infty$. That is, for example, for any $u\in \H^k(\bm_2;\F_2)(T)$, $\lim_{T\to\infty}\frac{\|\te_{k,T}u\|_{L^2,T}}{\|u\|_{L^2,T}}=1.$ Here $\tir^*_{k,T}$ is the adjoint of $\tir_{k,T}.$
	\end{prop}
	\begin{proof}\ \\
		$\bullet${\textit{$\te_{k,T}$ is almost isometric:}}\\
		For any $u\in \H^k(\bm_2,\F_2)(T)$, there exists $u_T\in \ker(\Delta_2)\cap\Omega^k_{\rel}(M_2,F_2)$ such that
		$u=\tP_{T,i} \Q_T(u_T)$, then by Proposition \ref{last2p}, \be\label{lasteq1}\|u\|^2_{L^2,T}= (1+O(T^{-\frac{1}{2}}))\|e^{f_T}u_T\|^2_{L^2,T}.\ee
		While by Proposition \ref{last2p}, Proposition \ref{last1p} and the fact that $\|\tP^\delta(T)\|=1$,
		\begin{align}\label{lasteq2}  \begin{split}
				&\ \ \ \ \|\te_{k,T}u\|^2_{L^2,T}=\|\tP^\delta(T) \cE u\|^2_{L^2,T}\\
				&= \|\tP^\delta(T) \Q_Tu_T\|^2_{L^2,T}+O(T^{-\frac{1}{2}})\|e^{f_T}u_T\|^2_{L^2,T}=\|e^{f_T}u_T\|^2_{L^2,T}(1+O(T^{-\frac{1}{2}})).
			\end{split}
		\end{align}
		It follows from (\ref{lasteq1}) and (\ref{lasteq2}) that $\lim_{T\to\infty}\frac{\|\te_{k,T}u\|^2_{L^2,T}}{\|u\|^2_{L^2,T}}=1.$ 
		\\ \ \\
		$\bullet${\textit{$\tir^*_{k,T}$ is almost isometric:}}\\
		For $u\in \H^k(\bm_1,\F_1)(T),$ we first show that \be\label{rktvu}\tir^*_{k,T}u=\tP^{\delta}(T)\cE(u).\ee
		Notice that for any $v\in \tO_{\sm}(\bm,\F)(T)$, 
		\begin{align*}
			(\tir_{k,T}v,u)_{L^2(\bm_2),T}=(v,\cE(u))_{L^2(\bm),T}=(v,\tP^{\delta}(T)\cE(u))_{L^2(\bm),T}.
		\end{align*}
		Following the same steps as above, one derives that $\tir^*_{k,T}$ is almost isometric.
	\end{proof}

	\begin{thm}\label{exathm}
		With maps $\te_{k,T}$ and $\tir_{k,T}$ given above, the sequence (\ref{short}) is exact.
	\end{thm}
	\begin{proof}
		Let $\lan\cdot,\cdot\ran_T$ be the metric on $\Lambda(T^*\bm)\otimes \F$ induced by $g^{T\bm}$ and $h_T^{\F}.$\\ \ \\
		$\bullet$ In follows from Proposition \ref{last8p} that $\te_{k,T}$ and $\tir^*_{k,T}$ are injective when $T$ is large.\\ \ \\
		$\bullet$ {$\cE(\ker(\tilde{\Delta}_{T,1}))\subset \ker(d^{\F,*}_T)$ and $\cE(\ker(\tilde{\Delta}_{T,2}))\subset\ker(d^{\F})$:}\\ \ \\
		Let $u\in\ker(\tD_{T,2})$. First, since $u$ satisfies relative boundary conditions, integration by parts shows that for any $\beta\in \Omega(\bm,\F)$, $\int_{\bm}\lan \cE(u),d^{\F,*}_T\beta\ran_T\dvol=0$. Thus, by Theorem \ref{hodec}, $\cE(u)\in\ker(d^{\F})$.
		Similarly, $\cE(\ker(\tilde{\Delta}_{T,1}))\subset \ker(d^{\F,*}_T)$.
		\\ \ \\
		$\bullet$ \textit{$\im\  \te_{k,T}=\ker \tir_{k,T}$:}\\
		For the dimension reason, it suffices to show that $\im \te_{k,T}\subset\ker \tir_{k,T}$. That is, it suffices to show $\im \te_{k,T}\perp \im \tir^*_{k,T}$. First, for any $u_i\in\ker(\tilde\Delta_{T,i})$, $i=1,2$, it's clear that \be\label{eq200}(\cE(u_1),\cE(u_2))_{L^2,T}=0.\ee
		
		Since $\cE(u_1)\in \ker(d^{\F,*}_T),\cE(u_2)\in\ker(d^{\F})$, one can see that
		$(1-\tP^{\delta}(T))\cE u_1\in\im\ d^{\F,*}_T, (1-\tP^{\delta}(T))\cE u_2\in\im\ d^{\F}$, which means
		\be\label{eq201}((1-\tP^{\delta}(T))\cE(u_1),(1-\tP^{\delta}(T))\cE(u_2))_{L^2,T}=0.\ee
		By (\ref{eq200}) and (\ref{eq201}), the definition of $\te_{k,T}$ and \eqref{rktvu}, $(\te_{k,T}u_2,\tir^*_{k,T}u_1)_{L^2,T}=0.$
		
	\end{proof}

	Integration by parts as in the proof of Theorem \ref{exathm}, one can show easily that
	\begin{prop}\label{exapp}
		$d^{\F}\circ\te_{k,T}=0,\tir_{k,T}\circ d^{\F}=0.$
	\end{prop}
	
	Hence, by Theorem \ref{exathm} and Proposition \ref{exapp}, we get the following long exact sequence again
	\be\label{mvpp}
	\mathcal{MV}(T):\cdots \stackrel{\p_{k-1,T}}\rightarrow \H^k\left(\bm_2;\F_2\right)(T)  \stackrel{e_{k,T}}{\rightarrow}\H^k\left(\bm;\F\right)(T)  \stackrel{r_{k,T}}{\rightarrow} \H^k\left(\bm_1;\F_1\right)(T)  \stackrel{\p_{k,T}}{\rightarrow} \cdots 
	\ee
	with metric induced by Hodge theory.

	\begin{proof}[Proof of Theorem \ref{int6p}]
		Since (\ref{com1}) and (\ref{com2}) are complexes of finite dimensional vector spaces, it follows from Proposition \ref{last8p} and  \cite[Proposition 2.6]{lesch2013gluing} (or \cite[Theorem 3.1 and Theorem 3.2]{milnor1966whitehead}) that
		\[\lim_{T\to\infty}\log\T_{\sm}(g^{T\bm},h^{\F},\nabla^{\F})(T)-\log\T(T)=0.\]
	\end{proof}
\ \\
{\textbf{Acknowledgment}   
		The author is grateful to Professor Xianzhe Dai and Professor Guofang Wei for their consistently stimulating conversations and encouragement. Appreciation is also extended to Martin Puchol and Yeping Zhang for their insightful discussions. This work is partially supported by the China Postdoctoral Science Foundation (Grant No. 2023M730092).}

	\bibliography{lib}
	
	\bibliographystyle{plain}
	
\end{document}